\documentclass[11pt]{article}
\usepackage{amsmath}
\usepackage{amsfonts}
\usepackage{amssymb}
\usepackage{mathrsfs}
\usepackage{amsthm}
\usepackage{palatino}
\usepackage[margin=2.5cm, vmargin={1.5cm}]{geometry}

\usepackage{dcolumn}

\usepackage{times}
\usepackage{graphicx}
\usepackage{xcolor}

\usepackage{pstricks}
\usepackage{pst-plot}
\usepackage{pst-grad}

\usepackage{wrapfig}

\usepackage{bm}

\renewcommand\footnotemark{}

\begin{document}

\title{Topographs for binary quadratic forms and class numbers}

\author{
  Cormac ~O'Sullivan\footnote{{\it Date:} June 4, 2025.
\newline \indent \ \ \
\newline \indent \ \ \
  \newline \indent \ \ \
Support for this project was provided by a PSC-CUNY Award, jointly funded by The Professional Staff Congress and The City
\newline \indent \ \ \
University of New York.}
  }

\date{}

\maketitle

%\maketitle

%modular symbol
\def\s#1#2{\langle \,#1 , #2 \,\rangle}

%domains
\def\H{{\mathbb{H}}}
\def\F{{\mathfrak F}}
\def\Fd{{\mathcal F}}
\def\Gd{{\mathcal G}}
\def\C{{\mathbb C}}
\def\R{{\mathbb R}}
\def\Z{{\mathbb Z}}
\def\Q{{\mathbb Q}}
\def\N{{\mathbb N}}
%symbols
\def\G{{\Gamma}}
\def\GH{{\G \backslash \H}}
\def\g{{\gamma}}
\def\L{{\Lambda}}
\def\ee{{\varepsilon}}
\def\K{{\mathcal K}}
\def\Re{\mathrm{Re}}
\def\Im{\mathrm{Im}}
\def\PSL{\mathrm{PSL}}
\def\SL{\mathrm{SL}}
\def\GL{\mathrm{GL}}
\def\Vol{\operatorname{Vol}}
\def\lqs{\leqslant}
\def\gqs{\geqslant}
\def\sgn{\operatorname{sgn}}
\def\res{\operatornamewithlimits{Res}}
\def\li{\operatorname{Li_2}}
\def\lip{\operatorname{Li}'_2}
\def\pl{\operatorname{Li}}

\def\ei{\mathrm{Ei}}

\def\clp{\operatorname{Cl}'_2}
\def\clpp{\operatorname{Cl}''_2}
\def\farey{\mathscr F}

\def\dm{{\mathcal A}}
\def\ov{{\overline{p}}}
\def\ja{{K}}

\def\nb{{\mathcal B}}
\def\cc{{\mathcal C}}
\def\nd{{\mathcal D}}

\def\u{{\text{\rm u}}}
\def\v{{\text{\rm v}}}
\def\ib{{\text{\,\rm i}}}
\def\jb{{\text{\,\rm j}}}
\def\qq{{f}}%{{\mathfrak F}}

\def\mg{{\mathcal M_{G}}}
\def\mz{{\mathcal M_{Z}}}
\def\stab{{\text{\rm Stab}}}
\def\aut{{\text{\rm Aut}}}
\def\ze{{z}}

\newcommand{\stira}[2]{{\genfrac{[}{]}{0pt}{}{#1}{#2}}}
\newcommand{\stirb}[2]{{\genfrac{\{}{\}}{0pt}{}{#1}{#2}}}
\newcommand{\eu}[2]{{\left\langle\!\! \genfrac{\langle}{\rangle}{0pt}{}{#1}{#2}\!\!\right\rangle}}
\newcommand{\eud}[2]{{\big\langle\! \genfrac{\langle}{\rangle}{0pt}{}{#1}{#2}\!\big\rangle}}
\newcommand{\norm}[1]{\left\lVert #1 \right\rVert}
\newcommand{\dx}[1]{\overset{*}{#1}}

\newcommand{\e}{\eqref}
\newcommand{\la}{\label}
\newcommand{\bo}[1]{O\left( #1 \right)}
\newcommand{\ol}[1]{\,\overline{\!{#1}}} % overline short italic

%theorems

\newtheorem{theorem}{Theorem}[section]
\newtheorem{lemma}[theorem]{Lemma}
\newtheorem{prop}[theorem]{Proposition}
\newtheorem{conj}[theorem]{Conjecture}
\newtheorem{cor}[theorem]{Corollary}
\newtheorem{assume}[theorem]{Assumptions}
\newtheorem{adef}[theorem]{Definition}

\newtheorem*{algo}{Reduction algorithm}
\newtheorem*{algo2}{Continued fraction algorithm}
\newtheorem*{algo3}{General continued fraction algorithm}
\newtheorem*{algo4}{Algorithm to find $\varepsilon_D$ and $\varepsilon^*_D$}

\numberwithin{figure}{section}
\numberwithin{table}{section}

%\newcounter{coundef}
%\newtheorem{adef}[coundef]{Definition}

\newcounter{counrem}
\newtheorem{remark}[counrem]{Remark}

\renewcommand{\labelenumi}{(\roman{enumi})}
\newcommand{\spr}[2]{\sideset{}{_{#2}^{-1}}{\textstyle \prod}({#1})}
\newcommand{\spn}[2]{\sideset{}{_{#2}}{\textstyle \prod}({#1})}

\numberwithin{equation}{section}

% this fixes spacing around \left, \right
\let\originalleft\left
\let\originalright\right
\renewcommand{\left}{\mathopen{}\mathclose\bgroup\originalleft}
\renewcommand{\right}{\aftergroup\egroup\originalright}

\bibliographystyle{alpha}

\begin{abstract}
In this work we study, in greater detail than before, J.H. Conway's topographs for integral binary quadratic forms.
These are trees in the plane with regions labeled by integers following a simple pattern. Each topograph can display the values of a single form, or represent an equivalence class of forms. We give a new treatment of reduction of forms to canonical equivalence class representatives by employing topographs and  a novel continued fraction for complex numbers. This allows uniform reduction for any positive, negative, square or non-square discriminant.
Topograph geometry also provides new class number formulas, and short proofs of results of Gauss relating to sums of three squares. Generalizations of the series of Hurwitz for class numbers give evaluations of certain infinite series, summed over the regions or edges of a topograph.
\end{abstract}

\section{Introduction}
A binary quadratic form is a  polynomial
$
  q(x,y)=a x^2 +b xy + cy^2.
$
 For fixed integers $a, b, c$ and $m$ it has historically been a challenging problem to describe the solutions of the Diophantine equation $q(x,y)=m$. Conway proposed an elegant answer in \cite{con97}   by presenting $q$ in a graphical way he called a  topograph. As reviewed in Section \ref{cla}, these topographs have  simple structures described in terms of descending to their rivers, lakes and wells. An example is shown in Figure \ref{ribbon}.
%*************************************
% two lakes
\SpecialCoor
\psset{griddots=5,subgriddiv=0,gridlabels=0pt}
\psset{xunit=0.2cm, yunit=0.2cm, runit=0.2cm}
%\psset{xunit=1cm, yunit=1cm, runit=1cm}
\psset{linewidth=1pt}
\psset{dotsize=5pt 0,dotstyle=*}
\begin{figure}[ht]
\centering
\begin{pspicture}(3,4)(54,33) %\psgrid

%\psline(0,0)(10,0)(10,5)(0,5)(0,0)
%\psset{arrowscale=2,arrowinset=0.5}
\psset{arrowscale=1.1,arrowinset=0.1}
\newrgbcolor{light}{0.9 0.9 1.0}
\newrgbcolor{light}{0.8 0.7 1.0}
%\newrgbcolor{light}{0 0 1}
%\newrgbcolor{blue2}{0.4 0.4 1.0}
%\newrgbcolor{blue2}{0.1 0.2 0.8}
\newrgbcolor{blue2}{0.3 0.3 0.8}

%left lake
\psline[linecolor=light,fillstyle=solid,fillcolor=light](3,7.3)(8.4,9.4)(11,14)(7.7,17.5)(3,18.8)

\psline(3,7.3)(8.4,9.4)(11,14)(7.7,17.5)(3,18.8)

%right lake
\psline[linecolor=light,fillstyle=solid,fillcolor=light](54,23)(49,22.3)(44,20.5)(40,17.2)(43.2,13)(48.5,10.3)(54,9.3)

\psline(54,23)(49,22.3)(44,20.5)(40,17.2)(43.2,13)(48.5,10.3)(54,9.3)

%top trees

\psline(10,21.3)(7.7,17.5)

\psline(13.8,26.8)(18.1,24.5)(20.5,20.5)(20.7,16.3)
\psline(18,29.2)(18.1,24.5)
\psline(22,29.5)(23,25)(20.5,20.5)
\psline(27.7,28.6)(23,25)

\psline(28.7,22)(30,17.4)

\psline(34.2,27.5)(36.2,22.9)(34.5,18.5)
\psline(39.5,27.9)(36.2,22.9)

\psline(42.5,24.8)(44,20.5)

\psline(50.8,31.3)(48.3,27)(49,22.3)
\psline(45.8,30.7)(48.3,27)

%bottom trees
\psline(8.4,9.4)(11.2,6.3)

\psline(16,14)(17.7,9.6)(16.5,4.6)
\psline(21.2,5)(17.7,9.6)

\psline(25.4,14.7)(27.5,10)(25,4.5)
\psline(31.7,5.6)(27.5,10)

\psline(43.2,13)(41,8.4)

\psline(48.5,10.3)(47.3,5.5)

%river
\psline[linewidth=3.5pt,linecolor=blue2](10.8,14)(16,14)(20.7,16.3)(25.4,14.7)(30,17.4)(34.5,18.5)(40.2,17.2)

%\psline[linewidth=2.4pt]{->}(11,14)(13.5,14)
%\psline[linewidth=2.4pt]{->}(16,14)(20.7,16.3)

\rput(6,13.5){$0$}

\rput(6,21){$25$}
\rput(14.5,19.2){$7$}
\rput(13,10.5){$-11$}
\rput(7,6){$-29$}

\rput(15.3,28.7){$_{85}$}
\rput(20.3,26.2){$40$}
\rput(25,29.8){$_{91}$}

\rput(25,19.3){$9$}
\rput(21,12){$-8$}
\rput(18.7,4.5){$_{-45}$}

\rput(32.5,21.7){$16$}
\rput(34,12.6){$-5$}
\rput(28,5.3){$_{-35}$}

\rput(37,28.1){$63$}
\rput(39.5,21.2){$13$}
\rput(46.1,24){$31$}
\rput(52.1,25.8){$49$}
\rput(48.4,31.5){$_{160}$}

\rput(44.5,9){$-23$}
\rput(51,6){$-41$}

\rput(48,17){$0$}

%\psline[ArrowInside=->, ArrowInsidePos=0.5](0,0)(9,9)

\end{pspicture}
\caption{Part of a topograph of discriminant $D=18^2$}
\label{ribbon}
\end{figure}
%*************************************
The underlying graph is a tree where all vertices have degree $3$. Starting with three adjacent numbers there is a simple rule to add more numbers, given in Definition \ref{topdef}. The connection with quadratic forms appears in Theorem \ref{abc}.

Another fundamental question involves counting quadratic forms.
Define
\begin{equation}\label{qm}
q|M := q(\alpha x +\beta y, \g x+\delta y) \qquad \text{for} \qquad M = \begin{pmatrix} \alpha &\beta \\ \g &\delta \end{pmatrix},
\end{equation}
giving a right action by matrices.
The forms $q_1$ and $q_2$  are equivalent, written $q_1 \sim q_2$, if
$q_1 = q_2|M$ for some $M \in \SL(2,\Z)$, and the much studied, and in some ways still mysterious, class numbers count the  equivalence classes  of each discriminant $D=b^2-4ac$.  We will use $h^*(D)$  to indicate the number of classes of  forms with integer coefficients and discriminant $D$. The number of these that are primitive, meaning that the $\gcd$ of their coefficients $a, b, c$ is $1$, is denoted by $h(D)$. Here $D$ can be any integer $\equiv 0$ or $1 \bmod 4$ and, by convention, only form classes representing nonnegative integers (i.e. having $a, c \gqs 0$) are counted if $D\lqs 0$.
%, (i.e. $a, c\gqs 0$).

The numbers in the topograph example in Figure \ref{ribbon} give the values taken by some quadratic form $q(x,y)$ as $x$ and $y$ vary over coprime integers. All forms equivalent to $q$ take the same values and in fact, as Rickards describes in \cite{ri21}, all these equivalent forms are themselves naturally represented on this topograph as its edges. Exploiting this geometric connection between topographs and quadratic forms will lead us to new  class number formulas  as well as simple new proofs of known ones.

As a motivating example, consider the following remarkable class number formula of Duke, Imamo\=glu and T\'oth.  This  result  in fact  inspired this entire project.

\begin{theorem} \cite[Thm. 3]{dit21} \la{t3}
For $D>0$ a fundamental discriminant,
\begin{equation} \la{dz}
  h(D) \log \varepsilon_D = D^{1/2} \sum_{\substack{ [a, \, b, \, c] \ \text{\rm Zagier reduced} \\ b^2-4ac  \, = \,  D}} \frac 1b +D^{3/2} 
  \sum_{\substack{ a, \, a+b+c, \, c \, > \, 0 \\ b^2-4ac  \, = \,  D}} \frac 1{3(2a+b)b(b+2c)}.
\end{equation}
\end{theorem}
On the right of \e{dz} are sums over integral quadratic forms of discriminant $D$. 
The first sum  is finite, indexed by Zagier reduced  forms that we will describe. The second sum is infinite and, from our perspective,  $(2a+b)b(b+2c)$ in the summand has an interesting shape:  it is the product of the  natural edge labels of three edges attached to a vertex in a topograph. So we may ask if it is possible to reinterpret \e{dz} as a sum over the vertices of the $h(D)$ topographs of discriminant $D$.

The authors in \cite{dit21} were extending similar formulas of Hurwitz in \cite{hur} who looked at the $D<0$ case. Our ultimate goal, achieved in Section \ref{sqrd}, is to further extend these results to the more difficult situation when $D$ is square by making use of the topographic framework.

The required theory is built up in stages and we treat a variety of topics:

\vskip 3mm
{\bf Topograph basics.} Topogaphs are introduced from a simplified perspective in Section \ref{ttop}, avoiding quadratic forms or any underlying structure initially, and focusing on their local rule. The connection to continued fractions and $\SL(2,\Z)$ is seen here.

\vskip 3mm
{\bf Topograph classification.} Conway's original classification is recounted in Section \ref{cla} with his appealing notions of rivers, lakes and wells. Topogaphs have different properties depending on whether the discriminant $D$ is positive or negative, square or non-square.

\vskip 3mm
{\bf Automorphs and units.} The proof of Theorem \ref{t3} in  \cite{dit21} needs  automorph groups. These are the elements $M$ of $\SL(2,\Z)$ that fix a particular form: $q|M=q$. On the topograph such elements correspond to paths linking areas that look the same locally. In Section \ref{ive} this novel point of view is used to prove some standard results, giving the generator of the automorph group of a form of discriminant $D$ and, when $D>0$, the related generating unit $\varepsilon_D$ seen on the left of \e{dz}.

\vskip 3mm
{\bf Class numbers for negative discriminants.}
As  seen in Section \ref{<}, counting topographs by using their internal properties gives simply derived (finite) class number formulas that complement the  well-known formulas of Dirichlet involving values of $L$-functions and Kronecker symbols.
We also  give a new short proof of a celebrated result of Gauss:
for $n>3$,
\begin{equation} \la{rr3}
  r'_3(n)=   \begin{cases}
                         12 h(-4n), & \mbox{if $n \equiv 1, 2 \bmod 4$}  \\
                         24 h(-n), & \mbox{if $n \equiv 3 \bmod 8$}\\
                            0, & \mbox{if $n \equiv 0, 4, 7 \bmod 8$}.
                       \end{cases}
\end{equation}
Here $r'_3(n)$ is the number of ways to write $n$ as a sum of squares of $3$ integers with their $\gcd$ being $1$.  Our proof of \e{rr3}  uses little more that the striking identity
\begin{equation}\label{krm}
  \left(\sum_{n \in \Z}(-q)^{n^2} \right)^3 = 1-6\sum_{e,f>0}(-1)^{e+f} q^{e f}-4 \sum_{e,f,g >0}(-1)^{e+f+g} q^{e f+f g+g e}
\end{equation}
of Krammer \cite{kr93} and some basic topograph geometry.

\vskip 3mm
{\bf Class numbers for positive discriminants.}
We highlight this example from Section \ref{rty}: for non-square $D>0$, 
\begin{equation}\label{ex}
  \varepsilon_D^{h(D)} = \prod_{\substack{ a+b+c \,< \, 0 \,  <  \, a,  \, c \\ b^2-4ac  \, = \,  D, \, \gcd(a,b,c)=1}} \frac{-b+\sqrt{D}}{2a},
\end{equation}
 where the product is finite. This seems to be a new formula, though it is based on an exercise from the book \cite{z81} involving Zagier's method of reducing forms.

\vskip 3mm
{\bf Continued fractions for complex numbers.}
In Section \ref{co} we find a natural generalization of the  continued fraction algorithm that applies to all complex numbers. It is related to the lattice reduction techniques of Lagrange and Gauss and has interesting geometric properties.  This algorithm appears to be original and should be of independent interest, differing from the continued fractions based on Gaussian integers going back to Hurwitz, and featured in \cite{dn14} for example.

\vskip 3mm
{\bf Reduction of forms.}
Conway mentions in \cite[p. 25]{con97} that topographs can be used to effectively decide when two given forms are equivalent. We see methods for this  in Section \ref{rdf} where a form $q$ is reduced to equivalent canonical forms based on the continued fractions of the
solutions to $q(x,1)=0$, given by the first and second roots
\begin{equation} \la{zeeq}
  \ze_q := \frac{-b+\sqrt{D}}{2a}, \qquad \ze'_q :=\frac{-b-\sqrt{D}}{2a} \qquad (a\neq 0).
\end{equation}
This makes an attractive topographical picture since, as described in \cite{sv18}, the numbers in the continued fraction expansion indicate how many forward left turns and forward right turns to alternately make in order to move from $q$ to the reduced form. We make this more precise and extend it to the $D<0$ case with our complex continued fractions. This allows a uniform treatment of reduction for all discriminants  $D$. Section \ref{red} also explains the   topograph geometry of Gauss's original method of reduction and the Zagier reduction used on the right of \e{dz}.

\vskip 3mm
{\bf The river as a binary necklace or word.} In the case of positive discriminants, topographs have rivers. The thickened path between the two lakes is the river in Figure \ref{ribbon}. Moving from left to right along it involves the sequence $LRLRR$ of left and right turns, making a binary word. Rivers can also be periodic, making  binary necklaces. Sections \ref{neck} and \ref{word} describe what kinds of words and necklaces are possible and show how their symmetries reflect the symmetries of the topograph as well as illuminating further number theoretic properties.

\vskip 3mm
Topographs give an interesting alternative viewpoint for the classical subjects of binary quadratic forms and quadratic fields.
It is certainly true that many technical computations and proofs  become much simpler with their aid. As Conway declared in \cite[p. vii]{con97}, ``just look!'' %Their geometry also deserves further study.
Topographs have also found recent %interesting 
applications in geometry and knot theory: in  \cite{ri21} they are used to count intersection numbers of closed modular geodesics, and Feh\'er employs them in \cite{feh} to find new examples of non-isotopic Seifert
surfaces.

\vskip 3mm
{\bf Acknowledgements.} I am grateful to William Duke, James Rickards and the referee for their helpful comments and suggestions.

\section{Continued fractions and the modular group} \la{co}
For any numbers $a_0,a_1, \dots, a_r$, denote the continued fraction built with them by
\begin{equation*}
  \langle a_0,a_1, \dots, a_r \rangle := a_0+\frac{1}{a_1 + \displaystyle \frac{1}{\ddots + \displaystyle \frac 1{a_r}}}.
\end{equation*}
\begin{algo2}
{\rm For any $x \in  \R$, start with index $i=0$.
\begin{itemize}
  \item Let $m=\lfloor x\rfloor$ and $a_i=m$. If $x-m=0$ then  finish. Otherwise   replace $x$ by $1/(x-m)$.
\end{itemize}
Increment $i$ and repeat.  The output is $(a_0,a_1, \dots )$.
}
\end{algo2}

This algorithm gives the standard continued fraction expansion $x= \langle a_0,a_1, \dots  \rangle$, finite for rational $x$ and otherwise infinite, with $a_0 \in \Z$ and $a_i \in \Z_{\gqs 1}$ for $i \gqs 1$.
Put
\begin{equation} \la{stu}
  S := \begin{pmatrix} 0 & -1\\ 1 & 0 \end{pmatrix}, \quad  T :=\begin{pmatrix} 1 & 1 \\ 0 & 1 \end{pmatrix}, \quad  U := TS= \begin{pmatrix} 1&-1\\1&0 \end{pmatrix}.
\end{equation}
The group $\SL(2,\Z)$ is generated by any two of $S$, $T$ and  $U$. Also let $\G:=\PSL(2,\Z)$. These groups act on the upper half plane $\H$ by linear fractional transformations and a
 standard fundamental domain is
\begin{equation}\label{fd}
\Fd:= \big\{ z =x+iy \in \H \, : \, -1/2\lqs x < 1/2, |z|\gqs 1 \text{ and } |z|=1 \implies x \lqs 0\big\}.
\end{equation}
 % ,  including the boundary on the left side. %Notice that $\F$ includes its mirror image in $\overline{\H}$.
%Write $\F':=\F \cup S\F \cup \{0\}$.
We next generalize the  continued fraction algorithm so that it can be applied to complex numbers. This is motivated by reduction of quadratic forms of negative discriminant, and will be exactly what we need in Section \ref{ne}. The algorithm seems to be new, though it is related to lattice reduction techniques going back to Lagrange and Gauss \cite{vv07}.
Set
$$\Fd':= \Fd \cup S\Fd \cup \{0\}\cup -(\Fd \cup S\Fd),$$
so that $\Fd'$ is mapped to itself under both $z \mapsto -z$ and $z \mapsto 1/z$ (for $z\neq 0$).

%*************************************
% combined cont frac fund domain
\SpecialCoor
\psset{griddots=5,subgriddiv=0,gridlabels=0pt}
\psset{xunit=1.5cm, yunit=1.5cm, runit=1.5cm}
%\psset{xunit=1cm, yunit=1cm, runit=1cm}
\psset{linewidth=1pt}
\psset{dotsize=7pt 0,dotstyle=*}
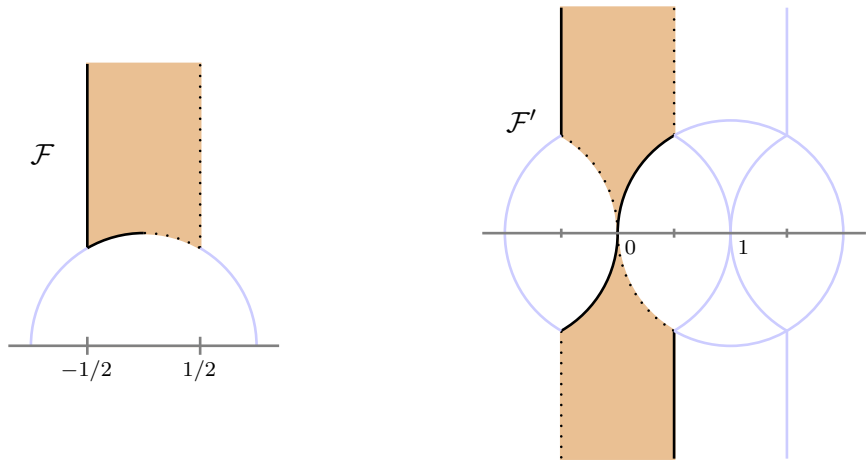
\begin{figure}[ht]
\centering
\begin{pspicture}(0.5,0)(9,4.2) %\psgrid

%\psline(0,0)(10,0)(10,5)(0,5)(0,0)
%\psset{arrowscale=2,arrowinset=0.5}
\psset{arrowscale=1.4,arrowinset=0.3,arrowlength=1.1}
\newrgbcolor{light}{0.8 0.8 1.0}
%\newrgbcolor{light}{0 0 1}
%\newrgbcolor{blue2}{0.4 0.4 1.0}
\newrgbcolor{pale}{1 0.7 1}
\newrgbcolor{pale}{1 0.7 0.4}
\newrgbcolor{pale}{0.9179 0.7539 0.5781}
%\newrgbcolor{pale}{0.957031 0.726563 0.726563}
%\newrgbcolor{pale}{0.87 0.57 0.77} %wine

\rput(6.7,2){%
        \begin{pspicture}(-1.2,-2)(2.2,2)

\pspolygon[linecolor=pale,fillstyle=solid,fillcolor=pale](-0.5,2)(-0.5,-2)(0.5,-2)(0.5,2)(-0.5,2)

\pswedge[linecolor=white,fillstyle=solid,fillcolor=white](-1,0){1}{-90}{90}
\pswedge[linecolor=white,fillstyle=solid,fillcolor=white](1,0){1}{90}{270}

\psarc[linecolor=pale](-1,0){1}{0}{60}
\psarc[linecolor=pale](1,0){1}{180}{240}

\psarc[linestyle=dotted](-1,0){1}{0}{60}
\psarc(1,0){1}{120}{180}
\psarc(-1,0){1}{-60}{0}
\psarc[linestyle=dotted](1,0){1}{180}{240}

\psline(-0.5,2)(-0.5,0.866025)
\psline[linestyle=dotted](0.5,2)(0.5,0.866025)
\psline[linestyle=dotted](-0.5,-2)(-0.5,-0.866025)
\psline(0.5,-2)(0.5,-0.866025)

\psarc[linecolor=light](0,0){1}{-60}{60}
\psarc[linecolor=light](0,0){1}{120}{240}

\psarc[linecolor=light](2,0){1}{120}{240}

\psarc[linecolor=light](1,0){1}{-120}{120}

\psline[linecolor=light](1.5,2)(1.5,0.866025)
\psline[linecolor=light](1.5,-2)(1.5,-0.866025)

\psline[linecolor=gray](-1.2,0)(2.2,0)
\psline[linecolor=gray](-0.5,0.04)(-0.5,-0.04)
\psline[linecolor=gray](0.5,0.04)(0.5,-0.04)
\psline[linecolor=gray](1.5,0.04)(1.5,-0.04)

%\rput(-0.5,-0.22){$_{-1/2}$}
%\rput(0.5,-0.22){$_{1/2}$}

\rput(-0.85,1){$\mathcal F'$}
\rput(0.12,-0.12){$_0$}
\rput(1.12,-0.12){$_1$}
\end{pspicture}}

\rput(2,2){%
        \begin{pspicture}(-1.5,-0.5)(1.5,2.5)
\pspolygon[linecolor=pale,fillstyle=solid,fillcolor=pale](-0.5,2.5)(-0.5,0)(0.5,0)(0.5,2.5)(-0.5,2.5)

\pswedge[linecolor=white,fillstyle=solid,fillcolor=white](0,0){1}{0}{180}

\psarc[linecolor=pale](0,0){1}{60}{90}
\psarc[linestyle=dotted](0,0){1}{60}{90}
\psarc(0,0){1}{90}{120}

\psline(-0.5,2.5)(-0.5,0.866025)
\psline[linestyle=dotted](0.5,2.5)(0.5,0.866025)

\psarc[linecolor=light](0,0){1}{0}{60}
\psarc[linecolor=light](0,0){1}{120}{180}

\psline[linecolor=gray](-1.2,0)(1.2,0)
\psline[linecolor=gray](-0.5,0.08)(-0.5,-0.08)
\psline[linecolor=gray](0.5,0.08)(0.5,-0.08)

\rput(-0.5,-0.22){$_{-1/2}$}
\rput(0.5,-0.22){$_{1/2}$}

\rput(-0.9,1.7){$\mathcal F$}
\end{pspicture}}

\end{pspicture}
\caption{The regions $\mathcal F$ and $\mathcal F'$}
\label{funds}
\end{figure}
%*************************************

\begin{algo3}
{\rm For any $z \in  \C$, start with index $i=0$.
\begin{itemize}
  \item Let $m=\lfloor \Re(z)\rfloor$. If $z-m-\delta \in \Fd'$ for $\delta=0$ or $1$ then let $a_{i}=m+\delta$, $z_0=z-m-\delta$ and finish. Otherwise  put $a_{i}=m$ and replace $z$ by $1/(z-m)$.
\end{itemize}
Increment $i$ and repeat. The output is $(a_0,a_1, \dots )$ if $z\in \R$ initially, or $( a_0,a_1, \dots, a_r, z_0)$ if $z\notin \R$.
}
\end{algo3}

There is an ambiguity at the final step of this procedure in exactly two cases. When $z-m = e^{\pm \pi i/3}$ then $\delta=0$ and $\delta=1$ are both valid. We may choose $\delta=0$ in these cases to give a definite result. This could also be achieved by removing the corners $e^{\pm 2\pi i/3}$ from $\Fd'$.
%The parameter $\delta$ above is included to ensure that the integers produced are positive after the first. This is seen in the next main result.

\begin{theorem} \label{genc}
For $z\in \R$ the general continued fraction algorithm produces the usual continued fraction expansion. When $z\notin \R$ it always terminates and we have
\begin{equation*}
  z =  \langle a_0,a_1, \dots, a_{r-1}, a_r+ z_0 \rangle,
\end{equation*}
with $a_0 \in \Z$, $a_i  \in \Z_{\gqs 1}$ for $i \gqs 1$,  and $z_0 \in \Fd'-\{0\}$.
\end{theorem}
\begin{proof}
A comparison with the usual continued fraction algorithm shows they match when $z\in \R$.
%In what follows it will be convenient to replace the single step the in generalized continued fraction algorithm with the two steps:
%\begin{enumerate}
%  \item Let $m=\lfloor \Re(z)\rfloor$. If $z-m-\delta \in \F'$ for $\delta=0$ or $1$ then put $a_{2i}=m+\delta$, $z_0=z-a_{2i}$ and finish. Otherwise  put $a_{2i}=m$ and replace $z$ by $-1/(z-m)$.
%  \item Next let $n=\lceil \Re(z)\rceil$. If $z-n+\delta \in \F'$ for $\delta=0$ or $1$ then put $a_{2i+1}=-n+\delta$, $z_0=z+a_{2i+1}$ and finish. Otherwise  put $a_{2i+1}=-n$ and replace $z$ by $-1/(z-n)$.
%\end{enumerate}
%These are easily seen to be equivalent, with $z$ now staying on one side of $\R$.
To understand the steps when $z\notin \R$ it is useful to set up the next geometric objects.
For relatively prime integers $p$ and $q$ with $q\gqs 0$, define the {\em fan} $F_{p/q}$ as follows:
$$
F_{1/0}=F_{-1/0}:=\bigcup_{n \in \Z} T^n \Fd, \qquad F_{p/q}:= \begin{pmatrix} p & * \\ q & * \end{pmatrix} F_{1/0},
$$
where $(\begin{smallmatrix}p & * \\ q & *\end{smallmatrix})$ is completed to an element of $\SL(2,\Z)$ and $F_{p/q}$ is well-defined. These fans tile $\H$ with all points in exactly one fan, except that points in the orbit of $i$ are in two and those in the orbit of $e^{\pi i/3}$ are in three. Each fan is inscribed by a Ford circle; see \cite[pp. 28, 29]{con97}. Define the alternative fundamental domain
$$
\Gd:= \big\{ z =x+iy \in \H \, : \, -1/2 < x \lqs 1/2, |z|\gqs 1 \text{ and } |z|=1 \implies x \gqs 0\big\},
$$
with boundary on the right. This gives the similar fans
$$
G_{1/0}=G_{-1/0}:=\bigcup_{n \in \Z} T^n \Gd, \qquad G_{p/q}:= \begin{pmatrix} p & * \\ q & * \end{pmatrix} G_{1/0},
$$
where $F_{p/q}$ and $G_{p/q}$ only differ by having complementary boundaries.
It is easy to check that
\begin{equation*}
  z \in F_{p/q} \iff 1/z  \in \overline G_{q/p}.
\end{equation*}

%*************************************
% fan
\SpecialCoor
\psset{griddots=5,subgriddiv=0,gridlabels=0pt}
\psset{xunit=5cm, yunit=5cm, runit=5cm}
%\psset{xunit=0.8cm, yunit=0.8cm, runit=0.8cm}
\psset{linewidth=1pt}
\psset{dotsize=5pt 0,dotstyle=*}
\begin{figure}[ht]
\centering
\begin{pspicture}(-1.1,-0.1)(1.1,1) %\psgrid

%\psset{arrowscale=2,arrowinset=0.5}
\psset{arrowscale=1.1,arrowinset=0.1}
%\newrgbcolor{light}{0.9 0.9 1.0}
\newrgbcolor{light}{0.8 0.7 1.0}
%\newrgbcolor{light}{0 0 1}
%\newrgbcolor{blue2}{0.4 0.4 1.0}
%\newrgbcolor{blue2}{0.1 0.2 0.8}
\newrgbcolor{bmid}{0.6 0.5 1}
\newrgbcolor{bb}{0.4 0.3 1}
\newrgbcolor{cx}{0.8 0.8 1.0}
\newrgbcolor{white2}{0.4 0.3 1}

%side parts

\psarc[linecolor=cx](-1,0){1}{60}{90}
\psarc[linecolor=cx](1,0){1}{90}{120}

\psarc[linecolor=cx](-0.3333,0){0.333}{120}{158.213}
\psarc[linecolor=cx](0.3333,0){0.333}{21.7868}{60}

\psarc[linecolor=cx](0.625,0){0.125}{32.2042}{81.7868}  %n=3 top
\psarc[linecolor=cx](-0.625,0){0.125}{98.2132}{147.796}

\psarcn[linecolor=cx](0.733333,0){0.0666667}{92.2042}{38.2132}  %n=4 top
\psarc[linecolor=cx](-0.733333,0){0.0666667}{87.7958}{141.787}

%\psdot(0.357143, 0.123718)
%\psdot(0.394737, 0.0455803)

\psarc[linecolor=cx](0.2,0){0.2}{13}{38.2132}
\psarc[linecolor=cx](-0.2,0){0.2}{141.787}{167}
\psarc[linecolor=cx](-0.8,0){0.2}{13}{38.2132}
\psarc[linecolor=cx](0.8,0){0.2}{141.787}{167}

%main

\psline[linecolor=gray](-1.1,0)(1.1,0)
\psline[linecolor=gray](-0.5,0.03)(-0.5,-0.03)
\psline[linecolor=gray](0.5,0.03)(0.5,-0.03)

\rput(-0.5,-0.08){$_{k-1/2}$}
\rput(0.5,-0.08){$_{k+1/2}$}

\psset{linewidth=0.1pt}

%n=0
\pscustom[linecolor=white2,fillstyle=solid,fillcolor=bb]{
\psarc(0,0){1}{60}{120}
\psarcn(-1,0){1}{60}{0}
\psarcn(1,0){1}{180}{120}
}

%n=1
\pscustom[linecolor=white2,fillstyle=solid,fillcolor=light]{
\psline(-0.5,0.2887)(-0.5,0.866)
\psarcn(-1,0){1}{60}{0}
\psarc(-0.333,0){0.333}{0}{120}
}

%n=-1
\pscustom[linecolor=white2,fillstyle=solid,fillcolor=light]{
\psline(0.5,0.2887)(0.5,0.866)
\psarc(1,0){1}{120}{180}
\psarcn(0.333,0){0.333}{180}{60}
}

%n=2
\pscustom[linecolor=white2,fillstyle=solid,fillcolor=bb]{
\psarc(-0.333,0){0.333}{0}{120}
\psarcn(-0.6667,0){0.333}{60}{21.7868}
\psarcn(-0.2,0){0.2}{141.787}{0}
}

%n=-2
\pscustom[linecolor=white2,fillstyle=solid,fillcolor=bb]{
\psarcn(0.333,0){0.333}{180}{60}
\psarc(0.6667,0){0.333}{120}{158.213}
\psarc(0.2,0){0.2}{38.2132}{180}
}

%n=3
\pscustom[linecolor=white2,fillstyle=solid,fillcolor=light]{
\psarc(-0.2,0){0.2}{0}{141.787}
\psarcn(-0.375,0){0.125}{81.7868}{32.2042}
\psarcn(-0.142857,0){0.142857}{152.204}{0}
}
%{-0.375, 0., 0.125, 32.2042, 261.787}, {0.375, 0., 0.125, 147.796, -81.7868}
%{-0.142857, 0., 0.142857, 0., 152.204}, {0.142857, 0., 0.142857, 180., 27.7958}

%n=-3
\pscustom[linecolor=white2,fillstyle=solid,fillcolor=light]{
\psarcn(0.2,0){0.2}{180}{38.2132}
\psarc(0.375,0){0.125}{98.2132}{147.796}
\psarc(0.142857,0){0.142857}{27.7958}{180}
}

%n=4
\pscustom[linecolor=white2,fillstyle=solid,fillcolor=bb]{
\psarc(-0.142857,0){0.142857}{0}{152.204}
\psarcn(-0.266667,0){0.0666667}{92.2042}{38.2132}
\psarcn(-0.111111,0){0.111111}{158.213}{0}
}
%{-0.266667, 0., 0.0666667, 38.2132, 92.2042}, {0.266667, 0., 0.0666667, 141.787, 87.7958}
%{-0.111111, 0., 0.111111, 0., 158.213}, {0.111111, 0., 0.111111, 180., 21.7868}

%n=-4
\pscustom[linecolor=white2,fillstyle=solid,fillcolor=bb]{
\psarcn(0.142857,0){0.142857}{180}{27.7958}
\psarc(0.266667,0){0.0666667}{87.7958}{141.787}
\psarc(0.111111,0){0.111111}{21.7868}{180}
}

%n=5
\pscustom[linecolor=white2,fillstyle=solid,fillcolor=light]{
\psarc(-0.111111,0){0.111111}{0}{158.213}
\psarcn(-0.208333,0){0.0416667}{98.2132}{42.1034}
\psarcn(-0.0909091,0){0.0909091}{162.103}{0}
}
%{-0.208333, 0., 0.0416667, 42.1034, 98.2132}, {0.208333, 0., 0.0416667, 137.897, 81.7868}
%{-0.0909091, 0., 0.0909091, 0., 162.103}, {0.0909091, 0., 0.0909091, 180., 17.8966}

%n=-5
\pscustom[linecolor=white2,fillstyle=solid,fillcolor=light]{
\psarcn(0.111111,0){0.111111}{180}{21.7868}
\psarc(0.208333,0){0.0416667}{81.7868}{137.897}
\psarc(0.0909091,0){0.0909091}{17.8966}{180}
}

%n=6
\pscustom[linecolor=white2,fillstyle=solid,fillcolor=bb]{
\psarc(-0.0909091,0){0.0909091}{0}{162.103}
\psarcn(-0.171429,0){0.0285714}{102.103}{44.8218}
\psarcn(-0.0769231,0){0.0769231}{164.822}{0}
}
%{-0.171429, 0., 0.0285714, 44.8218, 102.103}, {0.171429, 0., 0.0285714, 135.178, 77.8966}
%{-0.0769231, 0., 0.0769231, 0., 164.822}, {0.0769231, 0., 0.0769231, 180., 15.1782}

%n=-6
\pscustom[linecolor=white2,fillstyle=solid,fillcolor=bb]{
\psarcn(0.0909091,0){0.0909091}{180}{17.8966}
\psarc(0.171429,0){0.0285714}{77.8966}{135.178}
\psarc(0.0769231,0){0.0769231}{15.1782}{180}
}

%n=7
\pscustom[linecolor=white2,fillstyle=solid,fillcolor=light]{
\psarc(-0.0769231,0){0.0769231}{0}{164.822}
\psarcn(-0.145833,0){0.0208333}{104.822}{46.8264}
\psarcn(-0.0666667,0){0.0666667}{166.826}{0}
}
%{-0.145833, 0., 0.0208333, 46.8264, 104.822}, {0.145833, 0., 0.0208333, 133.174, 75.1782}
%{-0.0666667, 0., 0.0666667, 0., 166.826}, {0.0666667, 0., 0.0666667, 180., 13.1736}

%n=-7
\pscustom[linecolor=white2,fillstyle=solid,fillcolor=light]{
\psarcn(0.0769231,0){0.0769231}{180}{15.1782}
\psarc(0.145833,0){0.0208333}{75.1782}{133.174}
\psarc(0.0666667,0){0.0666667}{13.1736}{180}
}

%n=8
\pscustom[linecolor=white2,fillstyle=solid,fillcolor=bb]{
\psarc(-0.0666667,0){0.0666667}{0}{166.826}
\psarcn(-0.126984,0){0.015873}{106.826}{48.3649}
\psarcn(-0.0588235,0){0.0588235}{168.365}{0}
}
%{-0.126984, 0., 0.015873, 48.3649, 106.826}, {0.126984, 0., 0.015873, 131.635, 73.1736}
%{-0.0588235, 0., 0.0588235, 0., 168.365}, {0.0588235, 0., 0.0588235, 180., 11.6351}

%n=-8
\pscustom[linecolor=white2,fillstyle=solid,fillcolor=bb]{
\psarcn(0.0666667,0){0.0666667}{180}{13.1736}
\psarc(0.126984,0){0.015873}{73.1736}{131.635}
\psarc(0.0588235,0){0.0588235}{11.6351}{180}
}

%n=9
\pscustom[linecolor=white2,fillstyle=solid,fillcolor=light]{
\psarc(-0.0588235,0){0.0588235}{0}{168.365}
\psarcn(-0.1125,0){0.0125}{108.365}{49.5826}
\psarcn(-0.0526316,0){0.0526316}{169.583}{0}
}
%{-0.1125, 0., 0.0125, 49.5826, 108.365}, {0.1125, 0., 0.0125, 130.417, 71.6351}
%{-0.0526316, 0., 0.0526316, 0., 169.583}, {0.0526316, 0., 0.0526316, 180., 10.4174}

%n=-9
\pscustom[linecolor=white2,fillstyle=solid,fillcolor=light]{
\psarcn(0.0588235,0){0.0588235}{180}{11.6351}
\psarc(0.1125,0){0.0125}{71.6351}{130.417}
\psarc(0.0526316,0){0.0526316}{10.4174}{180}
}

\pscustom[linecolor=bmid,fillstyle=solid,fillcolor=bmid]{
\psarcn(-0.0526316,0){0.0526316}{169.583}{0}
\psarcn(0,0.5){0.5}{270}{257}
}

\pscustom[linecolor=bmid,fillstyle=solid,fillcolor=bmid]{
\psarc(0.0526316,0){0.0526316}{10.4174}{180}
\psarc(0,0.5){0.5}{270}{283}
}

\psset{linewidth=0.5pt}

\psarc(0,0){1}{60}{90} %n=0 top
\psarc[linecolor=white](0,0){1}{90}{120}
\psarc[linestyle=dashed,](0,0){1}{90}{120} %dotsep=1.5pt

\psline[linecolor=white](-0.5,0.2887)(-0.5,0.57735)
\psline[linestyle=dashed](-0.5,0.2887)(-0.5,0.57735) %n=1 top
\psline(-0.5,0.57735)(-0.5,0.866)
\psline(0.5,0.2887)(0.5,0.57735)
\psline[linecolor=white](0.5,0.57735)(0.5,0.866)
\psline[linestyle=dashed](0.5,0.57735)(0.5,0.866)

\psarc[linecolor=white](-0.6667,0){0.333}{21.7868}{40.8934}
\psarc[linestyle=dashed,dash=2pt 2pt](-0.6667,0){0.333}{21.7868}{40.8934}  %n=2 top
\psarc(-0.6667,0){0.333}{40.8934}{60}
\psarc[linecolor=white](0.6667,0){0.333}{120}{139.106}
\psarc[linestyle=dashed,dash=2pt 2pt](0.6667,0){0.333}{120}{139.106}
\psarc(0.6667,0){0.333}{139.106}{158.213}

\psarc(-0.375,0){0.125}{32.2042}{81.7868}  %n=3 top
\psarc(0.375,0){0.125}{98.2132}{147.796}

\psarcn(-0.266667,0){0.0666667}{92.2042}{38.2132}  %n=4 top
\psarc(0.266667,0){0.0666667}{87.7958}{141.787}

\psarcn(-0.208333,0){0.0416667}{98.2132}{42.1034}  %n=5 top
\psarc(0.208333,0){0.0416667}{81.7868}{137.897}

\psarcn(-0.171429,0){0.0285714}{102.103}{44.8218}  %n=6 top
\psarc(0.171429,0){0.0285714}{77.8966}{135.178}

\psarcn(-0.145833,0){0.0208333}{104.822}{46.8264}  %n=7 top
\psarc(0.145833,0){0.0208333}{75.1782}{133.174}

\psarcn(-0.126984,0){0.015873}{106.826}{48.3649}   %n=8 top
\psarc(0.126984,0){0.015873}{73.1736}{131.635}

\psarcn(-0.1125,0){0.0125}{108.365}{49.5826}   %n=9 top
\psarc(0.1125,0){0.0125}{71.6351}{130.417}

\psarc(0,0.5){0.5}{257}{283}

\rput(0,0.75){$n=0$}
\rput(-0.3,0.48){$_{n=1}$}
\rput(0.3,0.48){$_{n=-1}$}
\rput(-0.3,0.25){$_{2}$}
\rput(0.3,0.25){$_{-2}$}

\end{pspicture}
\caption{The fan $F_{k/1}$}
\label{fanf}
\end{figure}
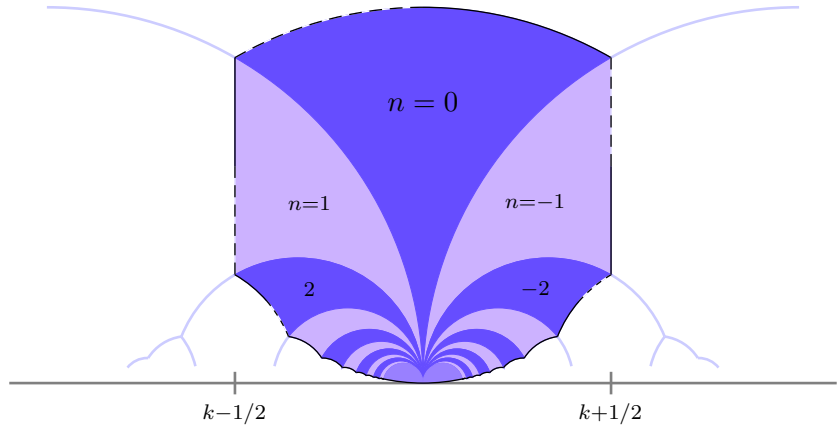
%*************************************

Suppose the algorithm is applied to an initial $z$ in $F_{p/q}$ with $q\gqs 2$. Then $m=\lfloor \Re(z)\rfloor$ also equals $\lfloor p/q\rfloor$ so that $z-m \in F_{p'/q}$ for $0<p'<q$. We have $z'=1/(z-m) \in \overline G_{q/p'}$ for $q/p'>1$. Similarly, with initial $z$ in $\overline G_{p/q}$ we obtain $z' \in F_{q/p'}$. Thus $z=m+1/z'$ and while the points remain in
$F_{a/b}$, $\overline G_{a/b}$ for $b\gqs 2$, the algorithm cannot terminate since $\Fd'-\{0\} \subseteq F_{1/0}\cup F_{0/1} \cup \overline G_{0/1} \cup  \overline G_{1/0}$. Following the usual continued fraction of $p/q$, we eventually obtain
\begin{equation*}
  z= \langle a_0,a_1, \dots, a_{v}, z' \rangle,
\end{equation*}
for some $z'$ in $F_{k/1}$ or $\overline G_{k/1}$ with $k\gqs 2$. The  fan $F_{k/1}$ breaks up, see Figure \ref{fanf}, as
\begin{equation*}
  F_{k/1} = T^k \bigcup_{n \in \Z} S T^n \Fd
\end{equation*}
and if $z'$ is in the part $T^k  S T^n \Fd$ then the algorithm terminates after a few more steps, producing
\begin{subequations} \la{kc} 
\begin{alignat}{2}
  z' & = \langle k-1,1,n-1+z_0 \rangle \quad &&\text{for \quad $n\gqs 2$}, \\
  z' & = \langle k-1,1+z_0 \rangle \quad &&\text{for  \quad $n=1$}, \\
  z' & = \langle k+z_0 \rangle \quad &&\text{for  \quad $n=0$}, \\
  z' & = \langle k,|n|+z_0 \rangle \quad &&\text{for  \quad $n\lqs -1$},
\end{alignat}
\end{subequations}
with $z_0 \in \Fd'-\{0\}$. Also $z'$ in  $T^k  S T^n \overline{\Gd} \subseteq \overline G_{k/1}$ gives the same result. (Slight adjustments to \e{kc} are required at the corners of the fans due to our ambiguous case convention.) The theorem follows.
\end{proof}

It may also be seen from the proof of Theorem \ref{genc} that when $z \in F_{p/q}$ or $\overline G_{p/q}$ then the general algorithm output $a_0, a_1, \dots, a_r$ has up to two more numbers than the continued fraction terms of $p/q$. Excluding these extra  numbers, these sequences agree except possibly differing by $1$ in the last.

For the application we have in mind, put
\begin{equation} \la{lru}
  L := T=\begin{pmatrix} 1 & 1 \\ 0 & 1 \end{pmatrix}, \quad R := TST = -ST^{-1}S =\begin{pmatrix} 1 & 0 \\ 1 & 1 \end{pmatrix}.
\end{equation}
In the next section, the matrices $L$ and $R$ will allow us to move left and right on the topograph.

\begin{cor} \la{lr}
For every $z \in \H$ there exist  $z_1 \in \Fd \cup S \Fd$ and integers $a_i$  so that
\begin{equation*}
  z= L^{a_0} R^{a_1} \cdots L^{a_{2n}} R^{a_{2n+1}}z_1
\end{equation*}
with $a_0 \in \Z$ and $a_i  \in \Z_{\gqs 1}$ for $i \gqs 1$, (and with $a_{2n+1}$  possibly $0$). These integers are produced by the general continued fraction algorithm applied to $z$.
\end{cor}
\begin{proof}
By Theorem \ref{genc},
$
  z =  \langle a_0,a_1, \dots, a_{r-1}, a_r+ z_0 \rangle,
$
for $z_0 \in \Fd'-\{0\}$.
Let $J :=(\begin{smallmatrix}0 & 1\\ 1 & 0\end{smallmatrix})$ and write
\begin{equation} \label{wt}
  z=T^{a_0}J T^{a_1} J \cdots J T^{a_r} z_0
\end{equation}
If $r$ is even then $z_0 \in \Fd \cup S \Fd$ and set $z_1$ to be $z_0$. Use
\begin{equation*}
  J T^m J = S T^{-m} S = R^m
\end{equation*}
in $\PSL(2,\Z)$ to obtain
\begin{equation*}
  z=L^{a_0}R^{a_1}  \cdots L^{a_r} z_1.
\end{equation*}

If $r$ is odd then $z_0$ in \e{wt} is in $\overline \H$. Set $z_1$ to be $1/z_0 \in \Fd \cup S \Fd$ and
\begin{equation*}
   z=T^{a_0}J T^{a_1} J \cdots J T^{a_r} J z_1 \implies z=L^{a_0}R^{a_1}  \cdots R^{a_r} z_1,
\end{equation*}
as desired.
\end{proof}

\begin{cor}
For every $z \in \H$ there exist  $z_1 \in \Fd \cup S \Fd$ and integers $a_i$ so that
\begin{equation*}
   z_1= L^{a_{2n+1}} R^{a_{2n}} \cdots L^{a_{1}} R^{a_{0}}z
\end{equation*}
with $a_0 \in \Z$ and  $a_i  \in \Z_{\gqs 1}$ for $i \gqs 1$, (and with $a_{2n+1}$ possibly $0$).
\end{cor}

\section{Topographs} \la{ttop}
\subsection{Setup}
Let $\mathcal T$ be a tree drawn in the plane, where all vertices have degree 3. This naturally breaks up the plane into regions. In the following,  labels are chosen from $\R$, making real topographs. From Section \ref{cla} on we will only consider integer topographs.

%*************************************
% explanation figures
\SpecialCoor
\psset{griddots=5,subgriddiv=0,gridlabels=0pt}
\psset{xunit=0.36cm, yunit=0.36cm, runit=0.2cm}
%\psset{xunit=1cm, yunit=1cm, runit=1cm}
\psset{linewidth=1pt}
\psset{dotsize=7pt 0,dotstyle=*}
\begin{figure}[ht]
\centering
\begin{pspicture}(0,0)(42,9) %\psgrid

%\psline(0,0)(10,0)(10,5)(0,5)(0,0)
%\psset{arrowscale=2,arrowinset=0.5}
\psset{arrowscale=1.6,arrowinset=0.3,arrowlength=1.1}
\newrgbcolor{light}{0.9 0.9 1.0}
%\newrgbcolor{light}{0 0 1}
%\newrgbcolor{blue2}{0.4 0.4 1.0}
\newrgbcolor{blue2}{0.3 0.3 0.8}

% AP
\rput(7,5){%
        \begin{pspicture}(1,1)(13,9)
          \psline(3,8)(5,5)
\psline(3,2)(5,5)
\psline[ArrowInside=->,ArrowInsidePos=0.6](5,5)(9,5)

\psline(9,5)(11,8)
\psline(9,5)(11,2)

\rput(2,5){\Large $r$}
\rput(7,7.2){\Large $s$}
\rput(7,2.8){\Large $t$}
\rput(12,5){\Large $u$}

\rput(7,0.5){(a)}
        \end{pspicture}}

% both
\rput(21,5){%
        \begin{pspicture}(1,1)(10,9)
          \psline[ArrowInside=->,ArrowInsidePos=0.6](5,5)(3,8)
\psline[ArrowInside=->,ArrowInsidePos=0.6](5,5)(3,2)
\psline[ArrowInside=->,ArrowInsidePos=0.6](5,5)(9,5)

%\psline(9,5)(11,8)
%\psline(9,5)(11,2)

\rput(7,4.2){\textcolor{red}{$e$}}
\rput(3.5,4.1){\textcolor{red}{$f$}}
\rput(4.5,7.1){\textcolor{red}{$g$}}

\rput(2,5){\Large $r$}
\rput(7,7.2){\Large $s$}
\rput(7,2.8){\Large $t$}

\rput(5,0.5){(b)}
        \end{pspicture}}

% disc
\rput(35,5){%
        \begin{pspicture}(1,1)(13,9)
          \psline[ArrowInside=->,ArrowInsidePos=0.6](3,8)(5,5)
\psline[ArrowInside=->,ArrowInsidePos=0.6](3,2)(5,5)
\psline[ArrowInside=->,ArrowInsidePos=0.6](5,5)(9,5)

\psline[ArrowInside=->,ArrowInsidePos=0.6](9,5)(11,8)
\psline[ArrowInside=->,ArrowInsidePos=0.6](9,5)(11,2)

\rput(7,4.2){\textcolor{red}{$g$}}
\rput(3.5,4.1){\textcolor{red}{$e$}}
\rput(4.6,7.1){\textcolor{red}{$f$}}

\rput(10.6,6.1){\textcolor{red}{$h$}}
\rput(9.4,3){\textcolor{red}{$i$}}

\rput(7,0.5){(c)}
        \end{pspicture}}

\end{pspicture}
\caption{Topographs locally}
\label{apfig}
\end{figure}
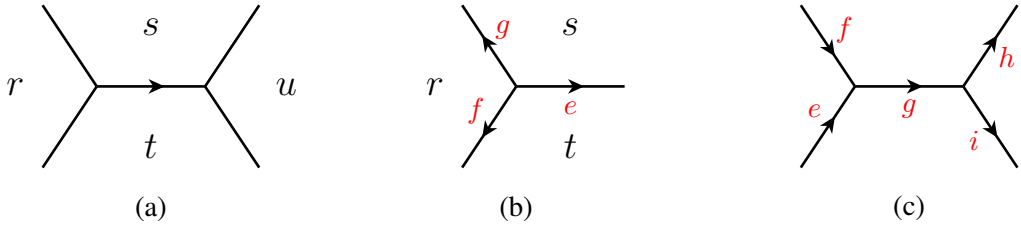
%*************************************

\begin{adef} \la{topdef}
{\rm A graph $\mathcal T$ as above is a {\em topograph} if all the regions are labelled and adjacent region labels, as shown in Figure \ref{apfig}(a), have $r+u=2(s+t)$.
}
\end{adef}

In other words, the region labels satisfy Conway's arithmetic progression rule \cite[p. 9]{con97}, with $r$, $s+t$ and $u$ forming a progression with common difference $\delta$ say. This difference can be used to label the directed edge from $r$ to $u$. Considering this edge in the opposite direction, its label must be $-\delta$.

The labels of regions and edges of a topograph, as in Figure \ref{apfig}(b), are related by
\begin{alignat}{3}\label{r to e}
  e  & =  s+t-r, & \quad f &= r+t-s, &\quad g &= r+s-t, \\
  r  & =  (f+g)/2,& s & =(e+g)/2,& t& = (e+f)/2. \label{e to r}
\end{alignat}
Two topographs are the same if there is a label preserving  isomorphism between them that also preserves orientation. For example, the topographs obtained from a topograph $\mathcal T$ under any reflection are all the same and may be denoted $\mathcal T^*$. Then $\mathcal T$ and $\mathcal T^*$ have opposite orientations and are  not equal in general. Clearly we may multiply every label by a number $\lambda$ to get a new topograph $\lambda \mathcal T$. We may also add two topographs in a natural way, though this depends on how we overlay them. %We see in ...  an important way to multiply them.
It is clear that giving the  labels of three regions meeting at a vertex, or the labels and directions of three edges incident to a  vertex, is enough to specify the entire topograph. %More generally, giving the labels and directions of any three edges that are not on the same region boundary uniquely specifies the topograph as well.

An easy verification, see \cite[Sect. 5.1]{ha22}, confirms that all the edges bordering a single region must be in arithmetic progression.
From this we may derive a simple invariant that can distinguish some topographs. With the diagram in Figure \ref{apfig}(c) we have
\begin{gather*}
  g-e=i-g, \quad g-f = h-g
  \quad \implies  \quad   (g-e)(g-f) = (i-g)(h-g) \\
 \implies  \quad  g^2-eg-fg+ef = ih-ig-hg+g^2,
\end{gather*}
and it follows that the quantity
\begin{equation}\label{dis}
D=-ef-fg-ge,
\end{equation}
where now all edges are directed away from the common vertex, is conserved at every vertex of a particular topograph. This is the  {\em discriminant} of the topograph. %\footnote{An argument for using $-D$ is that it is the determinant in \e{tr}.}).
In terms of regions at a  vertex, as in Figure \ref{apfig}(b),
\begin{align}
D & = r^2+s^2+t^2 -2rs-2rt-2st \notag\\
& =2(r^2+s^2+t^2)-(r+s+t)^2. \label{dis2}
\end{align}

\newpage
\subsection{Region labels}

\begin{wrapfigure}{r}{4cm}
%*************************************
% [a,b,c]
\SpecialCoor
\psset{griddots=5,subgriddiv=0,gridlabels=0pt}
\psset{xunit=0.3cm, yunit=0.3cm, runit=0.2cm}
%\psset{xunit=1cm, yunit=1cm, runit=1cm}
\psset{linewidth=1pt}
\psset{dotsize=7pt 0,dotstyle=*}
%\begin{figure}
\centering
\begin{pspicture}(0,-1)(8,7) %\psgrid

%\psline(0,0)(10,0)(10,5)(0,5)(0,0)
%\psset{arrowscale=2,arrowinset=0.5}
\psset{arrowscale=1.4,arrowinset=0.3,arrowlength=1.1}
\newrgbcolor{light}{0.7 0.7 1.0}
%\newrgbcolor{light}{0 0 1}
%\newrgbcolor{blue2}{0.4 0.4 1.0}
\newrgbcolor{blue2}{0.3 0.3 0.8}

\psline[linecolor=light](1,0)(4,2)(7,0)
\psline[linecolor=light](1,7)(4,5)(7,7)

\psline{->}(4,2)(4,5)

\rput(1,3.5){\Large $a$}
\rput(3.3,3.5){\textcolor{red}{$b$}}
\rput(7,3.5){\Large $c$}

\rput(4,-1.2){$[a,b,c]$}

\end{pspicture}
%\caption{Configuration $[a,b,c]$}
%\label{config}
%\end{figure}
%*************************************
\end{wrapfigure}

The set of region labels of a topograph has a simple characterization.

\begin{theorem} \label{abc}
Suppose a topograph contains the configuration $[a,b,c]$ shown in the diagram. Then its region labels are $a x^2+b xy+cy^2$ for all coprime integers $x$ and $y$.
\end{theorem}
\begin{proof}
  If we follow the direction of edge $b$, moving forward and left $n$ times, we end at the new configuration $C_1$ in \e{con}
\begin{equation}\label{con}
  C_1=[a,2na+b,an^2+bn+c], \qquad C_2=[a+bn+cn^2,b+2nc,c].
\end{equation}
   Moving instead forward and right $n$ times from $b$ leads to $C_2$. Combining these, by going left $a_0$ times and then right $a_1$ times leads to $[a',b',c']$ with
  \begin{align*}
    a' & = a_1^2\big\{ a\left(a_0+ 1/a_1 \right)^2+b\left(a_0+ 1/a_1  \right)+c\big\},\\
    b' & =a_1\big\{ 2a\left(a_0+ 1/a_1  \right)a_0+b\left(a_0+ 1/a_1 +a_0 \right)+2c\big\}, \\
   c' & =aa_0^2+ba_0+c.
  \end{align*}
The continued fractions $\langle a_0 \rangle$ and $\langle a_0,a_1 \rangle$ may be seen, and further left/right turns lead to further partial quotients. In general, let an alternating sequence of left and right turns be listed by  $a_0,a_1, \dots, a_n$, starting with $a_0 \gqs 0$ left turns.  Writing $\langle a_0,a_1, \dots, a_m \rangle = h_m/k_m$ in lowest terms, the configuration reached at the end of this journey along the topograph is $[a',b',c']$ with
\begin{subequations} \label{poly}
\begin{align}
    a' & = a h_n^2+bh_nk_n+ck_n^2,\label{poly:a}\\
    b' & =2a h_nh_{n-1}+b(h_n k_{n-1}+h_{n-1}k_n)+2c k_n k_{n-1}, \label{poly:b}\\
   c' & =a h_{n-1}^2+bh_{n-1}k_{n-1}+ck_{n-1}^2,\label{poly:c}
  \end{align}
\end{subequations}
when $n$ is odd and the same for $n$ even except that the formulas for $a'$ and $c'$ are switched. This may be verified by induction, although  a better way will be shown shortly.

Therefore every region reached by such a path has a label of the form $a x^2+b xy+cy^2$ for coprime integers $x,y \gqs 1$.  By changing the direction and sign of $b$, and switching $a$ and $c$, we may reach the remaining regions on the other side in the same way. These labels  have the desired form with one of $x,y$ negative. It follows that every region in the topograph has the required label type. Finally, since every positive fraction has a finite continued fraction expansion, and corresponding turn sequence, we find  there are regions with labels corresponding to all coprime $x,y \in \Z$.
\end{proof}

A different proof of this basic result is provided in \cite[Sect. 3.1]{sn20}. 
From \e{poly:b} the edge labels in a topograph containing $[a,b,c]$ may be given the more complicated characterization
\begin{equation*}
  2a x x'+b(xy'+x'y)+2c yy' \qquad \text{for} \qquad  \begin{pmatrix} x & x' \\ y & y' \end{pmatrix} \in \text{GL}(2,\Z),
\end{equation*}
with $\text{GL}(2,\Z)$ the group of integer matrices of determinant $\pm 1$.

 We now match the configuration $[a,b,c]$ in a topograph $\mathcal T$ with the binary quadratic form $q(x,y)=ax^2+bxy+cy^2$ and say that $\mathcal T$ contains $q$.
The variables $x$ and $y$ give a simpler way to see the changes to $[a,b,c]=q(x,y)$ at the start of the proof of Theorem \ref{abc}. Moving left $n$ times corresponds to the change of variables $q(x+ny,x)$, and moving right to $q(x,nx+y)$. In general, as we saw in the introduction, there is the matrix action
\begin{equation} \label{mat}
  q(x,y)|M:=q(\alpha x+\beta y, \g x+\delta y) \qquad \text{for} \qquad M=\begin{pmatrix} \alpha &\beta \\ \g &\delta \end{pmatrix},
\end{equation}
which is equivalent to
\begin{equation} \label{tr}
  q(x,y)|M =\frac 12\begin{pmatrix} x \\ y \end{pmatrix}^t M^t \begin{pmatrix} 2a & b \\ b & 2c \end{pmatrix} M \begin{pmatrix} x \\ y \end{pmatrix},
\end{equation}
where $t$ indicates the transpose. The key actions are given by the $\SL(2,\Z)$ matrices from \e{stu}, \e{lru}:
\begin{equation*}
  L: \  \text{go left}, \quad R: \ \text{go right}, \quad S: \  \text{rotate $180^\circ$}.
\end{equation*}
With a sequence of left and right turns $a_0,a_1, \dots, a_n$ and $\langle a_0,a_1, \dots, a_m \rangle = h_m/k_m$ as before, we have by the simple recurrences for $h_m$ and $k_m$ from \cite[Thm. 149]{hawr}, (and $h_{-1}:=1$, $k_{-1}:=0$),
\begin{align}
  M=L^{a_0}R^{a_1} \cdots L^{a_{n-1}}R^{a_{n}} & = \begin{pmatrix} h_n & h_{n-1} \\ k_n & k_{n-1} \end{pmatrix} \qquad \text{with $n$ odd}, \la{modd}\\
  M=L^{a_0}R^{a_1} \cdots R^{a_{n-1}}L^{a_{n}} & = \begin{pmatrix} h_{n-1} & h_n \\ k_{n-1} & k_{n} \end{pmatrix} \qquad \text{with $n$ even}. \la{meven}
\end{align}
 This yields
\begin{equation} \label{cf}
  q'(x,y) := q(x,y)|M = \begin{cases}
                         q(h_n x+h_{n-1}y, k_n x+k_{n-1}y), & \mbox{if $n$ odd}  \\
                         q(h_{n-1} x+h_{n}y, k_{n-1} x+k_{n}y), & \mbox{if $n$ even}.
                       \end{cases}
\end{equation}
Then the identities \e{poly} follow from $a'=q'(1,0)$, $c'=q'(0,1)$ and $b'=q'(1,1)-a'-c'$.

If a topograph contains $[a,b,c]$ then the discriminants in \e{dis}, \e{dis2} are easily seen to equal $b^2-4ac$ as expected.

\subsection{$\SL(2,\Z)$ acting on topographs} \la{sl2}

%*************************************
% sl2z figures
\SpecialCoor
\psset{griddots=5,subgriddiv=0,gridlabels=0pt}
\psset{xunit=0.4cm, yunit=0.4cm, runit=0.2cm}
%\psset{xunit=1cm, yunit=1cm, runit=1cm}
\psset{linewidth=1pt}
\psset{dotsize=7pt 0,dotstyle=*}
\begin{figure}[ht]
\centering
\begin{pspicture}(-1,0.5)(29,9) %\psgrid

%\psline(0,0)(10,0)(10,5)(0,5)(0,0)
%\psset{arrowscale=2,arrowinset=0.5}
\psset{arrowscale=1.8,arrowinset=0.3,arrowlength=1.1}
\newrgbcolor{light}{0.7 0.7 1.0}
%\newrgbcolor{light}{0 0 1}
%\newrgbcolor{blue2}{0.4 0.4 1.0}
\newrgbcolor{blue2}{0.3 0.3 0.8}

% 1
\rput(3,4.5){%
        \begin{pspicture}(1,-1)(7,8)

\psline[linecolor=light]{->}(1,0)(4,2)
\psline[linecolor=light]{->}(7,0)(4,2)
\psline[linecolor=light]{->}(4,5)(7,7)
\psline[linecolor=light]{->}(4,5)(1,7)

\psline{->}(4,2)(4,5)

\rput(1,3.5){\Large $a$}
\rput(3.3,3.5){\textcolor{red}{$b$}}
\rput(7,3.5){\Large $c$}

\rput(4,7.7){\Large $a+b+c$}
\rput(4,-0.5){\Large $a-b+c$}

\rput(1.1,5.7){\textcolor{red}{$b+2a$}}
\rput(1.1,1.4){\textcolor{red}{$b-2a$}}
\rput(6.99,5.7){\textcolor{red}{$b+2c$}}
\rput(6.9,1.4){\textcolor{red}{$b-2c$}}
\end{pspicture}}

% 2
\rput(14,4.5){%
        \begin{pspicture}(1,2)(7,8)

\psline[linecolor=light]{->}(7,7)(4,5)
\psline[linecolor=light]{->}(1,7)(4,5)

\psline{->}(4,2)(4,5)

\rput(1.5,5.3){$q|U$}
\rput(5.5,7){$q|U^2$}

\rput(4.8,3.5){$q$}

\end{pspicture}}

% 3
\rput(25,4.5){%
        \begin{pspicture}(1,-1)(7,8)

\psline[linecolor=light]{->}(1,0)(4,2)
\psline[linecolor=light]{->}(7,0)(4,2)
\psline[linecolor=light]{->}(4,5)(7,7)
\psline[linecolor=light]{->}(4,5)(1,7)

\psline{->}(4,2)(4,5)

\rput(4.8,3.5){$q$}

\rput(1.1,5.7){$q|L$}
\rput(1.1,1.4){$q|L^{-1}$}
\rput(6.99,5.7){$q|R$}
\rput(7.1,1.4){$q|R^{-1}$}
\end{pspicture}}

\end{pspicture}
\caption{The $\SL(2,\Z)$ action}
\label{sl2fig}
\end{figure}
%*************************************
Looking at the $\SL(2,\Z)$ action on $q=[a,b,c]$ in more detail, we see with Figure \ref{sl2fig},
\begin{equation} \la{uu}
\begin{aligned}
  q|L & = [a,b+2a,a+b+c], \qquad &  q|S & = [c,-b,a],\\
  q|L^{-1} & = [a,b-2a,a-b+c], &  q|U & = [a+b+c,-b-2a,a], \\
q|R & = [a+b+c,b+2c,c], &  q|U^2 & = [c,-b-2c,a+b+c], \\
 q|R^{-1} & = [a-b+c,b-2c,c]. &   &
\end{aligned}
\end{equation}
These elements $L$, $R$, $S$, $U$ let us maneuver around each topograph. Of course the action of $-I$ has no effect and we may pass to $\G=\PSL(2,\Z)$. As reviewed in \cite[Sect. 3.1]{ri21}, the underlying directed tree of a topograph $\mathcal T$ exactly reflects the structure of $\G$ so that for any $q_1$ and $q_2$ on $\mathcal T$ there is a unique $M\in \G$ giving $q_2=q_1|M$.
Here we must be clear to specify where the $q_i$ are on $\mathcal T$ since each configuration can appear more than once -- see Section \ref{ive}.
We have $q_1 \sim q_2$ if and only if they appear on the same topograph. Each edge corresponds to two forms which may be displayed by using directed half-edges as in Figure \ref{clas}. This gives a clear way to visualize  equivalence classes of forms.

%*************************************
%  directed half-edges
\SpecialCoor
\psset{griddots=5,subgriddiv=0,gridlabels=0pt}
\psset{xunit=0.35cm, yunit=0.35cm, runit=0.2cm}
%\psset{xunit=1cm, yunit=1cm, runit=1cm}
\psset{linewidth=1pt}
\psset{dotsize=7pt 0,dotstyle=*}
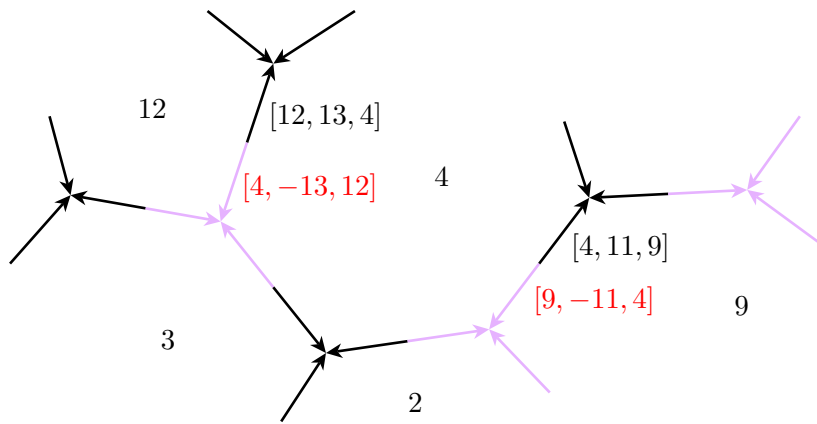
\begin{figure}[ht]
\centering
\begin{pspicture}(0,12)(32,29) %\psgrid

%\psline(0,0)(10,0)(10,5)(0,5)(0,0)
%\psset{arrowscale=2,arrowinset=0.5}
\psset{arrowscale=1.6,arrowinset=0.3,arrowlength=1.1}
\newrgbcolor{light}{0.9 0.7 1.0}
%\newrgbcolor{light}{0 0 1}
%\newrgbcolor{blue2}{0.4 0.4 1.0}
\newrgbcolor{blue2}{0.3 0.3 0.8}

\psline{->}(1,18.4)(3.3,21)
\psline{->}(2.5,24)(3.3,21)
\psline{->}(6.15, 20.5)(3.3,21)

\psline{->}(8.5,28)(11,26)
\psline{->}(14.1,28)(11,26)
\psline{->}(10,23)(11,26)

\psline[linecolor=light]{->}(6.15, 20.5)(9,20)
\psline[linecolor=light]{->}(10,23)(9,20)
\psline[linecolor=light]{->}(11., 17.5)(9,20)

\psline{->}(11., 17.5)(13,15)
\psline{->}(11.3,12.4)(13,15)
\psline{->}(16.1, 15.45)(13,15)

\psline[linecolor=light]{->}(16.1, 15.45)(19.2, 15.9)
\psline[linecolor=light]{->}(21.5,13.5)(19.2, 15.9)
\psline[linecolor=light]{->}(21.1, 18.4)(19.2, 15.9)

\psline{->}(21.1, 18.4)(23, 20.9)
\psline{->}(22.05, 23.8)(23, 20.9)
\psline{->}(26., 21.05)(23, 20.9)

\psline[linecolor=light]{->}(26., 21.05)(29, 21.2)
\psline[linecolor=light]{->}(31,24)(29, 21.2)
\psline[linecolor=light]{->}(32,19)(29, 21.2)

\rput(6.4,24.3){$12$}
\rput(7,15.5){$3$}
\rput(17.4,21.7){$4$}
\rput(16.4,13.1){$2$}
\rput(28.8,16.8){$9$}

\rput(13,24){$[12,13,4]$}
\rput(12.4,21.3){\textcolor{red}{$[4,-13,12]$}}

\rput(24.2,19){$[4,11,9]$}
\rput(23.2,17){\textcolor{red}{$[9,-11,4]$}}

\end{pspicture}
\caption{Visualizing all forms in an equivalence class}
\label{clas}
\end{figure}
%*************************************

Conway described in \cite[pp. 27 - 33]{con97} a striking natural realization in $\H$ of any topograph by letting its regions be the interiors of  fans as in Figure \ref{fanf}. The values of a form $q_0(x,y)$  then appear with the region label $q_0(u,v)$ on the fan $F_{u/v}$. Alternatively, from our point of view, use $q_0$  as an anchor point and  associate it with the arc $\psi$ that is the border between the domain $\Fd$ in \e{fd} and $S\Fd$, and directed rightwards. Other forms $q=q_0|M$ are then associated  to the directed arc $M \psi$.

\subsection{Variations}

Definition \ref{topdef} has the advantage  of showing how simply topographs may be defined, with Theorem \ref{abc} demonstrating the way quadratic forms  naturally  emerge. (An even simpler description is seen in Definition \ref{edgdef} below.) The original definition
in \cite[pp. 5--10]{con97}, see also \cite[Chap. 9]{wei17}, is in terms of lax bases and superbases, and this construction extends to quadratic forms in any number of variables. The definition  in \cite{ha22} is based on the Farey diagram with each region corresponding to a fraction. The emphasis in these works is to display the values of a single quadratic form. In \cite{msw19} they consider  variations, including topographs with all vertices of degree 4 or all of degree 6. Further interesting generalizations appear in \cite{sn20}.

The theory developed in this section has assumed that the region labels are real numbers, but applies to any additive subgroup of $\C$.  Definition \ref{topdef} makes sense  in fact for  labels in any abelian group. For example, reducing the integer region labels of a topograph   mod $|D|$, or mod $p$ for prime $p$ dividing $D$, lets us find its genus; see \cite[Sect 2]{cox}.

A slight variant of Definition \ref{topdef}  may be given by starting with the edges:
\begin{adef} (Edges primary) \la{edgdef}
{\rm Let $\mathcal T$ be a tree drawn in the plane, where all vertices have degree 3. Label the directed edges of $\mathcal T$. An edge may be considered in the opposite direction if the sign of its label is switched. Then $\mathcal T$ with this labelling is an {\em edge topograph} if every  directed path along the border of a region has successive edge labels in arithmetic progression.
}
\end{adef}

For labels in $\R$ this definition is equivalent to Definition \ref{topdef}  -- with \e{e to r} label each region by one half the difference of its border arithmetic progression, (this is independent of the direction taken). If $A$ is any abelian group and $\phi:A \to A$ sends $a$ to $a+a$ then we can also pass between Definitions \ref{topdef} and \ref{edgdef} with region labels in $A$ if  differences of edge labels are always in $\phi(A)$ and $\phi$ is injective. For example, integer edge labels of the same parity in Definition \ref{edgdef} give integer topographs in  Definition \ref{topdef}.

\section{Classifying topographs} \la{cla}

For all the following sections of this paper we assume that the binary quadratic forms under discussion have integer coefficients. Correspondingly, all topographs have integer region and edge labels.% We do not assume they are primitive.

\begin{lemma} \label{gcd}
On any topograph the $\gcd$ of the labels of three regions meeting at any vertex is always the same. Similarly the  $\gcd$ of labels of the three edges  incident to any vertex is fixed. The edge $\gcd$ must equal or be double the region $\gcd$. Edge labels must have the same parity as $D$.
\end{lemma}

The easy proof is omitted. In particular, if a topograph contains a primitive form $[a,b,c]$ then its region $\gcd$ is $1$ and its edge $\gcd$ is $1$ if $D$ is odd and $2$ if $D$ is even. So we may call a topograph primitive if its region $\gcd$ is $1$.
(We are following the convention that $\gcd(0,0,0)=0$. Then Lemma \ref{gcd} applies to the $0$ topograph which has all region labels $0$,  and is not primitive.)

Computing $h(D)$ then corresponds to counting primitive topographs of discriminant $D$. The principal forms
\begin{equation}\label{prin}
  [1,0,-D/4] \quad \text{if} \quad D\equiv 0 \bmod 4, \qquad   [1,1,(1-D)/4] \quad \text{if} \quad D\equiv 1 \bmod 4
\end{equation}
 show that $h(D)\gqs 1$. In the rest of this section we follow and extend Conway's  classification of topographs in \cite[pp. 8--26]{con97}. This uses the signs $-$, $0$ or $+$ of their regions to develop their structural properties.   The elementary {\em climbing lemma} is key: if there is a configuration $[a,b,c]$ %in the topograph
with all entries positive, then all edge and adjacent region labels must strictly increase as we continue forward. Similarly,  all entries negative  gives strictly decreasing labels. %See ...
 %All the regions of the $0$ topograph are lakes.

\begin{adef} \la{deflr}
{\rm A region with label $0$ is called a {\em lake}. A {\em river} in a topograph consists of all edges  that border a positively labeled region on one side and a negatively labeled region on the other.
}
\end{adef}

\subsection{$0+$ topographs} \la{0+}
The topographs in the $0+$  family contain at least one lake, at least one positive region and no negative regions. As in \e{con} with $a=0$, the regions adjacent to a lake have labels $bn+c$ for $n\in \Z$. So we must have $b=0$ and $c$ a positive integer. The discriminant is $D=0$, the region labels are $cy^2$ for $y\in \Z$ and the edge labels are all the even multiples of $c$. By the climbing lemma all regions and edges strictly increase moving away from the lake, so there is only one lake.  There is  one primitive  topograph in this family, containing $[0,0,1]$ and shown in Figure \ref{det0}. The  $0-$ topographs are just $-1$ times the $0+$ topographs.

%*************************************
%  det 0
\SpecialCoor
\psset{griddots=5,subgriddiv=0,gridlabels=0pt}
\psset{xunit=0.18cm, yunit=0.18cm, runit=0.18cm}
%\psset{xunit=1cm, yunit=1cm, runit=1cm}
\psset{linewidth=1pt}
\psset{dotsize=7pt 0,dotstyle=*}
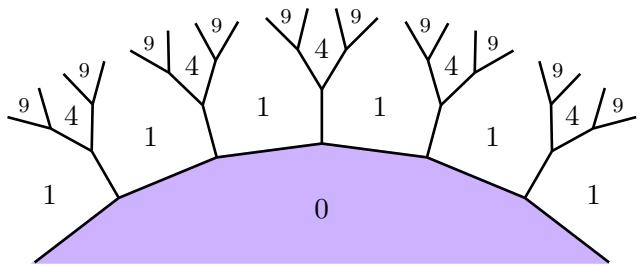
\begin{figure}[ht]
\centering
\begin{pspicture}(8,21)(52,41) %\psgrid

%\psline(0,0)(10,0)(10,5)(0,5)(0,0)
%\psset{arrowscale=2,arrowinset=0.5}
\psset{arrowscale=1.6,arrowinset=0.3,arrowlength=1.1}
\newrgbcolor{light}{0.9 0.7 1.0}
%\newrgbcolor{light}{0 0 1}
%\newrgbcolor{blue2}{0.4 0.4 1.0}
\newrgbcolor{blue2}{0.3 0.3 0.8}

\newrgbcolor{light}{0.8 0.7 1.0}
\psline[linecolor=light,fillstyle=solid,fillcolor=light](51.2132, 21.2132)(45., 25.9808)(37.7646, 28.9778)(30., 30.)(22.2354, 28.9778)(15., 25.9808)(8.7868, 21.2132)(51.2132, 21.2132)

\rput(30., 25.2){$0$}

\psline(51.2132, 21.2132)(45., 25.9808)(37.7646, 28.9778)(30., 30.)(22.2354, 28.9778)(15., 25.9808)(8.7868, 21.2132)

%t=6
\psline(30., 30.)(30., 34.)(28.2124, 36.9568)(25.8711, 39.2836)
\psline(28.2124, 36.9568)(28.9529, 39.9863)
\psline(30., 34.)(31.7876, 36.9568)(34.1289, 39.2836)
\psline(31.7876, 36.9568)(31.0471, 39.9863)

\rput(34.3074, 32.7177){$1$}
\rput(30., 37){$4$}
\rput(27.4166, 39.4154){$_9$}
\rput(32.5834, 39.4154){$_9$}

%t=5
\psline(37.7646, 28.9778)(38.7998, 32.8415)(37.8384, 36.1602)(36.1792,39.0137)
\psline(37.8384, 36.1602)(39.3378, 38.8948)
\psline(38.7998, 32.8415)(41.2918, 35.2349)(44.1555, 36.8764)
\psline(41.2918, 35.2349)(41.3606, 38.3528)

\rput(42.6286, 30.488){$1$}
\rput(39.5763, 35.7393){$4$}
\rput(37.7061, 38.741){$_9$}
\rput(42.6968,37.4038){$_9$}

%t=4
\psline(45., 25.9808)(47., 29.4449)(46.9303, 32.8993)(46.0661, 36.085)
\psline(46.9303, 32.8993)(49.0864, 35.1527)
\psline(47., 29.4449)(50.0265, 31.1117)(53.2175, 31.9562)
\psline(50.0265, 31.1117)(50.8999, 34.1056)

\rput(50.0891, 26.1807){$1$}
\rput(48.5, 32.0429){$4$}
\rput(47.4704, 35.4265){$_9$}
\rput(51.945,32.8431){$_9$}

%t=7
\psline(22.2354, 28.9778)(21.2002, 32.8415)(18.7082, 35.2349)(15.8445, 36.8764)
\psline(18.7082, 35.2349)(18.6394, 38.3528)
\psline(21.2002, 32.8415)(22.1616, 36.1602)(23.8208, 39.0137)
\psline(22.1616, 36.1602)(20.6622, 38.8948)

\rput(25.6926, 32.7177){$1$}
\rput(20.4237, 35.7393){$4$}
\rput(17.3032, 37.4038){$_9$}
\rput(22.2939, 38.741){$_9$}

%t=8
\psline(15., 25.9808)(13., 29.4449)(9.9735, 31.1117)(6.78248, 31.9562)
\psline(9.9735, 31.1117)(9.10006, 34.1056)
\psline(13., 29.4449)(13.0697, 32.8993)(13.9339, 36.085)
\psline(13.0697, 32.8993)(10.9136, 35.1527)

\rput(17.3714, 30.488){$1$}
\rput(11.5, 32.0429){$4$}
\rput(8.055, 32.8431){$_9$}
\rput(12.5296, 35.4265){$_9$}

\rput(9.91087, 26.1807){$1$}

\end{pspicture}
\caption{A primitive topograph with $D=0$}
\label{det0}
\end{figure}
%*************************************

\subsection{$0{+}-$ topographs} \label{opm}
Here we have regions of all three signs. The regions adjacent to a lake have labels $b'n+c'$ for $n\in \Z$, as before, and we must have $b'\neq 0$ for a $0{+}-$ topograph.

If $c'\equiv 0 \bmod b'$ then the topograph contains $[0,b',0]$ and  has two adjacent lakes (separated by Conway's weir). This gives a primitive topograph $\mathcal T$ with $D=1$, containing $[0,1,0]=xy$, with regions labelled by $\Z$ and edges labeled by the odd numbers. See the left side of Figure \ref{1 4}. Then $-\mathcal T$ equals $\mathcal T$ and, in the other cases, $b'\mathcal T$ is non-primitive with $D=b'^2$.

%*************************************
% combined 1 4
\SpecialCoor
\psset{griddots=5,subgriddiv=0,gridlabels=0pt}
\psset{xunit=0.26cm, yunit=0.26cm, runit=0.23cm}
%\psset{xunit=1cm, yunit=1cm, runit=1cm}
\psset{linewidth=1pt}
\psset{dotsize=7pt 0,dotstyle=*}
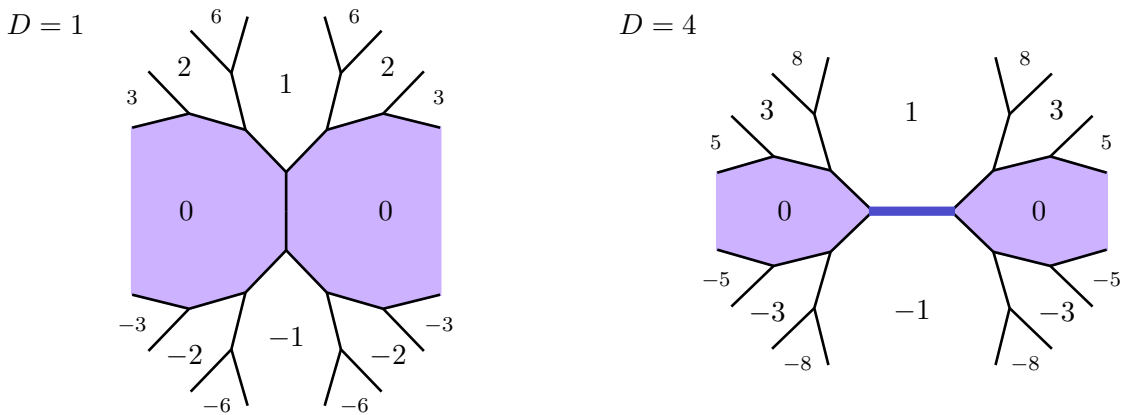
\begin{figure}[ht]
\centering
\begin{pspicture}(0,0)(53,22) %\psgrid

%\psline(0,0)(10,0)(10,5)(0,5)(0,0)
%\psset{arrowscale=2,arrowinset=0.5}
\psset{arrowscale=1.6,arrowinset=0.3,arrowlength=1.1}
\newrgbcolor{light}{0.8 0.8 1.0}
%\newrgbcolor{light}{0 0 1}
%\newrgbcolor{blue2}{0.4 0.4 1.0}
\newrgbcolor{pale}{1 0.7 1}
\newrgbcolor{pale}{1 0.7 0.4}
\newrgbcolor{light}{0.8 0.7 1.0}
%\newrgbcolor{blue2}{0.1 0.2 0.8}
\newrgbcolor{blue2}{0.3 0.3 0.8}

%1
\rput(10,11.4){%
\begin{pspicture}(-10,-10)(10,12) %\psgrid

\psline[linecolor=light,fillstyle=solid,fillcolor=light]
(-7.87479, -4.27258)(-4.96267, -4.99339)(-2.0803, -4.16156)(0., -2.)
(0,0)(0., 2.)(-2.0803, 4.16156)(-4.96267, 4.99339)(-7.87479, 4.27258)(-7.87479, -4.27258)

\psline[linecolor=light,fillstyle=solid,fillcolor=light]
(0,0)(0., -2.)(2.0803, -4.16156)(4.96267, -4.99339)(7.87479, -4.27258)
(7.87479, 4.27258)(4.96267, 4.99339)(2.0803, 4.16156)(0., 2.)(0,0)

%0
\psline(0,0)(0., 2.)(-2.0803, 4.16156)(-4.96267, 4.99339)(-7.87479, 4.27258)
\psline(-4.96267, 4.99339)(-7.04297, 7.15495)
\psline(-2.0803, 4.16156)(-2.80111, 7.07368)(-4.8814, 9.23525)
\psline(-2.80111, 7.07368)(-1.96928, 9.95606)
\psline(0., 2.)(2.0803, 4.16156)(4.96267, 4.99339)(7.87479, 4.27258)
\psline(4.96267, 4.99339)(7.04297, 7.15495)
\psline(2.0803, 4.16156)(2.80111, 7.07368)(4.8814, 9.23525)
\psline(2.80111, 7.07368)(1.96928, 9.95606)

\rput(5.1, 0){$0$}
\rput(0., 6.5){$1$}
\rput(-5.20074, 7.40391){$2$}
\rput(-7.84504, 5.82521){$_{3}$}
\rput(-3.52191, 9.9858){$_{6}$}
\rput(5.20074, 7.40391){$2$}
\rput(3.52191, 9.9858){$_{6}$}
\rput(7.84504, 5.82521){$_{3}$}

%1

\psline(0,0)(0., -2.)(2.0803, -4.16156)(4.96267, -4.99339)(7.87479, -4.27258)
\psline(4.96267, -4.99339)(7.04297, -7.15495)
\psline(2.0803, -4.16156)(2.80111, -7.07368)(4.8814, -9.23525)
\psline(2.80111, -7.07368)(1.96928, -9.95606)
\psline(0., -2.)(-2.0803, -4.16156)(-4.96267, -4.99339)(-7.87479, -4.27258)
\psline(-4.96267, -4.99339)(-7.04297, -7.15495)
\psline(-2.0803, -4.16156)(-2.80111, -7.07368)(-4.8814, -9.23525)
\psline(-2.80111, -7.07368)(-1.96928, -9.95606)

\rput(-5.1, 0){$0$}
\rput(0., -6.5){$-1$}
\rput(5.20074, -7.40391){$-2$}
\rput(7.84504, -5.82521){$_{-3}$}
\rput(3.52191, -9.9858){$_{-6}$}
\rput(-5.20074, -7.40391){$-2$}
\rput(-3.52191, -9.9858){$_{-6}$}
\rput(-7.84504, -5.82521){$_{-3}$}

%\psline[linecolor=red,linewidth=3.4pt](0,-2)(0,2)

\end{pspicture}}

%d4
\rput(42,11.4){%
\begin{pspicture}(-10,-10)(10,12) %\psgrid

\psline[linecolor=light,fillstyle=solid,fillcolor=light]
(2,0)(4.16156,-2.0803)(7.07368,-2.80111)(9.95606,-1.96928)
( 9.95606,1.96928)( 7.07368,2.80111)( 4.16156,2.0803)(2,0)

\psline[linecolor=light,fillstyle=solid,fillcolor=light]
(-2,0)(-4.16156,-2.0803)(-7.07368,-2.80111)(-9.95606,-1.96928)
( -9.95606,1.96928)( -7.07368,2.80111)(-4.16156,2.0803)(-2,0)

%0
\psline(0,0)(2,0)(4.16156,-2.0803)(4.99339,-4.96267)(4.27258,-7.87479)
\psline(4.99339,-4.96267)(7.15495,-7.04297)
\psline(4.16156,-2.0803)(7.07368,-2.80111)( 9.23525,-4.8814)
\psline(7.07368,-2.80111)(9.95606,-1.96928)
\psline(2,0)( 4.16156,2.0803)( 4.99339,4.96267)( 4.27258,7.87479)
\psline( 4.99339,4.96267)( 7.15495,7.04297)
\psline( 4.16156,2.0803)( 7.07368,2.80111)( 9.23525,4.8814)
\psline( 7.07368,2.80111)( 9.95606,1.96928)

\rput( 0,5.1){$1$}
\rput( 6.5,0.){$0$}
\rput( 7.40391,-5.20074){$-3$}
\rput( 5.82521,-7.84504){$_{-8}$}
\rput( 9.9858,-3.52191){$_{-5}$}
\rput( 7.40391,5.20074){$3$}
\rput( 9.9858,3.52191){$_{5}$}
\rput( 5.82521,7.84504){$_{8}$}

%1
\psline(0,0)(-2,0)( -4.16156,2.0803)( -4.99339,4.96267)( -4.27258,7.87479)
\psline( -4.99339,4.96267)( -7.15495,7.04297)
\psline( -4.16156,2.0803)( -7.07368,2.80111)( -9.23525,4.8814)
\psline( -7.07368,2.80111)( -9.95606,1.96928)
\psline(-2,0)( -4.16156,-2.0803)( -4.99339,-4.96267)( -4.27258,-7.87479)
\psline( -4.99339,-4.96267)( -7.15495,-7.04297)
\psline( -4.16156,-2.0803)( -7.07368,-2.80111)( -9.23525,-4.8814)
\psline( -7.07368,-2.80111)( -9.95606,-1.96928)

\rput( 0,-5.1){$-1$}
\rput( -6.5,0.){$0$}
\rput( -7.40391,5.20074){$3$}
\rput( -5.82521,7.84504){$_{8}$}
\rput( -9.9858,3.52191){$_{5}$}
\rput( -7.40391,-5.20074){$-3$}
\rput( -9.9858,-3.52191){$_{-5}$}
\rput( -5.82521,-7.84504){$_{-8}$}

\psline[linewidth=3.5pt,linecolor=blue2](-2.2,0)(2.2,0)

%\psline[linecolor=red,linewidth=3.4pt](0,-2)(0,2)

\end{pspicture}}

\rput(-2.3,20){$D=1$}
\rput(29,20){$D=4$}

\end{pspicture}
\caption{The only topograph of discriminant $1$ and the only primitive topograph of discriminant $4$}
\label{1 4}
\end{figure}
%*************************************

For $c'\not\equiv 0 \bmod b'$, the regions adjacent to the lake must go from positive to negative once as $n$ varies. A river edge separates the positive regions from the negative ones and river edges continue uniquely away from the lake. A river edge corresponds to $[a,b,c]$ with $ac<0$. As there are finitely many possibilities for such a configuration, the  river must repeat, and hence become periodic, or else finish at a second lake. However, in this case the river cannot become periodic as that would mean it was infinite in both directions. The last river edge must correspond to $[a,-a-c,c]$ and hence $D=(a-c)^2 \gqs 4$. %(If $D=1$ then this creates a weir.)
We shall always orient topographs with rivers so that the positive regions are above the river and the negative ones below. By the climbing lemma, labels of regions and edges increase in absolute value moving away from the lakes and river. So, in this $0{+}-$ case, there are exactly two lakes and the river is the unique simple path joining them. The simplest examples of rivers, with lengths $0$ and $1$, are displayed in Figure \ref{1 4}.

\subsection{$+$ topographs} \label{+}
A topograph in the $+$ family has all region labels positive. Direct each edge to make its label positive (or leave undirected if $0$). At each vertex there are three possibilities. If one edge is directed in to it then the other two edges must be directed out by the climbing lemma. If one edge at a vertex is undirected then the other two  must also be directed out.
 The last possibility is that all edges are directed out. Starting at any vertex we may therefore climb down and reach either a unique {\em vertex well}, which is a vertex with out-degree 3 as in Figure \ref{well},
%*************************************
% well
\SpecialCoor
\psset{griddots=5,subgriddiv=0,gridlabels=0pt}
\psset{xunit=0.35cm, yunit=0.35cm, runit=0.2cm}
%\psset{xunit=1cm, yunit=1cm, runit=1cm}
\psset{linewidth=1pt}
\psset{dotsize=7pt 0,dotstyle=*}
\begin{figure}[ht]
\centering
\begin{pspicture}(5,5.5)(27,20) %\psgrid

%\psline(0,0)(10,0)(10,5)(0,5)(0,0)
%\psset{arrowscale=2,arrowinset=0.5}
\psset{arrowscale=1.6,arrowinset=0.3,arrowlength=1.1}
\newrgbcolor{light}{0.9 0.9 1.0}
%\newrgbcolor{light}{0 0 1}
%\newrgbcolor{blue2}{0.4 0.4 1.0}
\newrgbcolor{blue2}{0.3 0.3 0.8}

%top trees

\psline[ArrowInside=->,ArrowInsidePos=0.6](16,12)(16,15.7)(12,17.3)(9,15.6)
\psline[ArrowInside=->,ArrowInsidePos=0.6](12,17.3)(11.2,19.6)
\psline[ArrowInside=->,ArrowInsidePos=0.6](16,15.7)(19.5,17.6)(20.5,19.8)
\psline[ArrowInside=->,ArrowInsidePos=0.6](19.5,17.6)(22.5,16.1)

%bottom left
\psline[ArrowInside=->,ArrowInsidePos=0.6](16,12)(11.8,10.1)(8.2,10.7)(6.3,12.7)
\psline[ArrowInside=->,ArrowInsidePos=0.6](8.2,10.7)(6.3,9.3)
\psline[ArrowInside=->,ArrowInsidePos=0.6](11.8,10.1)(11.3,7.6)(8.6,6.5)
\psline[ArrowInside=->,ArrowInsidePos=0.6](11.3,7.6)(14,6)

%bottom rgt
\psline[ArrowInside=->,ArrowInsidePos=0.6](16,12)(20.1,10.5)(23.7,11.2)(25.1,13)
\psline[ArrowInside=->,ArrowInsidePos=0.6](23.7,11.2)(25.4,9.6)
\psline[ArrowInside=->,ArrowInsidePos=0.6](20.1,10.5)(20.8,7.8)(23.7,6.7)
\psline[ArrowInside=->,ArrowInsidePos=0.6](20.8,7.8)(18.6,6)

\psdot[linecolor=red](16,12)

\rput(12,13.9){$2$}
\rput(20,13.9){$4$}
\rput(16,8.8){$5$}

\rput(10,18){$14$}
\rput(16,18.5){$7$}
\rput(22,18.5){$20$}

\rput(5.5,11.1){$19$}
\rput(8.8,8.8){$10$}
\rput(11.3,5.9){$28$}

\rput(26,11.6){$35$}
\rput(22.9,9.2){$16$}
\rput(21.3,5.9){$38$}

\rput(16.8,14){\textcolor{red}{$_1$}}
\rput(18,10.3){\textcolor{red}{$_7$}}
\rput(13.5,11.8){\textcolor{red}{$_3$}}

\end{pspicture}
\caption{A topograph of discriminant $D=-31$}
\label{well}
\end{figure}
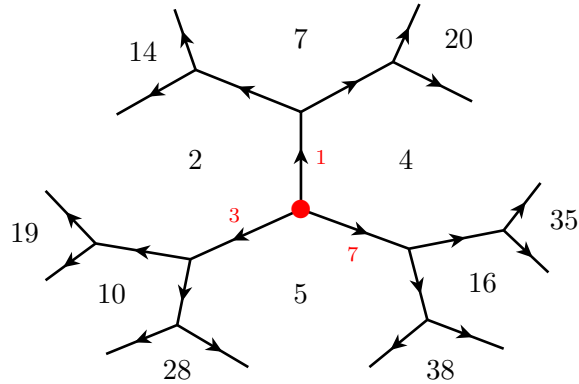
%*************************************
or a unique {\em  edge well} which is an edge with label $0$. Both well types are shown in Figure \ref{-3-4}.  The  quadratic forms on a $+$ topograph are called positive definite.
 The family of $-$ topographs is obtained  by multiplying by $-1$; the corresponding forms are called negative definite. At a vertex of a well of either type, recalling \e{dis}, $D=-ef-fg-ge<0$. Therefore $+$ topographs and $-$ topographs have negative discriminants.

\subsection{$+-$ topographs} \la{+-}
In this final  family, $+-$ topographs have both positive and negative regions but no lakes. Their forms  are called indefinite. As in Section \ref{opm}, the river is a single simple path separating the positive and negative regions.
%*************************************
% river
\SpecialCoor
\psset{griddots=5,subgriddiv=0,gridlabels=0pt}
\psset{xunit=0.22cm, yunit=0.21cm, runit=0.2cm}
%\psset{xunit=1cm, yunit=1cm, runit=1cm}
\psset{linewidth=1pt}
\psset{dotsize=5pt 0,dotstyle=*}
\begin{figure}[ht]
\centering
\begin{pspicture}(4,7)(47,35) %\psgrid

%\psline(0,0)(10,0)(10,5)(0,5)(0,0)
%\psset{arrowscale=2,arrowinset=0.5}
\psset{arrowscale=1.2,arrowinset=0.1,arrowlength=1.0}
\newrgbcolor{light}{0.9 0.9 1.0}
%\newrgbcolor{light}{0 0 1}
%\newrgbcolor{blue2}{0.4 0.4 1.0}
\newrgbcolor{blue2}{0.3 0.3 0.8}

%top trees

\psline(4.5,32)(7.2,29.8)(9.5,26.3)(9.5,22.2)
\psline(8.2,33.6)(7.2,29.8)

\psline(14.5,32)(11.8,29.8)(9.5,26.3)(9.5,22.2)
\psline(10.8,33.6)(11.8,29.8)

\psline(18.2,32.3)(21.3,30.2)(23.7,27.3)(24.5,23)
\psline(21.5,34)(21.3,30.2)
\psline(24,34.2)(25.3,31)(23.7,27.3)
\psline(28,34)(25.3,31)

\psline(39.8,32.3)(36.7,30.2)(34.3,27.3)(33.5,23)
\psline(36.5,34)(36.7,30.2)
\psline(34,34.2)(32.7,31)(34.3,27.3)
\psline(30,34)(32.7,31)

%bottom trees
\psline(17,20)(17,16)(15,12)(11.5,10)
\psline(15.5,8)(15,12)
\psline(17,16)(19,12)(22.5,10)
\psline(18.5,8)(19,12)

\psline(41,20)(41,16)(43,12)(46.5,10)
\psline(42.5,8)(43,12)
\psline(41,16)(39,12)(35.5,10)
\psline(39.5,8)(39,12)

%river
\psline[linewidth=3.5pt,linecolor=blue2](4,20)(9.5,22.2)(17,20)(24.5,23)(33.5,23)(41,20)

\psline[linewidth=3.5pt,linecolor=blue2,linestyle=dashed](41,20)(46.5,22.2)

%\psline[linewidth=3.5pt,linecolor=blue2]{->}(4,20)(7.7,21.48)
%\psline[linewidth=3.5pt,linecolor=blue2]{->}(41,20)(44.7,21.48)

\rput(6,24){$5$}
\rput(17,25){$5$}
\rput(29,27){$8$}
\rput(40,26){$5$}

\rput(6.2,32.6){$_{53}$}
\rput(9.5,30.3){$24$}
\rput(13,32.9){$_{53}$}

\rput(19.8,33.1){$_{60}$}
\rput(23.2,31.2){$29$}
\rput(26,34){$_{69}$}

\rput(38.2,33.1){$_{60}$}
\rput(34.8,31.2){$29$}
\rput(32,34){$_{69}$}

\rput(11,17){$-4$}
\rput(29,17.5){$-3$}
\rput(44.3,17.5){$-4$}

\rput(13.4,9.2){$_{-43}$}
\rput(17,11.2){$-19$}
\rput(20.2,9.2){$_{-40}$}

\rput(44.6,9.2){$_{-43}$}
\rput(41,11.2){$-19$}
\rput(37.4,9.2){$_{-40}$}

\end{pspicture}
\caption{A topograph of discriminant $D=96$ with its periodic river}
\label{riv96}
\end{figure}
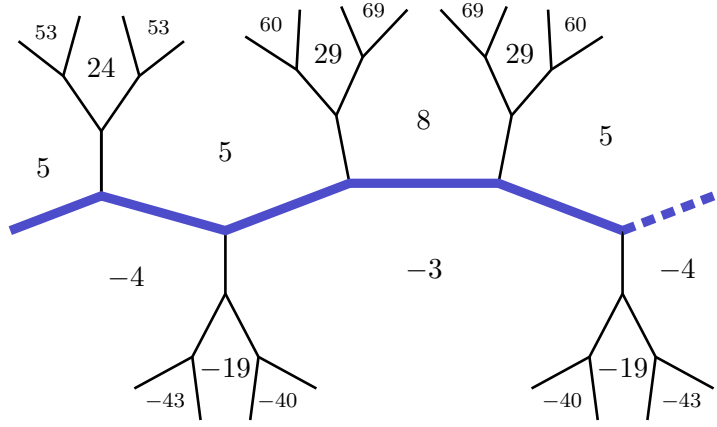
%*************************************
This time it is periodic, meaning that configurations like $[a,b,c]$ on the path of the river, with $a>0>c$, must eventually repeat, as in Figure \ref{riv96}.
Also $b^2+4a|c|=D$ here implies that $D\gqs 4$. We show next that in this case $D$ cannot be a perfect square -- see another proof in \cite[Prop. 11.2]{wei17}.

\begin{lemma} \label{perf}
A topograph with a perfect square discriminant $D$  must contain a lake.
\end{lemma}
\begin{proof}
  Let $q=[a,b,c]$ be a form on this topograph where we may assume $ac \neq 0$. Then its first root, see \e{zeeq}, is rational with a finite continued fraction:
  \begin{equation*}
    \ze_q = \langle a_0,a_1, \dots, a_n\rangle = h_n/k_n.
  \end{equation*}
Let $M=L^{a_0}R^{a_1} \cdots R^{a_n}$, assuming for now that $n$ is odd. Following this path of left and right turns from $q$ leads to $q'=q|M$ with, by \e{cf},
\begin{equation*}
  q'(1,0)=q(h_n,k_n)=k_n^2 \cdot q(\ze_q,1)=0.
\end{equation*}
Hence $q'=[0,b',c']$ and the path has reached a lake. Similarly, $n$ even leads to $q'=[a',b',0]$.
\end{proof}

%Therefore $+-$ topographs have positive discriminants that are not perfect squares.
We will also require the next result.

\begin{lemma} \la{har}
An infinite river in a $+-$ topograph cannot just make $L$ turns or just $R$ turns as we move along it. In particular it must consist of at least $2$ edges before repeating and must contain an $L$ turn and an $R$ turn. The total number of river edges, without repeating, on all $+-$ topographs of a fixed discriminant $D$ is finite.
\end{lemma}
\begin{proof}
If a river makes only $L$ turns, or only $R$ turns, then it borders a single region and its edge labels are in a non-constant arithmetic progression. This is not possible as the river is periodic.

 By counting configurations with $a>0>c$, the  desired number of river edges is
$
  \sum_{b} \sigma (\frac{D-b^2}4  ),
$
where $|b|<\sqrt{D}$, $b \equiv D \bmod 2$ and $\sigma(n)$ equals the number of divisors of $n$. (Then the number of $+-$ topographs with discriminant $D$ is at most half of this sum.)
\end{proof}

\subsection{Summary} \la{summ}
This classification has partitioned the integers $D\equiv 0,1 \bmod 4$ into four sets. Topographs of these discriminants (primitive or not) have the following features.
\begin{enumerate}
 \item If $D = 0$ then we can have the $0$ topograph with all regions $0$. Otherwise there is one lake, and other region labels are all positive or all negative. (Cases $0$, $0+$, $0-$.)
  \item If $D < 0$   their region labels take only positive values  or only negative values. These are the only topographs with wells. (Cases $+$, $-$.)
  \item If $D > 0$ is  not a perfect square then each topograph  has an infinite periodic river, separating positive and negative regions. There are no lakes. (Case $+-$.)
  \item If $D > 0$ is a perfect square   then there are exactly two lakes and they are joined by a river, possibly of  length zero. They contain both positive and negative regions. (Case $0{+}-$.)
\end{enumerate}
The square cases lead to a simple topographic proof of a result of Gauss. Let $\phi$ be Euler's totient function.

\begin{prop} \la{hm2}
  We have $h^*(0)=\infty$, $h(0)=1$ and for all integers $m\gqs 1$
  \begin{equation*}
    h^*(m^2)=m, \qquad h(m^2)=\phi(m).
  \end{equation*}
\end{prop}
\begin{proof}
  As we have seen, the topograph containing $[0,0,1]$ and the topograph containing $[0,0,-1]$ are the only primitive ones of discriminant $0$. Then $h(0)=1$ as we are not including topographs with negative region labels in this count. The topographs containing  $[0,0,c]$ show that $h^*(0)=\infty$.

  Topographs of discriminant $m^2$ for $m \gqs 1$ either contain $[0,m,0]$, with two adjacent lakes, or a river between two nonadjacent lakes. In the latter case consider the rightmost river edge. It corresponds to the configuration $[a,-a-c,c]$ with $a>0>c$. Hence $(a-c)^2=m^2$ and $a+|c|=m$. This is possible only if $m\gqs 2$, which we now assume. Therefore $a$ with $1\lqs a \lqs m-1$ specifies this topograph. It follows that $h^*(m^2)=m$. Also $[a,-a-c,c]$ is primitive iff $\gcd(a,m)=1$, giving $h(m^2)=\phi(m)$. Our argument shows that when $m=1$ the topograph must contain two adjacent lakes and $[0,1,0]$. Hence $h^*(1)=h(1)=1$.
\end{proof}

Section \ref{sqd} gives another characterization of topographs with square discriminants. For any discriminant $D$ it is easily seen by factoring out the $\gcd$ of the coefficients of $[a,b,c]$ that
\begin{equation*}
  h^*(D)=\sum_{n^2 \mid D}h(D/n^2).
\end{equation*}
Hence $h^*(D)=h(D)$ when $D$ is squarefree, or more generally one of the {\em fundamental discriminants}:
\begin{equation*}
   \dots,   -19,-15,-11, -8,-7, -4,-3, 5, 8, 12, 13, 17, 21, 24, \dots .
\end{equation*}
 Apart from $h^*(0)$, we will see that $h^*(D)$ and $h(D)$ are finite for all $D$ and may be calculated in various ways.
The relationship between binary quadratic forms and ideal classes in quadratic number fields  is not treated in this paper, but see for example \cite{bue89}, \cite[Chap. 5]{coh93}, \cite[Chap. 10]{z81} for this important connection.

%lemmerm cor 3.33

%We have seen that the number $h(D)$ of classes of forms of each discriminant is finite, given by...

\section{Class numbers for $D<0$} \la{<}

\subsection{Counting wells}
As discussed in Section \ref{sl2}, topographs correspond to quadratic form equivalence classes and so may be used to count them. The class numbers $h(D)$ and $h^*(D)$ count the number of primitive and not necessarily primitive topographs of discriminant $D$, respectively. For $D<0$ the Hurwitz number $H(|D|)$ is the same as $h^*(D)$ except that, for $a\gqs 1$, it counts the topograph containing $[a,a,a]$ with weight $1/3$, and the topograph containing $[a,0,a]$ with weight $1/2$. These topographs are multiples of those shown in Figure \ref{-3-4}, with rotational symmetries of order $3$ and $2$. No other primitive topographs  of negative discriminant have orientation preserving symmetry; see Corollary \ref{aut}.

%*************************************
% combined -3-4
\SpecialCoor
\psset{griddots=5,subgriddiv=0,gridlabels=0pt}
\psset{xunit=0.23cm, yunit=0.23cm, runit=0.23cm}
%\psset{xunit=1cm, yunit=1cm, runit=1cm}
\psset{linewidth=1pt}
\psset{dotsize=7pt 0,dotstyle=*}
\begin{figure}[ht]
\centering
\begin{pspicture}(0,-1)(50,22) %\psgrid

%\psline(0,0)(10,0)(10,5)(0,5)(0,0)
%\psset{arrowscale=2,arrowinset=0.5}
\psset{arrowscale=1.4,arrowinset=0.3,arrowlength=1.1}
\newrgbcolor{light}{0.8 0.8 1.0}
%\newrgbcolor{light}{0 0 1}
%\newrgbcolor{blue2}{0.4 0.4 1.0}
\newrgbcolor{pale}{1 0.7 1}
\newrgbcolor{pale}{1 0.7 0.4}

\rput(11,10.5){%
        \begin{pspicture}(-10,-10)(10,12) %\psgrid

%\psline(0,0)(10,0)(10,5)(0,5)(0,0)
%\psset{arrowscale=2,arrowinset=0.5}
\psset{arrowscale=1.6,arrowinset=0.3,arrowlength=1.1}
\newrgbcolor{light}{0.9 0.7 1.0}
%\newrgbcolor{light}{0 0 1}
%\newrgbcolor{blue2}{0.4 0.4 1.0}
\newrgbcolor{blue2}{0.3 0.3 0.8}

%0
\psline(0,0)(0., 3.)(-1.82628, 5.38006)(-4.59792, 6.52811)(-7.57226, 6.13653)
\psline(-4.59792, 6.52811)(-6.42421, 8.90817)
\psline(-1.82628, 5.38006)(-2.21786, 8.35439)(-4.04415, 10.7345)
\psline(-2.21786, 8.35439)(-1.06981, 11.126)
\psline(0., 3.)(1.82628, 5.38006)(4.59792, 6.52811)(7.57226, 6.13653)
\psline(4.59792, 6.52811)(6.42421, 8.90817)
\psline(1.82628, 5.38006)(2.21786, 8.35439)(4.04415, 10.7345)
\psline(2.21786, 8.35439)(1.06981, 11.126)

\rput(4.41673, 2.55){$1$}
\rput(0., 7.5){$3$}
\rput(-4.56571, 8.95015){$7$}
\rput(-7.36956, 7.67616){$_{13}$}
\rput(-2.60944, 11.3287){$_{19}$}
\rput(4.56571, 8.95015){$7$}
\rput(2.60944, 11.3287){$_{19}$}
\rput(7.36956, 7.67616){$_{13}$}

%1
\psline(0,0)(-2.59808, -1.5)(-3.74613, -4.27164)(-3.35455, -7.24597)(-1.52826, -9.62603)
\psline(-3.35455, -7.24597)(-4.5026, -10.0176)
\psline(-3.74613, -4.27164)(-6.12619, -6.09792)(-7.27424, -8.86956)
\psline(-6.12619, -6.09792)(-9.10052, -6.4895)
\psline(-2.59808, -1.5)(-5.57241, -1.10842)(-7.95247, 0.717863)(-9.10052, 3.4895)
\psline(-7.95247, 0.717863)(-10.9268, 1.10944)
\psline(-5.57241, -1.10842)(-8.34405, -2.25647)(-11.3184, -1.86489)
\psline(-8.34405, -2.25647)(-10.1703, -4.63653)

\rput(-4.41673, 2.55){$1$}
\rput(-6.49519, -3.75){$3$}
\rput(-5.4682, -8.4291){$7$}
\rput(-2.96297, -10.2203){$_{13}$}
\rput(-8.50625, -7.92421){$_{19}$}
\rput(-10.0339, -0.521054){$7$}
\rput(-11.1157, -3.40452){$_{19}$}
\rput(-10.3325, 2.54415){$_{13}$}

%2
\psline(0,0)(2.59808, -1.5)(5.57241, -1.10842)(7.95247, 0.717863)(9.10052, 3.4895)
\psline(7.95247, 0.717863)(10.9268, 1.10944)
\psline(5.57241, -1.10842)(8.34405, -2.25647)(11.3184, -1.86489)
\psline(8.34405, -2.25647)(10.1703, -4.63653)
\psline(2.59808, -1.5)(3.74613, -4.27164)(3.35455, -7.24597)(1.52826, -9.62603)
\psline(3.35455, -7.24597)(4.5026, -10.0176)
\psline(3.74613, -4.27164)(6.12619, -6.09792)(7.27424, -8.86956)
\psline(6.12619, -6.09792)(9.10052, -6.4895)

\rput(0., -5.1){$1$}
\rput(6.49519, -3.75){$3$}
\rput(10.0339, -0.521054){$7$}
\rput(10.3325, 2.54415){$_{13}$}
\rput(11.1157, -3.40452){$_{19}$}
\rput(5.4682, -8.4291){$7$}
\rput(8.50625, -7.92421){$_{19}$}
\rput(2.96297, -10.2203){$_{13}$}

\psdot[linecolor=red](0,0)

\end{pspicture}}

\rput(42,11.4){%
\begin{pspicture}(-10,-10)(10,12) %\psgrid

%\psline(0,0)(10,0)(10,5)(0,5)(0,0)
%\psset{arrowscale=2,arrowinset=0.5}
\psset{arrowscale=1.6,arrowinset=0.3,arrowlength=1.1}
\newrgbcolor{light}{0.9 0.7 1.0}
%\newrgbcolor{light}{0 0 1}
%\newrgbcolor{blue2}{0.4 0.4 1.0}
\newrgbcolor{blue2}{0.3 0.3 0.8}

%0
\psline(0,0)(0., 2.)(-2.0803, 4.16156)(-4.96267, 4.99339)(-7.87479, 4.27258)
\psline(-4.96267, 4.99339)(-7.04297, 7.15495)
\psline(-2.0803, 4.16156)(-2.80111, 7.07368)(-4.8814, 9.23525)
\psline(-2.80111, 7.07368)(-1.96928, 9.95606)
\psline(0., 2.)(2.0803, 4.16156)(4.96267, 4.99339)(7.87479, 4.27258)
\psline(4.96267, 4.99339)(7.04297, 7.15495)
\psline(2.0803, 4.16156)(2.80111, 7.07368)(4.8814, 9.23525)
\psline(2.80111, 7.07368)(1.96928, 9.95606)

\rput(5.1, 0){$1$}
\rput(0., 6.5){$2$}
\rput(-5.20074, 7.40391){$5$}
\rput(-7.84504, 5.82521){$_{10}$}
\rput(-3.52191, 9.9858){$_{13}$}
\rput(5.20074, 7.40391){$5$}
\rput(3.52191, 9.9858){$_{13}$}
\rput(7.84504, 5.82521){$_{10}$}

%1
\psline(0,0)(0., -2.)(2.0803, -4.16156)(4.96267, -4.99339)(7.87479, -4.27258)
\psline(4.96267, -4.99339)(7.04297, -7.15495)
\psline(2.0803, -4.16156)(2.80111, -7.07368)(4.8814, -9.23525)
\psline(2.80111, -7.07368)(1.96928, -9.95606)
\psline(0., -2.)(-2.0803, -4.16156)(-4.96267, -4.99339)(-7.87479, -4.27258)
\psline(-4.96267, -4.99339)(-7.04297, -7.15495)
\psline(-2.0803, -4.16156)(-2.80111, -7.07368)(-4.8814, -9.23525)
\psline(-2.80111, -7.07368)(-1.96928, -9.95606)

\rput(-5.1, 0){$1$}
\rput(0., -6.5){$2$}
\rput(5.20074, -7.40391){$5$}
\rput(7.84504, -5.82521){$_{10}$}
\rput(3.52191, -9.9858){$_{13}$}
\rput(-5.20074, -7.40391){$5$}
\rput(-3.52191, -9.9858){$_{13}$}
\rput(-7.84504, -5.82521){$_{10}$}

\psline[linecolor=red,linewidth=3.4pt](0,-2)(0,2)

\end{pspicture}}

\rput(-2.3,20){$D=-3$}
\rput(29,20){$D=-4$}

\end{pspicture}
\caption{The topographs of discriminants $D=-3$ and $D=-4$}
\label{-3-4}
\end{figure}
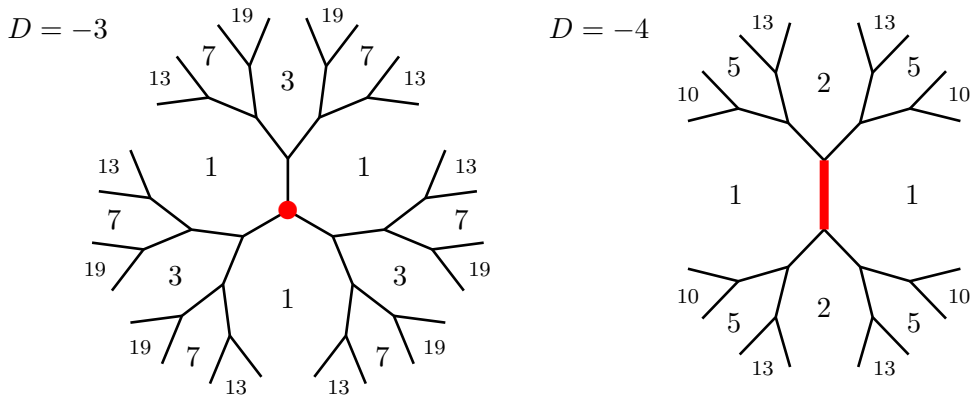
%*************************************

Topographs of negative discriminant may be counted by their  well configurations from Section \ref{+}. %For example, with $a=1$ above we obtain  $h(-3)=h(-4)=1$.

\begin{theorem} \la{class}
We have  $h(-3)=h(-4)=1$. Suppose  $D\equiv 0,1 \bmod 4$ is  $<-4$. Put $n :=|D|$ if $D$ is odd and  $n :=|D|/4$ otherwise. Then, using $\doteq$ to indicate equality provided the conditions following are met,
\begin{equation} \label{hed}
  h(D) \doteq 2 \sum_{\substack{ e>f>g >0 \\ ef+eg+fg=n}} 1 + \sum_{\substack{ e,f >0  \\ e^2+2ef=n}} 1
  +  \sum_{\substack{ e>f >0 \\ ef=n}} 1,
\end{equation}
where the sums are over pairs or triples of  integers with $\gcd=1$. In the first two sums, the pairs or triples should be all odd if $D$ is odd, and not all odd if $D$ is even. The last sum is only included when $D$ is even.
\end{theorem}
\begin{proof}
We count the possible wells of topographs of discriminant $D$.
An edge well corresponds to $[a,0,c]$ with $D=-4ac$ and $\gcd(a,c)=1$. The case $a=c$ gives one topograph with $D=-4$. Otherwise we obtain the last sum in \e{hed} when $D$ is even.

 A vertex well has positively labeled outward directed edges so that $D=-ef-fg-ge$. As in Lemma \ref{gcd} and the discussion after it, we have $\gcd(e,f,g)=1$ or $2$  as $D$ is odd or even, respectively. Let $n=|D|$ in the first case and $e,f,g$ must all be odd. Let $n=|D|/4$ in the second case, where we also replace $e, f, g$ by half their values and these half values cannot all be odd, (or the surrounding regions would all be even by \e{e to r} and the topograph not primitive). The case of equality, $e=f=g=1$ gives one topograph with $D=-3$. Otherwise, if only two edges are equal, say $e$ and $g$, then their contribution is the second sum in \e{hed}. If all edges are unequal, then they may be ordered and we see they produce two  topographs, corresponding to each orientation.
\end{proof}

Theorem \ref{class} seems to be new. Mordell in \cite{mor} proved the  $D$ even case of the next result for Hurwitz numbers. His proof is longer,  linking solutions of $ef+fg+ge=n$ to reduced forms more directly.

\begin{theorem} \la{hur}
Suppose  $D\equiv 0,1 \bmod 4$ is  $<0$. Put $n :=|D|$ if $D$ is odd and  $n :=|D|/4$ otherwise. Then
\begin{equation} \la{hhd}
  H(|D|) \doteq 2 \sum_{\substack{ e>f>g >0 \\ ef+eg+fg=n}} 1 + \sum_{\substack{ e,f >0, \ e\neq f \\ e^2+2ef=n}} 1
  +  \frac 13 \sum_{\substack{ e >0 \\ 3 e^2=n}} 1+ \frac 12 \sum_{\substack{ e,f >0 \\ ef=n}} 1.
\end{equation}
In the first two sums, the pairs or triples should be all odd if $D$ is odd. The last sum is only included when $D$ is even.  %Also $H(0)=-1/12$.
\end{theorem}
\begin{proof}
The proof is very similar to the previous one, removing the $\gcd =1$ condition. Extra terms corresponding to $[e,e,e]$ with weight $1/3$ and $[e,0,e]$ with weight $1/2$ are added. (The formula for $h^*(D)$ is the same as \e{hhd} but with these  extra term weights put equal to $1$.)
\end{proof}

Following Mordell, we may make \e{hhd} much neater.
Let $\Upsilon(n)$ be the number of ordered triple solutions $(e,f,g)$ to $ef+fg+ge=n$ in nonnegative integers, where solutions with one of $e, f, g$ equal to zero count with weight $1/2$. Let $\Upsilon_{odd}(n)$ be the number of  solutions in positive odd integers. %Breaking the solutions into the cases of all distinct, two distinct, three equal and one zero gives...
This notation allows the next simple restatement of Theorem \ref{hur}.  Mordell proved the first part $\Upsilon(n) =   3 H(4n)$, (he was counting forms with even middle coefficients, forcing $|D|=4n$).

\begin{cor} \la{mo}
For positive integers $n$,
\begin{equation*}
  \Upsilon(n) =   3 H(4n), \qquad \Upsilon_{odd}(n) =
   \begin{cases}
                         3 H(n), & \mbox{if $n \equiv 3 \bmod 4$}  \\
                         0, & \mbox{if $n \not\equiv 3 \bmod 4$}.
                       \end{cases}
\end{equation*}
\end{cor}

\subsection{Sums of three squares}
Theorem \ref{hur} may also be used to connect $H(|D|)$ to $r_3(n)$, the number of ways to write $n$ as a sum of squares of three integers. This link goes back to Gauss's work in the {\em Disquisitiones} from 1801.
Krammer's identity \e{krm} implies, for $n \gqs 1$,
\begin{equation} \la{crt}
  (-1)^{n+1} r_3(n) = 4 \sum_{\substack{ e, f, g >0 \\ ef+eg+fg=n}} (-1)^{e+f+g} + 6 \sum_{\substack{ e,f >0 \\ ef=n}} (-1)^{e+f}.
\end{equation}
It is mentioned at the end of \cite{kr93}, that \e{crt} may be used to give a proof of Theorem \ref{gat} below. This proof is supplied in \cite[Sects. 5, 6, 7]{mo17} along with another proof of \e{crt} in \cite[Sect. 3]{mo17}, though Mortenson was unaware of Krammer's paper. The identity \e{crt} is also stated in Crandall \cite[Eq. (6.2)]{cra}. Both \cite{kr93} and  \cite{cra} refer to an earlier $q$-series identity of Andrews as the inspiration for \e{crt}.

 An expanded version of \e{crt} is
\begin{equation} \la{cran}
  (-1)^{n+1} \frac{r_3(n)}{12}= 2 \sum_{\substack{ e>f>g >0 \\ ef+eg+fg=n}} (-1)^{e+f+g} + \sum_{\substack{ e,f >0, \ e\neq f \\ e^2+2ef=n}} (-1)^f
  +  \frac 13 \sum_{\substack{ e >0 \\ 3 e^2=n}} (-1)^e + \frac 12 \sum_{\substack{ e,f >0 \\ ef=n}} (-1)^{e+f}.
\end{equation}
By matching up the pieces on the right sides of \e{hhd} and \e{cran} we can link $ r_3(n)$ with $H(|D|)$
 and  obtain a  proof of the next formula that is %a good deal 
 simpler than the proof in \cite{mo17}.

\begin{theorem} \la{gat}
For $n> 0$,
 \begin{equation} \la{cas}
  r_3(n)=   \begin{cases}
                         12 H(4n), & \mbox{if $n \equiv 1, 2 \bmod 4$}  \\
                         12(H(4n)-2 H(n)), & \mbox{if $n \equiv 3 \bmod 8$} \\
                         r_3(n/4), & \mbox{if $n \equiv 0 \bmod 4$} \\
                         0 , & \mbox{if $n \equiv 7 \bmod 8$}.
                       \end{cases}
\end{equation}
\end{theorem}
\begin{proof}
As squares are $\equiv 0, 1, 4 \bmod 8$ we easily see no solutions when $n \equiv 7 \bmod 8$ and only even square solutions when $n \equiv 0 \bmod 4$.
 Checking the parity possibilities when $n \equiv 1, 2 \bmod 4$ confirms that \e{hhd} and \e{cran} match, giving $r_3(n) = 12 H(4n)$. The final case has $n \equiv 3 \bmod 8$. Then by Theorem \ref{hur}, $H(4n)$ equals the right side of \e{hhd}
where
 the first sum  can have $e, f, g$ all odd or exactly one even. In the second and third sums  $e, f$ must be odd. Theorem \ref{hur} also implies
 \begin{equation*}
  H(n) \doteq  2 \sum_{\substack{ e>f>g >0 \\ ef+eg+fg=n}} 1 + \sum_{\substack{ e,f >0, \ e\neq f \\ e^2+2ef=n}} 1
  +  \frac 13 \sum_{\substack{ e >0 \\ 3 e^2=n}} 1,
\end{equation*}
with $e, f, g$  odd in all sums. Therefore
\begin{equation*}
  H(4n) - 2H(n)=  2 \sum_{\substack{ e>f>g >0 \\
  \text{one even}\\ ef+eg+fg=n }} 1 - 2 \sum_{\substack{ e>f>g >0\\
  \text{all odd} \\ ef+eg+fg=n }} 1 - \sum_{\substack{ e,f >0, \ e\neq f\\ \text{all odd} \\ e^2+2ef=n }} 1
  -  \frac 13 \sum_{\substack{ e >0 \\ 3 e^2=n}} 1+ \frac 12 \sum_{\substack{ e,f >0 \\ ef=n}} 1,
\end{equation*}
agreeing with the right side of \e{cran}.
\end{proof}

 See also \cite[p. 90]{hiza} for more information on Theorem \ref{gat} and the connection to weight $3/2$ modular forms, as well as further references in \cite{mo17}.

\subsection{Primitive sums of three squares}

%Returning to Theorem \ref{class},  it naturally relates $h(D)$ to the number of primitive solutions of $ef+fg+ge=n$ to give an analog of Corollary \ref{mo}. In another direction,
We have already stated Gauss's often-quoted formula \e{rr3}, relating $h(D)$ to  primitive sums of three squares. Grosswald gives a brief sketch of Gauss's elaborate proof in \cite[pp. 59--60]{gro85}, based on ternary quadratic forms, and there do not seem to be any recent demonstrations of \e{rr3} in the literature. We show next that it follows quickly from Theorem \ref{class} and \e{crt}.

Let $s(n)$ denote $(-1)^{n+1}$ times the right side of \e{crt}. Also let $s'(n)$ be a primitive version:
\begin{equation} \la{crt2}
   s'(n) = 4(-1)^{n+1} \sum_{\substack{ e, f, g >0 \\ ef+eg+fg=n \\ \gcd(e,f,g)=1}} (-1)^{e+f+g} + 6 (-1)^{n+1}\sum_{\substack{ e,f >0 \\ ef=n  \\ \gcd(e,f)=1}} (-1)^{e+f}.
\end{equation}

\begin{lemma} \la{mma}
For all $n\gqs 1$ with $4 \nmid n$ we have $r_3'(n)=s'(n)$.
\end{lemma}
\begin{proof}
Easily, $r_3(n)=\sum_{k^2 | n} r_3'(n/k^2)$ for any $n\gqs 1$. It follows by M\"obius inversion, as in \cite[Eq. (1.3)]{ch07}, that
\begin{equation} \la{tx}
  r_3'(n)=\sum_{k^2 | n} \mu(k) \cdot  r_3(n/k^2).
\end{equation}
Write $s'(n)=\alpha'(n)+\beta'(n)$ with $\alpha'(n)$ the first sum in \e{crt2} and $\beta'(n)$ the second.
Similarly write $s(n)=\alpha(n)+\beta(n)$ where $\alpha(n)$ and $\beta(n)$ omit the $\gcd$ conditions of $\alpha'(n)$  and $\beta'(n)$.
Suppose we have a summand of $\alpha(n)$ with $\gcd(e,f,g)=k$. Then, factoring out $k$,
\begin{equation*}
  (-1)^{e+f+g}   = (-1)^{k(e'+f'+g')}
 = (-1)^{e'+f'+g'}
\end{equation*}
when $k$ is odd. Hence
$\alpha(n)=\sum_{k^2 | n} \alpha'(n/k^2)$ for $n\gqs 1$ when $4 \nmid n$. So for these $n$ values
\begin{equation*}
  \alpha'(n)=\sum_{k^2 | n} \mu(k)  \cdot \alpha(n/k^2).
\end{equation*}
The same is true for $\beta(n)$ and $\beta'(n)$ and hence, for $4 \nmid n$,
\begin{equation} \la{tx2}
 s'(n)=\sum_{k^2 | n} \mu(k)  \cdot s(n/k^2).
\end{equation}
With \e{tx}, \e{tx2} the proof is complete since $r_3(n)$ and $s(n)$ agree.
\end{proof}

\begin{theorem} \la{garx}
  For $n>3$,
\begin{equation} \la{gar}
  r'_3(n)=   \begin{cases}
                         12 h(-4n), & \mbox{if $n \equiv 1, 2 \bmod 4$}  \\
                         12\big(h(-4n)- h(-n)\big), & \mbox{if $n \equiv 3 \bmod 8$}\\
                            0, & \mbox{if $n \equiv 0, 4, 7 \bmod 8$}.
                       \end{cases}
\end{equation}
\end{theorem}
\begin{proof}
The cases $n \equiv 0, 4, 7 \bmod 8$ are easily dealt with, as at the start of the proof of Theorem \ref{gat}.
For the remaining $n$ values, by Lemma \ref{mma},
\begin{equation} \la{pp}
  (-1)^{n+1} \frac{r'_3(n)}{12}  \doteq  2 \sum_{\substack{ e>f>g >0 \\ ef+eg+fg=n}} (-1)^{e+f+g} + \sum_{\substack{ e,f >0,  \\ e^2+2ef=n}} (-1)^f
  +  \sum_{\substack{ e >f >0 \\ ef=n}} (-1)^{e+f},
\end{equation}
for  indices with $\gcd = 1$. As in the proof of Theorem \ref{gat} it is a simple matter to verify that \e{pp} agrees  with the expressions for $ h(-4n)$ or $h(-4n)- h(-n)$ given by \e{hed}.
\end{proof}

The $n \equiv 3 \bmod 8$ cases of Theorems \ref{gat} and \ref{garx} are simplified  with the next lemma, giving the usual quoted form as in \e{rr3}.

\begin{lemma}
For  integers $n \equiv 3 \bmod 8$,
\begin{equation} \la{stan}
  h(-4n)=3h(-n) \quad \text{if \quad $n>3$} \qquad \text{and} \qquad H(4n)=4H(n) \quad \text{if \quad $n>0$}.
\end{equation}
\end{lemma}
\begin{proof}
For the  first equality we quote a result from \cite[Sect. 5.3]{coh93} that is related to Dirichlet's class number formula. Set $w(D)$ to be $2$, $4$, $6$ as $D<-4$, $D=-4$ and $D=-3$, respectively. Any non-square $D \equiv 0, 1 \bmod 4$ can be factored uniquely as $m f^2$ with $m$ a fundamental discriminant. Then
\begin{equation} \la{er}
  \frac{h(D)}{w(D)}= \frac{h(m)}{w(m)} f \prod_{p | f} \left(1-\frac 1p  \left(\frac mp  \right) \right) \qquad (D<0),
\end{equation}
for  $p$ prime and $(\frac mp)$ the Kronecker symbol. Let $m$ be the squarefree part of $-n$ so that $-n = m f^2$. Then $f$ is odd and $m \equiv 5 \bmod 8$. For $D=-n$ and $D=-4n$ we have $D=m f^2$ and $D=m (2f)^2$, respectively, with $m$ fundamental. Hence  \e{er} implies
\begin{equation*}
  h(-4n)= 2\left(1-\frac 12  \left(\frac m2  \right) \right)h(-n) = 3h(-n).
\end{equation*}
The proof of the second equality in \e{stan} is similar, with the details given in \cite[Lemma 2.5]{mo17}.
\end{proof}
%See \cite[p. 74]{z81} for a version of \e{er} that is also valid for non-square $D>0$.

\section{Reduction of  forms} \la{rdf}

In this section we show how continued fractions of roots may be used to efficiently and explicitly reduce each form  to a canonical form, or cycle of forms, in its equivalence class. The connection between reduction and continued fractions goes back to Dirichlet in \cite{dir}; this  is also discussed in \cite{fr05,smi18}. Our treatment here is new,  unifying the cases of $D$ positive, negative and zero. The basic idea is simple: on a topograph, reduction  means descending to a river, lake or well. As usual, forms and topographs are not assumed to be primitive. %We include the case of $D$ a square which adds some complications.

\begin{adef} \la{zer}
{\rm For $q=[a,b,c]$ with $a \neq 0$, its {\em first root $\ze_q$} and {\em second root $\ze'_q$} are given by 
\begin{equation*} %\la{zeeq}
  \ze_q := \frac{-b+\sqrt{D}}{2a}, \qquad \ze'_q :=\frac{-b-\sqrt{D}}{2a} \qquad (a\neq 0),
\end{equation*}
as already seen in \e{zeeq}. For $a=0$ (so that the discriminant $D$ is necessarily a square) and $q\neq [0,0,0]$, set
 \begin{equation} \la{a=000}
  \ze_q=\infty, \quad \ze_q'=-c/b \quad \text{if} \quad b \lqs 0 \qquad \text{and} \qquad
  \ze_q=-c/b, \quad \ze_q'=\infty \quad \text{if} \quad b\gqs 0,
\end{equation}
allowing the value $1/0=\infty$. Clearly $q(\ze_q,1)=0$ when $\ze_q \neq \infty$ and $q(\ze'_q,1)=0$ when $\ze'_q \neq \infty$.
}
\end{adef}

Put $(\begin{smallmatrix}r & s\\ t & u \end{smallmatrix})\infty = r/t$.
When $q_0=q|M$ for $M \in \SL(2,\Z)$ and $q\neq [0,0,0]$, the roots of these forms are related in all cases by
\begin{equation}\label{roots}
  \ze_{q_0}=M^{-1} \ze_q, \qquad  \ze'_{q_0}=M^{-1} \ze'_q.
\end{equation}
This follows by verifying it for the generators $T$ and $S$.

For $D>0$ the roots are in $\R \cup \{\infty\}$ and we see that each form $q$ corresponds to a directed hyperbolic geodesic in $\H$ from $\ze_q$ to $\ze'_q$. The image of any of these geodesics (a semicircle centered on $\R$ or a vertical line) under the action of $\SL(2,\Z)$ is another such geodesic. For $D<0$ the roots are conjugate non-reals, each form $q$ just corresponds to the point $z_q$ in $\H$, and $\SL(2,\Z)$ acts on these points. See Figure \ref{paf}; we will exploit this geometry in Section \ref{di}.

\begin{adef}
{\rm For a quadratic form $q\neq [0,0,0]$, the {\em path associated with  its root $\ze_q \in \C \cup \{\infty\}$} is defined as follows.  Let $a_0$, $a_1$, $a_2, \dots$ be the general continued fraction coefficients of $\ze_q$, as described in Section \ref{co}. Starting at configuration $q$  on the topograph,  make an alternating sequence of  turns, first going left $a_0$ times, then right $a_1$ times, and so on.  {\em Matrices associated with $\ze_q$ and this path} are, taking  coefficients up to $a_r$,
\begin{equation} \la{assoc}
  M=L^{a_0}R^{a_1}L^{a_2} \cdots R^{a_r},
\end{equation}
(ending with $L^{a_r}$ if $r$ is even). If $\ze_q = \infty$ or $0$ then we remain at $q$ and $M=I$. The paths and matrices for the second root $\ze'_q$ are defined analogously.
}
\end{adef}

We will need the next technical lemma.

\begin{lemma} \la{cont}
Suppose you stop at a form $q_1$, partway along the path associated with the root $\ze_q$ of $q$. If you then continue along the path associated with the root $\ze_{q_1}$ of $q_1$, it will be identical to the continuation of the original path from this point. Similarly for the real second roots $\ze'_q$ and $\ze'_{q_1}$.
\end{lemma}
\begin{proof}
Suppose first that $\ze_q = x \in \R$ and $x=\langle a_0,a_1,a_2,\dots \rangle$. We examine the result of stopping after the first edge on the path. For $x=0$ there is no path, but otherwise there are four possibilities:
\begin{enumerate}
  \item If $x\gqs 1$ then $a_0\gqs 1$. Let $q_1=q|L$ so that $z_{q_1}=L^{-1}z_q = x-1 = \langle a_0-1,a_1,a_2,\dots \rangle.$
  \item If $x<0$ then $a_0\lqs -1$. Let $q_1=q|L^{-1}$ so that $z_{q_1}=L z_q = x+1 = \langle a_0+1,a_1,a_2,\dots \rangle.$
  \item If $0<x\lqs 1/2$ then $a_0=0$ and $a_1 \gqs 2$. Let $q_1=q|R$ so that $$z_{q_1}=R^{-1} z_q = \frac x{1-x} = \langle 0,a_1-1,a_2,\dots \rangle.$$
  \item Lastly, if $1/2<x<1$ then $a_0=0$, $a_1 =1$ and $a_2 =\lfloor \frac x{1-x}\rfloor$. Let $q_1=q|R$ so that $z_{q_1} = \frac x{1-x} = \langle a_2,a_3,\dots \rangle.$
\end{enumerate}
So the new path always matches the original one. The argument is identical for the second root $\ze'_q$, and by induction we obtain the lemma for $\ze_q$ and $\ze'_q$ real. The corresponding argument for $\ze_q \in \H$, involving the general continued fraction, is more elaborate. However, we obtain the result in this case very simply with Theorem \ref{negred} and Proposition \ref{towell} below, since they show that all paths  lead directly to the well.
\end{proof}

\subsection{Negative discriminants} \la{ne}
We assume here that all forms and topographs have discriminant $D<0$. Forms $[a,b,c]$ have $a, c>0$, (are positive definite), and topographs have all region labels positive -- see Section \ref{+}. The following definition is standard.

\begin{adef}
{\rm A form $[a,b,c]$ of discriminant $D<0$ is {\em reduced} when $|b|\lqs a \lqs c$. If  $|b|= a$ or $a = c$, it is also required that $b\gqs 0$.
}
\end{adef}

\begin{algo}
{\rm For a form $q$ of discriminant $D<0$,  apply the general continued fraction algorithm to its first root $\ze_{q} \in \H$. This produces $(a_0,a_1, \dots, a_r)$ with an associated path and matrix $M$ as well as $z_0 \in \Fd \cup S\Fd$ with $M^{-1}\ze_{q} =z_0$ by Corollary \ref{lr}. Replace $M$ by $MS$ if necessary to ensure $z_0 $ is in the fundamental domain $\Fd$.  The output is $q_0=q|M$ and $M$.
}
\end{algo}

\begin{theorem} \label{negred}
The form produced by the above algorithm is  reduced, showing that every form of negative discriminant is explicitly equivalent to a reduced form. Two forms (not necessarily primitive) are equivalent if and only if their reduced forms are equal.
\end{theorem}
\begin{proof} The algorithm produces
$$
q_0=q|M \quad \text{with} \quad z_{q_0} = M^{-1}z_q \in \Fd,
$$
and it is easy to check that $q_0$ is reduced if and only if $\ze_{q_0} \in \Fd$, as in  \cite[Sect. 5.3.1]{coh93} for example.

It is clear that two forms must be equivalent if their reduced forms are equal. Conversely, the first roots of two equivalent forms lie in the same $\G$ orbit by \e{roots}. This orbit intersects $\Fd$ exactly once, meaning their reduced forms must be the same.
\end{proof}

%*************************************
% d -20
\SpecialCoor
\psset{griddots=5,subgriddiv=0,gridlabels=0pt}
\psset{xunit=0.32cm, yunit=0.32cm, runit=0.32cm}
%\psset{xunit=1cm, yunit=1cm, runit=1cm}
\psset{linewidth=1pt}
\psset{dotsize=7pt 0,dotstyle=*}
\begin{figure}[ht]
\centering
 \begin{pspicture}(-10,-10)(10,12) %\psgrid

%\psline(0,0)(10,0)(10,5)(0,5)(0,0)
%\psset{arrowscale=2,arrowinset=0.5}
\psset{arrowscale=2.4,arrowinset=0.3,arrowlength=1.1}
\newrgbcolor{light}{0.8 0.8 1.0}
%\newrgbcolor{light}{0 0 1}
%\newrgbcolor{blue2}{0.4 0.4 1.0}
\newrgbcolor{pale}{1 0.7 1}
\newrgbcolor{pale}{1 0.85 0.65}
\newrgbcolor{light}{0.87 0.57 0.77}

%\pscurve[linewidth=13pt,linecolor=pale](0,0)(-2.59808, -1.5)(-5.57241, -1.10842)(-8.34405, -2.25647)(-10.1703, -4.63653)
\pscurve[linewidth=7pt,linecolor=pale](0,0.5)(-2.59808, -1.2)(-5.57241, -1.10842)(-8.34405, -2.25647)(-10.1703, -4.63653)

%\pscurve[linewidth=5pt,linecolor=light](-2.59808, -1.5)(0,0)(0., 3.)(1.82628, 5.38006)(4.59792, 6.52811)(7.57226, 6.13653)
\pscurve[linewidth=5pt,linecolor=light](-2.59808, -1.7)(0,-0.2)(0., 3.)(1.82628, 5.38006)(4.59792, 6.52811)(7.57226, 6.13653)
\pscurve[linewidth=5pt,linecolor=light](1.82628, 5.38006)(4.59792, 6.52811)(6.42421, 8.90817)

%0
\psline(0,0)(0., 3.)(-1.82628, 5.38006)(-4.59792, 6.52811)(-7.57226, 6.13653)
\psline(-4.59792, 6.52811)(-6.42421, 8.90817)
\psline(-1.82628, 5.38006)(-2.21786, 8.35439)(-4.04415, 10.7345)
\psline(-2.21786, 8.35439)(-1.06981, 11.126)
\psline(0., 3.)(1.82628, 5.38006)(4.59792, 6.52811)(7.57226, 6.13653)
\psline(4.59792, 6.52811)(6.42421, 8.90817)
\psline(1.82628, 5.38006)(2.21786, 8.35439)(4.04415, 10.7345)
\psline(2.21786, 8.35439)(1.06981, 11.126)

\rput(4.41673, 2.55){$3$}
\rput(0., 7.5){$10$}
\rput(-4.56571, 8.95015){$23$}
\rput(-7.36956, 7.67616){$_{42}$}
\rput(-2.60944, 11.3287){$_{63}$}
\rput(4.56571, 8.95015){$23$}
\rput(2.60944, 11.3287){$_{63}$}
\rput(7.36956, 7.67616){$_{42}$}

%1
\psline(0,0)(-2.59808, -1.5)(-3.74613, -4.27164)(-3.35455, -7.24597)(-1.52826, -9.62603)
\psline(-3.35455, -7.24597)(-4.5026, -10.0176)
\psline(-3.74613, -4.27164)(-6.12619, -6.09792)(-7.27424, -8.86956)
\psline(-6.12619, -6.09792)(-9.10052, -6.4895)
\psline(-2.59808, -1.5)(-5.57241, -1.10842)(-7.95247, 0.717863)(-9.10052, 3.4895)
\psline(-7.95247, 0.717863)(-10.9268, 1.10944)
\psline(-5.57241, -1.10842)(-8.34405, -2.25647)(-11.3184, -1.86489)
\psline(-8.34405, -2.25647)(-10.1703, -4.63653)

\rput(-4.41673, 2.55){$3$}
\rput(-6.49519, -3.75){$7$}
\rput(-5.4682, -8.4291){$15$}
\rput(-2.96297, -10.2203){$_{27}$}
\rput(-8.50625, -7.92421){$_{42}$}
\rput(-10.0339, -0.521054){$18$}
\rput(-11.1157, -3.40452){$_{47}$}
\rput(-10.3325, 2.54415){$_{35}$}

%2
\psline(0,0)(2.59808, -1.5)(5.57241, -1.10842)(7.95247, 0.717863)(9.10052, 3.4895)
\psline(7.95247, 0.717863)(10.9268, 1.10944)
\psline(5.57241, -1.10842)(8.34405, -2.25647)(11.3184, -1.86489)
\psline(8.34405, -2.25647)(10.1703, -4.63653)
\psline(2.59808, -1.5)(3.74613, -4.27164)(3.35455, -7.24597)(1.52826, -9.62603)
\psline(3.35455, -7.24597)(4.5026, -10.0176)
\psline(3.74613, -4.27164)(6.12619, -6.09792)(7.27424, -8.86956)
\psline(6.12619, -6.09792)(9.10052, -6.4895)

\rput(0., -5.1){$2$}
\rput(6.49519, -3.75){$7$}
\rput(10.0339, -0.521054){$18$}
\rput(10.3325, 2.54415){$_{35}$}
\rput(11.1157, -3.40452){$_{47}$}
\rput(5.4682, -8.4291){$15$}
\rput(8.50625, -7.92421){$_{42}$}
\rput(2.96297, -10.2203){$_{27}$}

\rput(0., -2){\small{reduced}}

\psline[ArrowInside=->,ArrowInsidePos=0.65](-10.1703, -4.63653)(-8.34405, -2.25647)
\psline[ArrowInside=->,ArrowInsidePos=0.65](4.59792, 6.52811)(7.57226, 6.13653)
\psline[ArrowInside=->,ArrowInsidePos=0.65](0,0)(-2.59808, -1.5)

\psdot[linecolor=red](0,0)

\end{pspicture}
\caption{Reduction on a  topograph of discriminant $D=-20$}
\label{-20}
\end{figure}
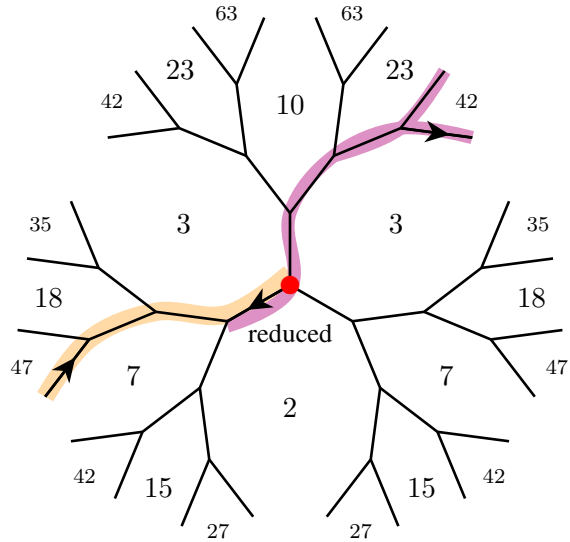
%*************************************

It follows from Theorem \ref{negred} that every topograph of negative discriminant contains a unique reduced form. The next result links it to the wells of Section \ref{+}.

\begin{prop} \la{towell}
Let $\mathcal T$ be a topograph of discriminant $D<0$ with reduced form $[a,b,c]$. If $\mathcal T$ has an edge well then it coincides with $[a,b,c]$. If $\mathcal T$ has a vertex well then, of the three edges incident to it, $[a,b,c]$ coincides with an edge with label of minimal absolute value.
\end{prop}
\begin{proof}
If $\mathcal T$ has an edge well with adjacent regions $a \lqs c$ then $[a,0,c]$ is reduced. Also if the reduced form $[a,b,c]$ has $b=0$ then this  coincides with the unique edge well.

In the remaining case $\mathcal T$ has a vertex well and $b \neq 0$. If $b>0$, let $v$ be the vertex that $[a,b,c]$ is directed away from. The labels of all the edges directed away from $v$ are $b$, $2a-b$ and $2c-b$, as in Figure \ref{sl2fig}. Since $[a,b,c]$ is reduced
\begin{equation*}
  2a-b\gqs 2b-b=b>0, \qquad 2c-b\gqs 2b-b=b>0.
\end{equation*}
Therefore $v$ is the vertex well and $b$ is minimal. The case $b<0$ is similar, with  the vertex that $[a,b,c]$ is directed to being the vertex well.
\end{proof}

With the climbing lemma it may also be seen that the edge label $b$ of the reduced form is minimal in absolute value among all edge labels of $\mathcal T$. Further,
the smallest region labels of $\mathcal T$ are $a\lqs c \lqs a+c-b$.

Figure \ref{well} shows an example with the reduced form $[2,1,4]$ incident to the well. See also   the reduced form $[2,2,3]$ in Figure \ref{-20}. To apply the reduction algorithm to the form $q=[47,-36,7]$, indicated with an arrow on the left of Figure \ref{-20}, we find  $z_q=\frac{36+\sqrt{-20}}{94}=\langle 0,2,1+z_0 \rangle$ as a general continued fraction, giving the path $L^0R^2L^1$ and requiring a final $S$. This path is highlighted in the figure.

For the form $q=[42,22,3]$ indicated  on the right, we have $z_q=\frac{-22+\sqrt{-20}}{84}=\langle -1,1,2,1+z_0 \rangle$  giving the path $L^{-1}R^1L^2R^1$. As $q$ was directed away from the well, the initial $L^{a_0}$ with $a_0<0$ moves in reverse, (see the right of Figure \ref{sl2fig}), going one edge outside the optimal path so that the succeeding turns may move forward directly to the reduced form.

 The reduction in this section gives a precise  version of the geometric reduction outlined by Buell in \cite[p. 18]{bue89}.

\subsection{Square discriminants} \la{sqd}

Now topographs have discriminant $D=m^2$ for $m \gqs 0$. Their properties are detailed in Sections \ref{0+}, \ref{opm}.

\begin{prop} \label{which-lake}
Let $q$ be a form on a topograph with  discriminant $D=m^2>0$. Then the path from $q$ associated with $\ze_q$ leads to the right lake. Precisely, this path contains only one edge on the right lake  and this edge must be at the end of the path. Likewise, the path from $q$ associated with the second root $\ze'_q$ leads to the left lake, containing only one edge on the left lake  which must be at its end. The same is true when $D=m=0$ and $q\neq[0,0,0]$:  then $\ze_q=\ze'_q$ and there is only one lake.
\end{prop}
\begin{proof}
Notice that the lake edges of the right lake have labels $m$ directed clockwise around the lake. The  left lake edges have labels $m$ directed counter clockwise. This follows from our choice of having the positive regions above the river; see Figure \ref{red 18}.
For the trivial cases of $\ze_q=\infty$ or $0$ we must have $q$ already on the right lake. In the same way, $q$ is on the left lake if $\ze'_q=\infty$ or $0$.

For $\ze_q \neq \infty, 0$, as seen in Lemma \ref{perf}, the final configuration on the path from $q$ associated with $\ze_q$ is $[a,b,0]$ if it finishes with an $L$. We must have $b=\pm m$ and claim that in fact $b=m$. For this, write $b=\delta m$ and suppose that $q_0$ is the next to last configuration on the path. Then, with \e{uu},
\begin{equation*}
  q_0|L= [a,\delta m,0] \quad \implies  \quad q_0=[a,-2 a+\delta m, a-\delta m].
\end{equation*}
When $a\neq 0$ then $\ze_{q_0} = 1+m(1-\delta)/(2a)$. When $a=0$ then $\ze_{q_0} =\infty$ if $\delta=-1$ and $\ze_{q_0} = 1$ if $\delta=1$. However, the path associated to $q_0$ must match the path from $q$ by Lemma \ref{cont} and hence we require $\ze_{q_0} = 1$ and $\delta=1$, proving the claim.

So the path from $q$ associated with $\ze_q$ has final configuration $[a,m,0]$, for some $a$, if it finishes with a left turn.
In the same way, the final configuration is $[0,-m,c]$, for some $c$, if it finishes with a right turn. Also the same is true for the path from $q$ associated with the second root $\ze'_q$ except that  $m$ and $-m$ are switched.

 Consequently, the final configuration on the path from $q$ associated with $\ze_q$ must be on the right lake, with the next to last path edge not on the lake. It follows that no other path edges can be on the right lake since that would create a circuit in the tree. Similarly for the other cases.
\end{proof}

 For definiteness we may force the continued fractions to be $\langle a_0,a_1, \dots, a_n \rangle$ for $n$ odd. This can be done since, if $n$ is even, $\langle \dots, a_{n-1}, a_n \rangle$ equals $\langle \dots, a_{n-1}+1 \rangle$ if $a_n=1$ and equals $\langle \dots, a_{n-1},a_n-1,1 \rangle$  if $a_n\gqs 2$. This makes the associated paths finish with an $R$. In this way, since paths consist of forward left and right turns on a tree, (the first turn could involve going backwards), the properties found in Proposition \ref{which-lake} completely determine the paths associated with each root.  The proposition may also be used to give a simple reduction algorithm.

\begin{adef} \la{defx}
{\rm A form of discriminant $D=m^2>0$ is {\em reduced} if it equals $[0,m,c]$ for $0<c \lqs m$. The reduced forms of discriminant $D=0$ equal $[0,0,c]$.
}
\end{adef}

\begin{algo}
{\rm Starting with any form $q$ of discriminant $D=m^2$,   take the path associated with its first root $\ze_q$ leading to $q_1=q|M_1$ on the right lake. Then take the path from $q_1$ associated with its second root $\ze'_{q_1}$ to $q_2=q_1|M_2$, ending with a right turn on the left lake. If $q_2=[0,m,0]$ then  replace $M_2$ by $M_2L$ to get $q_2=[0,m,m]$. The output is $q_2$ and $M=M_1M_2$. (The form $q=[0,0,0]$ is already reduced so we ignore it.)
}
\end{algo}

\begin{theorem} \la{redsq}
The form produced by the above algorithm is  reduced, showing that every form of square discriminant is explicitly equivalent to a reduced form. Two forms  are equivalent if and only if their reduced forms are equal. %So reduction unique...
\end{theorem}
\begin{proof}
Suppose $q$ is a form on a topograph with discriminant $D>0$ and nonadjacent lakes. The algorithm's second path from the right lake to the left one must pass along the river. The river meets the left lake at a vertex with surrounding regions labelled $0$, $c$, $c-m$ with $c>0$ above the river and $c-m<0$ below it. Hence $0<c<m$. By Proposition \ref{which-lake}, the  algorithm's output is $[0,m,c]$.

If the lakes are adjacent, the algorithm gives a reduced form from $[0,m,0]$  with an extra $L$ turn -- see the left of Figure \ref{1 4} for example. If $D=0$ then the path ending with a right turn on the single lake gives $[0,0,c]$.  We have shown that the reduction algorithm always produces a reduced form.

 If two forms are equivalent then they appear on the same topograph and their reduced forms are the same configuration on the left lake, (or on the single lake if $D=0$). Two forms with equal reduced forms must be equivalent.
\end{proof}

%*************************************
% red two lakes
\SpecialCoor
\psset{griddots=5,subgriddiv=0,gridlabels=0pt}
\psset{xunit=0.2cm, yunit=0.2cm, runit=0.2cm}
%\psset{xunit=1cm, yunit=1cm, runit=1cm}
\psset{linewidth=1pt}
\psset{dotsize=5pt 0,dotstyle=*}
\begin{figure}[ht]
\centering
\begin{pspicture}(3,4)(54,33) %\psgrid

%\psline(0,0)(10,0)(10,5)(0,5)(0,0)
%\psset{arrowscale=2,arrowinset=0.5}
\psset{arrowscale=2.4,arrowinset=0.3,arrowlength=1.1}
\newrgbcolor{light}{0.9 0.9 1.0}
\newrgbcolor{light}{0.8 0.7 1.0}
%\newrgbcolor{light}{0 0 1}
%\newrgbcolor{blue2}{0.4 0.4 1.0}
\newrgbcolor{blue2}{0.1 0.2 0.8}
\newrgbcolor{blue2}{0.3 0.3 0.8}
\newrgbcolor{pale}{1 0.85 0.65}
\newrgbcolor{light2}{0.87 0.57 0.77}

%left lake
\psline[linecolor=light,fillstyle=solid,fillcolor=light](3,7.3)(8.4,9.4)(11,14)(7.7,17.5)(3,18.8)

%right lake
\psline[linecolor=light,fillstyle=solid,fillcolor=light](54,23)(49,22.3)(44,20.5)(40,17.2)(43.2,13)(48.5,10.3)(54,9.3)

\pscurve[linewidth=7pt,linecolor=light2](39.5,27.9)(36.2,22.9)(34.5,18.5)(40.2,16.2)(42.8,12.4)%(43.2,12)

%\pscurve[linewidth=7pt,linecolor=pale](44,20.5)(40,17.2)(43.2,13)
\pscurve[linewidth=7pt,linecolor=pale](44.5,21)(41,17.7)(43.7,13.5)
\pscurve[linewidth=7pt,linecolor=pale](44,21)(40.2,17.7)(34.5,19)(30,17.9)(25.4,15.2)(20.7,16.8)(16,14.5)(10.8,14.7)(7.7,18.2)

\psline(54,23)(49,22.3)(44,20.5)(40,17.2)(43.2,13)(48.5,10.3)(54,9.3)

\psline(3,7.3)(8.4,9.4)(11,14)(7.7,17.5)(3,18.8)

%top trees

\psline(10,21.3)(7.7,17.5)

\psline(13.8,26.8)(18.1,24.5)(20.5,20.5)(20.7,16.3)
\psline(18,29.2)(18.1,24.5)
\psline(22,29.5)(23,25)(20.5,20.5)
\psline(27.7,28.6)(23,25)

\psline(28.7,22)(30,17.4)

\psline(34.2,27.5)(36.2,22.9)(34.5,18.5)
\psline(39.5,27.9)(36.2,22.9)

\psline(42.5,24.8)(44,20.5)

\psline(50.8,31.3)(48.3,27)(49,22.3)
\psline(45.8,30.7)(48.3,27)

\psline[ArrowInside=->,ArrowInsidePos=0.65](39.5,27.9)(36.2,22.9)
\psline[ArrowInside=->,ArrowInsidePos=0.65](40,17.2)(43.2,13)
\psline[ArrowInside=->,ArrowInsidePos=0.65](11,14)(7.7,17.5)
\psline[ArrowInside=->,ArrowInsidePos=0.65](11.2,6.3)(8.4,9.4)

%bottom trees
\psline(8.4,9.4)(11.2,6.3)

\psline(16,14)(17.7,9.6)(16.5,4.6)
\psline(21.2,5)(17.7,9.6)

\psline(25.4,14.7)(27.5,10)(25,4.5)
\psline(31.7,5.6)(27.5,10)

\psline(43.2,13)(41,8.4)

\psline(48.5,10.3)(47.3,5.5)

%river
\psline[linewidth=3.5pt,linecolor=blue2](10.8,14)(16,14)(20.7,16.3)(25.4,14.7)(30,17.4)(34.5,18.5)(40.2,17.2)

%\psline[linewidth=2.4pt]{->}(11,14)(13.5,14)
%\psline[linewidth=2.4pt]{->}(16,14)(20.7,16.3)

\rput(6,13.5){$0$}

\rput(6,21){$25$}
\rput(14.5,20){$7$}  %\rput(14.5,19.2){$7$}
\rput(13,10.5){$-11$}
\rput(7,6){$-29$}

\rput(15.3,28.7){$_{85}$}
\rput(20.3,26.2){$40$}
\rput(25,29.8){$_{91}$}

\rput(25,19.3){$9$}
\rput(21,12){$-8$}
\rput(18.7,4.5){$_{-45}$}

\rput(32.5,21.7){$16$}
\rput(34,12.6){$-5$}
\rput(28,5.3){$_{-35}$}

\rput(37,28.1){$63$}
\rput(39.5,21.2){$13$}
\rput(46.1,24){$31$}
\rput(52.1,25.8){$49$}
\rput(48.4,31.5){$_{160}$}

\rput(44.5,9){$-23$}
\rput(51,6){$-41$}

\rput(48,17){$0$}

%\psline[ArrowInside=->, ArrowInsidePos=0.5](0,0)(9,9)

\rput(12.8, 16.8){\small{reduced}}

\end{pspicture}
\caption{Reduction on a topograph of discriminant $D=18^2$}
\label{red 18}
\end{figure}
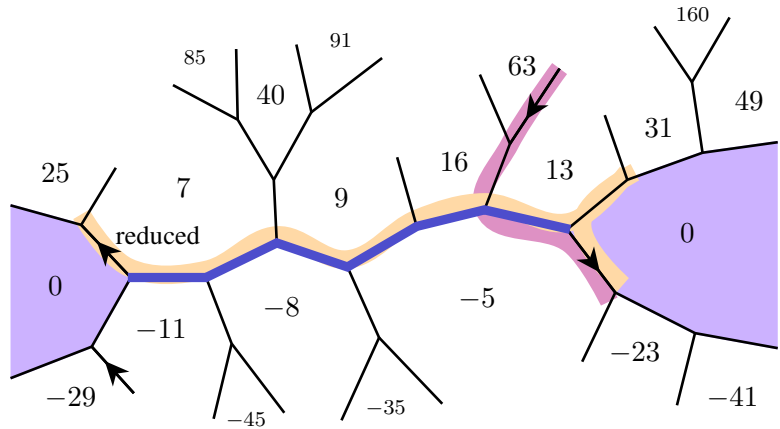
%*************************************

Figure \ref{red 18} shows the reduction algorithm applied to $q=[13,-60,63]$. Then $z_q=3=\langle 2,1 \rangle$, giving a path to $q_1=[0,-18,-5]$. Next, with \e{a=000}, $z_{q_1}'= -5/18=\langle -1,1,2,1,1,2 \rangle$  leading from $q_1$ to the reduced form $[0,18,7]$. It can be seen that the reduction path is a little inefficient. We cannot just use the second root $z_q'$ since in some cases, such as the indicated $[-29,40,-11]$, the path associated to the second root ends up at a  form on the left lake that is not reduced.

Since the lake edge labels are $\pm m$, the  labels of the regions adjacent to each lake must be in arithmetic progression with difference $m$. The progressions on each lake are related as follows.

\begin{prop} \label{prok}
Suppose a topograph is primitive with discriminant  $D=m^2$ for $m\gqs 1$. %The regions adjacent to the two lakes have labels in arithmetic progressions.
There exist $r$ and $s$ with $r s \equiv 1 \bmod m$ so that the region labels at the left lake are $\equiv r \bmod m$ while those at the right are $\equiv s \bmod m$.
\end{prop}
\begin{proof}
Let $q=[r,-m,0]$ for $r> 0$ correspond to a lake edge on the left lake. Then $\ze_q=m/r$ in lowest terms and with \e{meven} the associated matrix is $M=(\begin{smallmatrix}h_{n-1} & m\\ k_{n-1} & r \end{smallmatrix})$, giving the path to the right lake,  where we may assume $n$ is even and we finish with an $L$. Hence $q|M=[s,m,0]$ by Proposition \ref{which-lake} and
\begin{equation*}
  r(h_{n-1} x+ m y)^2 \equiv s x^2 \bmod m \quad \implies  \quad  r h_{n-1}^2 \equiv s \bmod m.
\end{equation*}
Also $\det M =1$ implies $r  h_{n-1} \equiv 1 \bmod m$. Therefore $h_{n-1} \equiv s \bmod m$ and so $r s \equiv 1 \bmod m$.
\end{proof}

\subsection{Positive non-square discriminants} \la{pnsd}
Let $D>0$ be a non-square discriminant in this section. For a form $q$ of this discriminant, the components of the matrices in \e{assoc} associated with $\ze_q$ and $\ze'_q$ may be described as follows:
\begin{align} \la{mml}
  M_L=M_L(q) := L^k = \begin{pmatrix} 1 & k \\ 0 & 1 \end{pmatrix} \quad &\text{for} \quad k=\lfloor \ze_q \rfloor = \left\lfloor \frac{-b+\sqrt{D}}{2a}\right\rfloor,\\
  \la{mmr}
  M_R=M_R(q) := R^k = \begin{pmatrix} 1 & 0 \\ k & 1 \end{pmatrix}\quad &\text{for} \quad k=\lfloor 1/\ze_q \rfloor = \left\lfloor -\frac{b+\sqrt{D}}{2c}\right\rfloor.
\end{align}
Similarly $M_L'$ and $M_R'$ are defined with the second root $\ze'_q$. By Lemma \ref{cont} we obtain for example
\begin{equation*}
  M=L^{a_0} R^{a_1} L^{a_2} = M_L(q) \cdot  M_R(q_1)  \cdot M_L(q_2)
\end{equation*}
with $q_1=q|M_L(q)$ and $q_2=q_1|M_R(q_1)$ for an initial matrix associated with $\ze_q$.

Recall Section \ref{+-}; topographs here have infinite rivers with finite periods.
The analog of Proposition \ref{which-lake} is the next result.

\begin{prop} \label{riverpath}
Let $q$ be a form on a topograph with non-square discriminant $D>0$. Then the path from $q$ associated with $\ze_q$ leads to the river and then follows it forever going rightwards. The path  associated with the second root $\ze'_q$ also leads to the river, following it forever going leftwards.% (upstream).
\end{prop}
\begin{proof}
As $\ze_q$ has an infinite
 continued fraction expansion, consider its convergent
  \begin{equation*}
    \langle a_0,a_1, \dots, a_n\rangle = h_n/k_n
  \end{equation*}
and assume $n$ is odd. Recall from \cite[Chap. 10]{hawr} that $h_{2m}/k_{2m}<\ze_q<h_{2m+1}/k_{2m+1}$ and $h_n/k_n = \ze_q +\varepsilon_n$ with $0<\varepsilon_n<1/k_n^2$ for $k_n$ strictly increasing with  $n$.
Let $M=L^{a_0}R^{a_1} \cdots R^{a_n}$. Following this path of left and right turns from $q$ leads to $q'=q|M$ with, using \e{cf},
\begin{align}
  q'(1,0)=q(h_n,k_n) & =k_n^2 \cdot q(\ze_q + \varepsilon_n,1) \notag \\
  & =k_n^2 \Bigl(a(\ze_q + \varepsilon_n)^2+b(\ze_q + \varepsilon_n)+c\bigr) \notag \\
  & =k_n^2 \Bigl(a\ze_q^2 + b\ze_q +c+\varepsilon_n(2a \ze_q+b+a\varepsilon_n)\bigr) \notag \\
  & =k_n^2 \varepsilon_n(2a \ze_q+b+a\varepsilon_n) \notag \\
  & =k_n^2 \varepsilon_n (\sqrt{D}+ a \varepsilon_n). \label{riv}
\end{align}
Similarly,
\begin{equation}\label{riv2}
  q'(0,1)=q(h_{n-1},k_{n-1})=k_{n-1}^2 \varepsilon_{n-1} (\sqrt{D}+ a  \varepsilon_{n-1}),
\end{equation}
with $-1/k_{n-1}^2<\varepsilon_{n-1}<0$.
Therefore $q'=[a',b',c']$ has $a'>0>c'$ for all $n$ large enough by \e{riv} and \e{riv2}.
So the path from $q$ associated with $\ze_q$ reaches the river and follows it with configurations oriented  rightwards. As this path only makes forward left and right turns, it must stay on the river moving rightwards.

In the same way, the path associated with $\ze'_q$ leads to leftward oriented configurations on the river since $\sqrt{D}$ in \e{riv}, \e{riv2} becomes $-\sqrt{D}$.
\end{proof}

The above result is stated in \cite[Prop. 3]{sv18}. The periodicity of the river reflects the eventual periodicity of the continued fractions of the quadratic irrationals $\ze_q$ and $\ze'_q$.
It is also seen with \e{riv} and \e{riv2} that regions adjacent to these paths %associated to $\ze_q$ and $\ze'_q$
have bounded labels. By the climbing lemma, every path of forward left and right turns not  on the river must have unbounded adjacent regions. This is quantified in \cite{sv18}.

We obtain an easy kind of reduction. The following definition is from  \cite[Sect. 1.1]{cz93} where these forms were used to construct rational period functions.

\begin{adef}
{\rm A form $[a,b,c]$  is {\em simple} if $a>0>c$. These are exactly the rightward directed river edges.
}
\end{adef}

\begin{theorem}
Every form of positive non-square discriminant is explicitly equivalent to a simple form. Two  forms are equivalent if and only if the simple forms they correspond to are on the same river.
\end{theorem}
\begin{proof}
Proposition \ref{riverpath} shows that  for any form $q$, the path associated with its first root $\ze_q$ leads to a simple form. Continuing this path, edge by edge, leads to all the simple forms in this equivalence class since these are just the rightward directed  edges of this topograph's periodic river.
\end{proof}

A natural alternative, giving fewer representatives, is just to use the simple forms that appear at the end of  sequences of $L$s and sequences of $R$s. These inflection points are the `river bend' configurations of \cite{msw19}.

%*************************************
% river turns
\SpecialCoor
\psset{griddots=5,subgriddiv=0,gridlabels=0pt}
\psset{xunit=0.38cm, yunit=0.3cm, runit=0.2cm}
%\psset{xunit=1cm, yunit=1cm, runit=1cm}
\psset{linewidth=1pt}
\psset{dotsize=5pt 0,dotstyle=*}
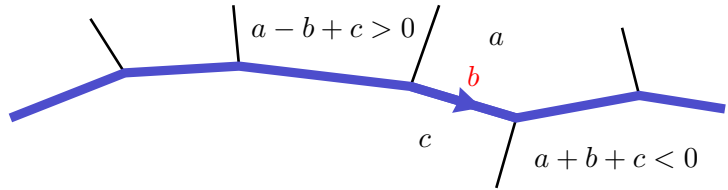
\begin{figure}[ht]
\centering
\begin{pspicture}(0,1.3)(26,9.8) %\psgrid

%\psline(0,0)(10,0)(10,5)(0,5)(0,0)
%\psset{arrowscale=2,arrowinset=0.5}
\psset{arrowscale=1.2,arrowinset=0.1,arrowlength=1.0}
\newrgbcolor{light}{0.9 0.9 1.0}
%\newrgbcolor{light}{0 0 1}
%\newrgbcolor{blue2}{0.4 0.4 1.0}
\newrgbcolor{blue2}{0.3 0.3 0.8}

\psline(2.8,8.4)(4,6)
\psline(7.8,9)(8,6.3)
\psline(15,9)(14,5.4)
\psline(17,0.9)(17.7,4)
\psline(21.4,8)(22,5)

%river
\psline[linewidth=3.5pt,linecolor=blue2](0,4)(4,6)(8,6.3)(14,5.4)(17.7,4)(22,5)(25,4.4)

\psline[linewidth=3.5pt,linecolor=blue2,ArrowInside=->,ArrowInsidePos=0.6](14,5.4)(17.7,4)

\rput(16.2,5.87){\textcolor{red}{$b$}}
\rput(11.3,8){$a-b+c>0$}
\rput(17,7.5){$a$}
\rput(14.5,3){$c$}
\rput(21.2,2.2){$a+b+c<0$}

\end{pspicture}
\caption{Along the river}
\label{alo}
\end{figure}
%*************************************

\begin{adef} \la{sr}
{\rm A form $[a,b,c]$  is {\em simply reduced} if it is simple, ($a>0>c$), and has $|a+c|<|b|$.
}
\end{adef}

If a form $q=[a,b,c]$ appears at the end of a sequence of $R$ turns, moving rightwards along the river with the next turn $L$, then we see from Figure \ref{alo} that $|a+c|<-b$. To get to the end of the next $L$ sequence apply $M_L$ from \e{mml}.
At the end of a sequence of $L$ turns, a form $q=[a,b,c]$ must similarly have $|a+c|<b$. To get to the end of the next $R$ sequence apply
$M_R$ from \e{mmr}.
Applying alternately $M_L$ and  $M_R$ repeatedly  produces a cycle of simply reduced forms on the river.

\begin{algo}
{\rm Start with any form $q$  of positive non-square discriminant. Repeatedly apply $M_L$ and $M_R$ alternately: $q \to q|M_L$, $q \to q|M_R$, $q \to q|M_L$, etc. Finish when this sequence becomes periodic, and the output is a cycle of simply reduced forms.
}
\end{algo}

The following result now follows from Proposition \ref{riverpath}.

\begin{theorem} \label{simp}
Every form of non-square discriminant $D>0$ is explicitly equivalent to a cycle of simply reduced forms. Two  forms are equivalent if and only if their  cycles of simply reduced forms agree.
\end{theorem}

%*************************************
% long river
\SpecialCoor
\psset{griddots=5,subgriddiv=0,gridlabels=0pt}
\psset{xunit=0.23cm, yunit=0.23cm, runit=0.23cm}
%\psset{xunit=1cm, yunit=1cm, runit=1cm}
\psset{linewidth=1pt}
\psset{dotsize=5pt 0,dotstyle=*}
\begin{figure}[ht]
\centering
\begin{pspicture}(2,0)(61,12) %\psgrid

%\psline(0,0)(10,0)(10,5)(0,5)(0,0)
%\psset{arrowscale=2,arrowinset=0.5}
\psset{arrowscale=1.2,arrowinset=0.1,arrowlength=1.0}
\newrgbcolor{light}{0.9 0.9 1.0}
%\newrgbcolor{light}{0 0 1}
%\newrgbcolor{blue2}{0.4 0.4 1.0}
\newrgbcolor{blue2}{0.3 0.3 0.8}
\newrgbcolor{red2}{1 0.55 0.65}

\psline(5.3,1)(6,4)
\psline(11,1)(10,4)
\psline(10.5,10)(13.5,8)(14,5)
\psline(15,11)(13.5,8)
\psline(17.2,11)(18.5,8)(18,5)
\psline(21,10.2)(18.5,8)
\psline(22,1)(22,4)

\psline(23,10)(25.2,7.6)(26,5)
\psline(26.2,11)(25.2,7.6)
\psline(28,11)(29.7,8.3)(30,5.5)
\psline(30.7,11.2)(29.7,8.3)

\psline(33.1,11)(34.6,8.2)(34,5.5)
\psline(37,11)(34.6,8.2)
\psline(38.5,10.9)(39.2,7.8)(38,5)
\psline(41.8,10)(39.2,7.8)

\psline(41,0.7)(42,4)
\psline(46,0.3)(46,3.5)
\psline(50.7,0.6)(50,4)

\psline(50,10.3)(53,8.8)(53.5,5.5)
\psline(55,11.2)(53,8.8)
\psline(56.3,1.5)(57,5)

%river
\psline[linewidth=2.5pt,linecolor=blue2](2,5)(6,4)(10,4)(14,5)(18,5)(22,4)(26,5)(30,5.5)
(34,5.5)(38,5)(42,4)(46,3.5)(50,4)(53.5,5.5)(57,5)

%\psline[linewidth=2.5pt,linecolor=blue2,ArrowInside=->,ArrowInsidePos=0.6](2,5)(6,4)
\psline[linewidth=2.5pt,linecolor=blue2,linestyle=dashed](57,5)(61,4)

\psline[linewidth=2.5pt,linecolor=red2,ArrowInside=->,ArrowInsidePos=0.6](10,4)(14,5)
\psline[linewidth=2.5pt,linecolor=red2,ArrowInside=->,ArrowInsidePos=0.6](18,5)(22,4)
\psline[linewidth=2.5pt,linecolor=red2,ArrowInside=->,ArrowInsidePos=0.6](22,4)(26,5)
\psline[linewidth=2.5pt,linecolor=red2,ArrowInside=->,ArrowInsidePos=0.6](38,5)(42,4)
\psline[linewidth=2.5pt,linecolor=red2,ArrowInside=->,ArrowInsidePos=0.6](50,4)(53.5,5.5)
\psline[linewidth=2.5pt,linecolor=red2,ArrowInside=->,ArrowInsidePos=0.6](53.5,5.5)(57,5)

\rput(8,7){\textcolor{gray}{$_{59}$}}
\rput(16.1,7.8){\textcolor{gray}{$_{181}$}}
\rput(22.0,7){\textcolor{gray}{$_{145}$}}
\rput(27.5,7.3){\textcolor{gray}{$_{271}$}}

\rput(32.1,7.9){\textcolor{gray}{$_{299}$}}
\rput(36.8,7.7){\textcolor{gray}{$_{229}$}}
\rput(46.3,6.5){\textcolor{gray}{$_{61}$}}
\rput(58,7.5){\textcolor{gray}{$_{59}$}}

\rput(2.6,2.3){\textcolor{gray}{$_{-245}$}}
\rput(8,2){\textcolor{gray}{$_{-221}$}}
\rput(16,2.5){\textcolor{gray}{$_{-79}$}}
\rput(32,2.7){\textcolor{gray}{$_{-49}$}}

\rput(43.8,1.9){\textcolor{gray}{$_{-205}$}}
\rput(48.2,1.7){\textcolor{gray}{$_{-239}$}}
\rput(53.4,3.1){\textcolor{gray}{$_{-151}$}}
\rput(58.8,2.2){\textcolor{gray}{$_{-245}$}}

\end{pspicture}
\caption{A cycle of six simply reduced forms in a river}
\label{smriv}
\end{figure}
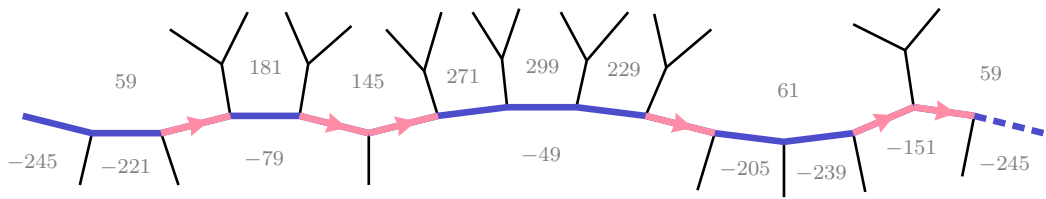
%*************************************

\section{Gauss and Zagier reduction} \la{red}
Discriminants in this section are positive and non-square. For forms with such discriminants, the main reduction methods in the literature are due to Gauss and Zagier. To help analyze them, we first examine the effects of $M_L$, $M_R$, $M_L'$ and $M_R'$ on a form $q=[a,b,c]$ in a topograph. Starting at $q$, let $P$ be the unique simple directed path to the river, meeting it at vertex $v$ and continuing rightwards along it for $M_L$, $M_R$ and leftwards along it for $M_L'$ and $M_R'$.

\begin{prop} \la{four}
With the above notation:
\begin{enumerate}
  \item If $q$ has the same direction as $P$ then $M_L$ and $M_L'$ move forward and left, (repeatedly applying $L$), going as far as possible along the path $P$. See the left of Figure \ref{red ml}.
  \item If $q$ is directed against $P$ then $M_L$ and $M_L'$ move  backwards and left, (repeatedly applying $L^{-1}$), going one edge outside the path $P$. See the right of Figure \ref{red ml}.
  \item If $q$ has the same direction as $P$ then $M_R$ and $M_R'$ move  forward and right, (repeatedly applying $R$), going as far as possible along  $P$.
  \item If $q$ is directed against $P$ then $M_R$ and $M_R'$ move  backwards and right, (repeatedly applying $R^{-1}$), going one edge outside $P$.
\end{enumerate}
\end{prop}
\begin{proof}
Proposition \ref{riverpath} shows we are in a highly constrained situation. In part (i) above, label $q$ as $q_1$. Then $M_L$ and $M_L'$ correspond to the first $k_1 =\lfloor \frac{-b\pm\sqrt{D}}{2a}\rfloor \gqs 0$ left turns on the path $P$ associated with each root of $q$. They must go as far as possible along $P$ so that the next right turns from $R^{a_1}$ remain on the path. Similarly in part (ii), label $q$ as $q_2$. The only way the succeeding right turns from $R^{a_1}$ remain on the path is for $M_L$ and $M_L'$ to reverse $|k_2|$ places to an edge just off $P$ for $k_2=\lfloor \frac{-b\pm\sqrt{D}}{2a}\rfloor \lqs -1$. Reversing any further would mean the following  turns would never meet the river.

For part (iii) we have $M_R(q_1)=M_R(q_2|S)=R^k$ with $q_2=[a,b,c]$ and by \e{mmr},
\begin{equation*}
  k=\left\lfloor -\frac{-b+\sqrt{D}}{2a}\right\rfloor = -\left\lceil \frac{-b+\sqrt{D}}{2a}\right\rceil = -\left\lfloor \frac{-b+\sqrt{D}}{2a}\right\rfloor-1 = -k_2-1.
\end{equation*}
The formula for $M_R'(q_1)$ is the same, with $-\sqrt{D}$ instead of $\sqrt{D}$. So  we move forward and right $|k_2|-1$ times, one less than in part (ii).  Part (iv) is similar, moving  backwards and right  $k_1+1$ times.
\end{proof}

%*************************************
% ml mr
\SpecialCoor
\psset{griddots=5,subgriddiv=0,gridlabels=0pt}
\psset{xunit=0.29cm, yunit=0.29cm, runit=0.2cm}
%\psset{xunit=1cm, yunit=1cm, runit=1cm}
\psset{linewidth=1pt}
\psset{dotsize=5pt 0,dotstyle=*}
\begin{figure}[ht]
\centering
\begin{pspicture}(0,0)(42,12) %\psgrid

%\psline(0,0)(10,0)(10,5)(0,5)(0,0)
%\psset{arrowscale=2,arrowinset=0.5}
\psset{arrowscale=2.4,arrowinset=0.3,arrowlength=1.1}
\newrgbcolor{light}{0.9 0.9 1.0}
\newrgbcolor{light}{0.8 0.7 1.0}
%\newrgbcolor{light}{0 0 1}
%\newrgbcolor{blue2}{0.4 0.4 1.0}
\newrgbcolor{blue2}{0.1 0.2 0.8}
\newrgbcolor{blue2}{0.3 0.3 0.8}
\newrgbcolor{pale}{1 0.85 0.65}
\newrgbcolor{light2}{0.87 0.57 0.77}

\rput(8,6){%
        \begin{pspicture}(0,0)(16,12)

\pscurve[linewidth=9pt,linecolor=pale](8,12)(5.5,10)(5.5,7)(8,5)(8,2)(12,1)(16,2)

\psline(8,2)(8,5)(5.5,7)(5.5,10)(3,12)
\psline(5.5,10)(8,12)
\psline(5.5,7)(2,8)
\psline(8,5)(10.5,7)

%river
\psline[linewidth=2.8pt,linecolor=blue2](0,2)(4,1)(8,2)(12,1)(16,2)

%\psline[linewidth=2.4pt]{->}(11,14)(13.5,14)
%\psline[linewidth=2.4pt]{->}(16,14)(20.7,16.3)

\psline[ArrowInside=->,ArrowInsidePos=0.69](8,12)(5.5,10)
\psline[linecolor=red,ArrowInside=->,ArrowInsidePos=0.69](5.5,7)(8,5)
\psline[linewidth=2.5pt,linecolor=red](5.5,7)(8,5)

\rput(8,0.5){$v$}
\rput(7.8,9.8){$q$}
\rput(4.9,4.7){$q|M_L$}
\rput(10.5,3.5){$P$}
\end{pspicture}}

%2
\rput(34,6){%
        \begin{pspicture}(0,0)(16,12)

\pscurve[linewidth=9pt,linecolor=pale](3,12)(5.5,10)(5.5,7)(8,5)(8,2)(12,1)(16,2)

\psline(8,2)(8,5)(5.5,7)(5.5,10)(3,12)
\psline(5.5,10)(8,12)
\psline(5.5,7)(2,8)
\psline(8,5)(10.5,7)

%river
\psline[linewidth=2.8pt,linecolor=blue2](0,2)(4,1)(8,2)(12,1)(16,2)

%\psline[linewidth=2.4pt]{->}(11,14)(13.5,14)
%\psline[linewidth=2.4pt]{->}(16,14)(20.7,16.3)

\psline[ArrowInside=->,ArrowInsidePos=0.69](5.5,10)(3,12)
\psline[linecolor=red,ArrowInside=->,ArrowInsidePos=0.69](2,8)(5.5,7)
\psline[linewidth=2.5pt,linecolor=red](2,8)(5.5,7)

\rput(8,0.5){$v$}
\rput(5.5,12){$q$}
\rput(2.8,6){$q|M_L$}
\rput(10.5,3.5){$P$}
\end{pspicture}}

\end{pspicture}
\caption{The effect of $M_L$}
\label{red ml}
\end{figure}
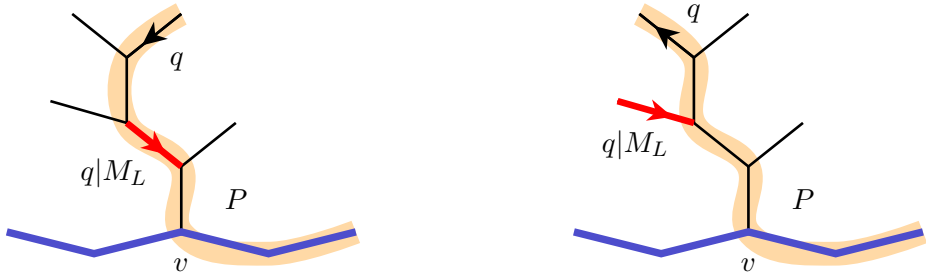
%*************************************

We are focussing on non-square $D>0$, but the effects of $M_L$, $M_R$, $M_L'$ and $M_R'$ are similar for square discriminants $D$ and analogously for $D<0$.

\subsection{Gauss reduction}
Gauss reduction can be described with a combination of $M_L$ and $M_R$ from \e{mml}, \e{mmr}. For $q=[a,b,c]$
put
\begin{align}
  \mg(q)  := M_R(q)S = S T^{-k} = \begin{pmatrix} 0 & -1 \\ 1 & -k \end{pmatrix} \quad &\text{for} \quad k= \left\lfloor -\frac{b+\sqrt{D}}{2c}\right\rfloor \quad \text{if} \quad c<0, \la{c-}\\
  \mg(q)  := S M_L(q|S) = S T^{k} = \begin{pmatrix} 0 & -1 \\ 1 & k \end{pmatrix} \quad &\text{for} \quad k= \left\lfloor \frac{b+\sqrt{D}}{2c}\right\rfloor \quad \text{if} \quad c>0. \la{c+}
\end{align}
Altogether
\begin{equation} \la{c}
   \mg(q) = S T^{k} = \begin{pmatrix} 0 & -1 \\ 1 & k \end{pmatrix} \quad \text{for} \quad k = \sgn(c)\left\lfloor \frac{b+\sqrt{D}}{2|c|}\right\rfloor.
\end{equation}
Then Gauss reduction is obtained by repeatedly applying the  map $q \mapsto q| \mg$. This is equivalent to the reduction described in
\cite[Sect. 3.1]{bue89} and \cite[Sect. 5.6.1]{coh93}. %, \cite{fla89}.
See also \cite[Sect. 3]{zey16} and \cite{smi18}.

Proposition \ref{four} allows us to see how the reduction works.  Let $P$ be the unique simple directed path from $q$ to the river, meeting it at vertex $v$ and then continuing  rightwards along it. Assume first that $q$ lies above the river. If $q$ has the opposite direction to $P$ then by \e{c+} and Proposition \ref{four} (i), we have that $q|\mg(q)$ is on $P$ and, after a sequence of left turns, matches its direction. If $q$ has the same direction as $P$ then, with Proposition \ref{four} (ii), backwards left turns are made towards the river going one edge beyond $P$ with $q|\mg(q)$ directed towards the river. The left of Figure \ref{exg} shows an example of this reduction process for  the form $[503851,442423,97121]$. If $q$ lies below the river we get the mirror image of the above, with edges in the opposite direction.

%*************************************
% gred long river
\SpecialCoor
\psset{griddots=5,subgriddiv=0,gridlabels=0pt}
\psset{xunit=0.23cm, yunit=0.23cm, runit=0.2cm}
%\psset{xunit=1cm, yunit=1cm, runit=1cm}
\psset{linewidth=1pt}
\psset{dotsize=5pt 0,dotstyle=*}
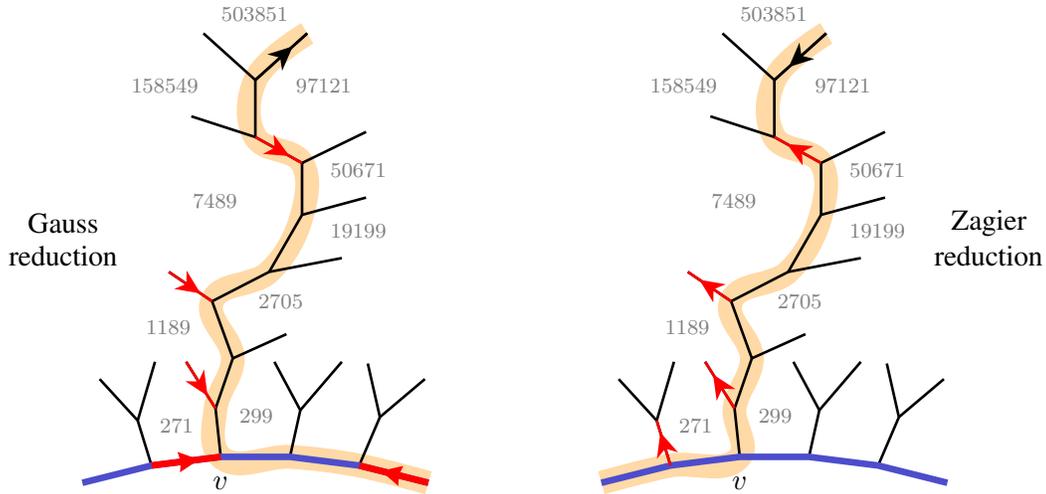
\begin{figure}[ht]
\centering
\begin{pspicture}(0,-1)(50,29) %\psgrid

%\psline(0,0)(10,0)(10,5)(0,5)(0,0)
%\psset{arrowscale=2,arrowinset=0.5}
\psset{arrowscale=2.4,arrowinset=0.3,arrowlength=1.1}
\newrgbcolor{light}{0.9 0.9 1.0}
%\newrgbcolor{light}{0 0 1}
%\newrgbcolor{blue2}{0.4 0.4 1.0}
\newrgbcolor{blue2}{0.3 0.3 0.8}
\newrgbcolor{pale}{1 0.85 0.65}

\rput(10,14){%
        \begin{pspicture}(0,0)(20,28) %\psgrid

\pscurve[linewidth=9pt,linecolor=pale](20,0)(16,1)(12,1.5)(8,1.5)
(7.7,4.3)(8.7,7.2)(7.5,10.5)(10.8,12.2)(12.7,15.5)
(12.7,18.5)(10,20)(10,23.3)(13,26)

\psline(1,6)(3.2,3.6)(4,1)
\psline(4.2,7)(3.2,3.6)
\psline(6,7)(7.7,4.3)(8,1.5)
\psline(8.7,7.2)(7.7,4.3)

\psline(11.1,7)(12.6,4.2)(12,1.5)
\psline(15,7)(12.6,4.2)
\psline(16.5,6.9)(17.2,3.8)(16,1)
\psline(19.8,6)(17.2,3.8)

\psline(11.8,8.6)(8.7,7.2)(7.5,10.5)
\psline(5,12.2)(7.5,10.5)(10.8,12.2)(15,13)
\psline(10.8,12.2)(12.7,15.5)(16.4,16.5)
\psline(12.7,15.5)(12.7,18.5)(16.4,20.3)
\psline(12.7,18.5)(10,20)(6.3,21.2)
\psline(10,20)(10,23.3)(7,26)
\psline(10,23.3)(13,26)

%river
\psline[linewidth=2.5pt,linecolor=blue2](0,0)(4,1)(8,1.5)
(12,1.5)(16,1)(20,0)

\psline[ArrowInside=->,ArrowInsidePos=0.69](10,23.3)(13,26)
\psline[linecolor=red,ArrowInside=->,ArrowInsidePos=0.69](10,20)(12.7,18.5)
\psline[linecolor=red,ArrowInside=->,ArrowInsidePos=0.69](5,12.2)(7.5,10.5)
\psline[linecolor=red,ArrowInside=->,ArrowInsidePos=0.69](6,7)(7.7,4.3)

\psset{arrowscale=1.2,arrowinset=0.1,arrowlength=1.0}
\psline[linewidth=2.5pt,linecolor=red,ArrowInside=->,ArrowInsidePos=0.6](4,1)(8,1.5)
\psline[linewidth=2.5pt,linecolor=red,ArrowInside=->,ArrowInsidePos=0.6](20,0)(16,1)

\rput(5.5,3.3){\textcolor{gray}{$_{271}$}}
\rput(10.1,3.9){\textcolor{gray}{$_{299}$}}
\rput(5,9){\textcolor{gray}{$_{1189}$}}
\rput(11.5,10.5){\textcolor{gray}{$_{2705}$}}
\rput(7.7,16.3){\textcolor{gray}{$_{7489}$}}
\rput(16,14.6){\textcolor{gray}{$_{19199}$}}
\rput(16,18.1){\textcolor{gray}{$_{50671}$}}
\rput(4.8,23){\textcolor{gray}{$_{158549}$}}
\rput(14,23){\textcolor{gray}{$_{97121}$}}
\rput(10,27.1){\textcolor{gray}{$_{503851}$}}

\rput(8,0){$v$}

\rput(-1.1,15){Gauss}
\rput(-1.1,13.1){reduction}

\end{pspicture}}

%2
\rput(40,14){%
        \begin{pspicture}(0,0)(20,28) %\psgrid

\pscurve[linewidth=9pt,linecolor=pale](0,0)(4,1)(8,1.5)
(7.7,4.3)(8.7,7.2)(7.5,10.5)(10.8,12.2)(12.7,15.5)
(12.7,18.5)(10,20)(10,23.3)(13,26)

\psline(1,6)(3.2,3.6)(4,1)
\psline(4.2,7)(3.2,3.6)
\psline(6,7)(7.7,4.3)(8,1.5)
\psline(8.7,7.2)(7.7,4.3)

\psline(11.1,7)(12.6,4.2)(12,1.5)
\psline(15,7)(12.6,4.2)
\psline(16.5,6.9)(17.2,3.8)(16,1)
\psline(19.8,6)(17.2,3.8)

\psline(11.8,8.6)(8.7,7.2)(7.5,10.5)
\psline(5,12.2)(7.5,10.5)(10.8,12.2)(15,13)
\psline(10.8,12.2)(12.7,15.5)(16.4,16.5)
\psline(12.7,15.5)(12.7,18.5)(16.4,20.3)
\psline(12.7,18.5)(10,20)(6.3,21.2)
\psline(10,20)(10,23.3)(7,26)
\psline(10,23.3)(13,26)

%river
\psline[linewidth=2.5pt,linecolor=blue2](0,0)(4,1)(8,1.5)
(12,1.5)(16,1)(20,0)

\psline[ArrowInside=->,ArrowInsidePos=0.69](13,26)(10,23.3)
\psline[linecolor=red,ArrowInside=->,ArrowInsidePos=0.69](12.7,18.5)(10,20)
\psline[linecolor=red,ArrowInside=->,ArrowInsidePos=0.69](7.5,10.5)(5,12.2)
\psline[linecolor=red,ArrowInside=->,ArrowInsidePos=0.69](7.7,4.3)(6,7)

\psline[linecolor=red,ArrowInside=->,ArrowInsidePos=0.69](4,1)(3.2,3.6)

%\psset{arrowscale=1.2,arrowinset=0.1,arrowlength=1.0}
%\psline[linewidth=2.5pt,linecolor=red,ArrowInside=->,ArrowInsidePos=0.6](4,1)(8,1.5)
%\psline[linewidth=2.5pt,linecolor=red,ArrowInside=->,ArrowInsidePos=0.6](20,0)(16,1)

\rput(5.5,3.3){\textcolor{gray}{$_{271}$}}
\rput(10.1,3.9){\textcolor{gray}{$_{299}$}}
\rput(5,9){\textcolor{gray}{$_{1189}$}}
\rput(11.5,10.5){\textcolor{gray}{$_{2705}$}}
\rput(7.7,16.3){\textcolor{gray}{$_{7489}$}}
\rput(16,14.6){\textcolor{gray}{$_{19199}$}}
\rput(16,18.1){\textcolor{gray}{$_{50671}$}}
\rput(4.8,23){\textcolor{gray}{$_{158549}$}}
\rput(14,23){\textcolor{gray}{$_{97121}$}}
\rput(10,27.1){\textcolor{gray}{$_{503851}$}}

\rput(8,0){$v$}

\rput(22.4,15){Zagier}
\rput(22.4,13.1){reduction}

\end{pspicture}}

\end{pspicture}
\caption{Reduction examples}
\label{exg}
\end{figure}
%*************************************

Hence Gauss reduction must reach the river, first stopping at the edge on the river to the left of $v$, (unless on an initial left turn sequence). This edge must be directed rightwards. Now \e{c-} applies and  goes as far as possible along the river making right turns, with the final $S$ directing the edge leftwards. Next \e{c+} applies making left turns as far as possible along the river (moving rightwards) with the final edge directed rightwards. In this way Gauss reduction  produces a cycle of forms on the river made up of: (a) rightward directed edges at the end of a sequence of left turns,  (b) leftward directed edges at the end of a sequence of right turns.
As in Figure \ref{alo}, they may be characterized as follows:

\begin{adef} \la{best}
{\rm A form $[a,b,c]$  is {\em G-reduced} or {\em Gauss reduced} if $a c<0$ and $|a+c|< b$.
}
\end{adef}

It may be seen that these are the same as the simply reduced forms of Definition \ref{sr}, except that the simply reduced forms are all directed rightwards. Figure \ref{griv} illustrates how the Gauss reduction in Figure \ref{exg} continues.

%*************************************
% gred long river
\SpecialCoor
\psset{griddots=5,subgriddiv=0,gridlabels=0pt}
\psset{xunit=0.23cm, yunit=0.23cm, runit=0.2cm}
%\psset{xunit=1cm, yunit=1cm, runit=1cm}
\psset{linewidth=1pt}
\psset{dotsize=5pt 0,dotstyle=*}
\begin{figure}[ht]
\centering
\begin{pspicture}(2,0)(61,12) %\psgrid

%\psline(0,0)(10,0)(10,5)(0,5)(0,0)
%\psset{arrowscale=2,arrowinset=0.5}
\psset{arrowscale=1.2,arrowinset=0.1,arrowlength=1.0}
\newrgbcolor{light}{0.9 0.9 1.0}
%\newrgbcolor{light}{0 0 1}
%\newrgbcolor{blue2}{0.4 0.4 1.0}
\newrgbcolor{blue2}{0.3 0.3 0.8}

\psline(5.3,1)(6,4)
\psline(11,1)(10,4)
\psline(10.5,10)(13.5,8)(14,5)
\psline(15,11)(13.5,8)
\psline(17.2,11)(18.5,8)(18,5)
\psline(21,10.2)(18.5,8)
\psline(22,1)(22,4)

\psline(23,10)(25.2,7.6)(26,5)
\psline(26.2,11)(25.2,7.6)
\psline(28,11)(29.7,8.3)(30,5.5)
\psline(30.7,11.2)(29.7,8.3)

\psline(33.1,11)(34.6,8.2)(34,5.5)
\psline(37,11)(34.6,8.2)
\psline(38.5,10.9)(39.2,7.8)(38,5)
\psline(41.8,10)(39.2,7.8)

\psline(41,0.7)(42,4)
\psline(46,0.3)(46,3.5)
\psline(50.7,0.6)(50,4)

\psline(50,10.3)(53,8.8)(53.5,5.5)
\psline(55,11.2)(53,8.8)
\psline(56.3,1.5)(57,5)

%river
\psline[linewidth=2.5pt,linecolor=blue2](2,5)(6,4)(10,4)(14,5)(18,5)(22,4)(26,5)(30,5.5)
(34,5.5)(38,5)(42,4)(46,3.5)(50,4)(53.5,5.5)(57,5)

%\psline[linewidth=2.5pt,linecolor=blue2,ArrowInside=->,ArrowInsidePos=0.6](2,5)(6,4)
\psline[linewidth=2.5pt,linecolor=blue2,linestyle=dashed](57,5)(61,4)

\psline[linewidth=2.5pt,linecolor=red2,ArrowInside=->,ArrowInsidePos=0.6](10,4)(14,5)
\psline[linewidth=2.5pt,linecolor=red2,ArrowInside=->,ArrowInsidePos=0.6](22,4)(18,5)
\psline[linewidth=2.5pt,linecolor=red2,ArrowInside=->,ArrowInsidePos=0.6](22,4)(26,5)
\psline[linewidth=2.5pt,linecolor=red2,ArrowInside=->,ArrowInsidePos=0.6](42,4)(38,5)
\psline[linewidth=2.5pt,linecolor=red2,ArrowInside=->,ArrowInsidePos=0.6](50,4)(53.5,5.5)
\psline[linewidth=2.5pt,linecolor=red2,ArrowInside=->,ArrowInsidePos=0.6](57,5)(53.5,5.5)

\rput(8,7){\textcolor{gray}{$_{59}$}}
\rput(16.1,7.8){\textcolor{gray}{$_{181}$}}
\rput(22.0,7){\textcolor{gray}{$_{145}$}}
\rput(27.5,7.3){\textcolor{gray}{$_{271}$}}

\rput(32.1,7.9){\textcolor{gray}{$_{299}$}}
\rput(36.8,7.7){\textcolor{gray}{$_{229}$}}
\rput(46.3,6.5){\textcolor{gray}{$_{61}$}}
\rput(58,7.5){\textcolor{gray}{$_{59}$}}

\rput(2.6,2.3){\textcolor{gray}{$_{-245}$}}
\rput(8,2){\textcolor{gray}{$_{-221}$}}
\rput(16,2.5){\textcolor{gray}{$_{-79}$}}
\rput(32,2.7){\textcolor{gray}{$_{-49}$}}

\rput(43.8,1.9){\textcolor{gray}{$_{-205}$}}
\rput(48.2,1.7){\textcolor{gray}{$_{-239}$}}
\rput(53.4,3.1){\textcolor{gray}{$_{-151}$}}
\rput(58.8,2.2){\textcolor{gray}{$_{-245}$}}

\end{pspicture}
\caption{A cycle of six G-reduced forms}
\label{griv}
\end{figure}
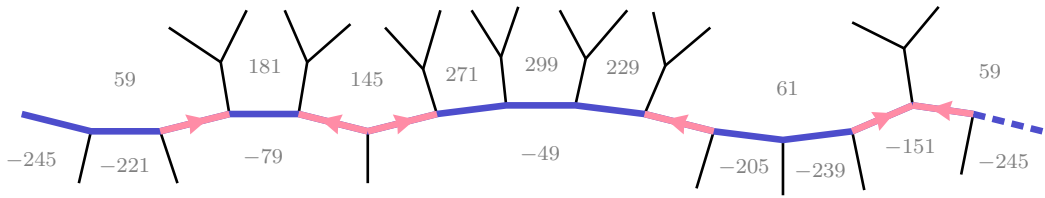
%*************************************

Our graphical arguments have proven  the next result.

\begin{theorem}
Every form of non-square discriminant $D>0$ is explicitly equivalent to a cycle of G-reduced forms. Two  forms are equivalent if and only if their  cycles of G-reduced forms are the same.
\end{theorem}

We remark that the most common definition  of  `Gauss reduced'  in the literature has the following more complicated condition, used by Gauss himself:
\begin{equation}\label{best2}
  0 < b < \sqrt{D}, \qquad \sqrt{D}-b < 2|a| < \sqrt{D}+b.
\end{equation}
Frobenius in \cite[Sect.~1]{fro} demonstrated the equivalence of \e{best2} and Definition \ref{best}, as well as other variants.

\subsection{Zagier reduction}
In \cite[p. 122]{z81}, Zagier  defined his reduction matrix as
\begin{equation*}
   \mz(q)  :=  -T^{-k}S = \begin{pmatrix} k & 1 \\ -1 & 0 \end{pmatrix} \quad \text{for} \quad k= \left\lceil \frac{b+\sqrt{D}}{2a}\right\rceil.
\end{equation*}
This can be written in our notation as
\begin{equation*}
   \mz(q) = S \cdot M_R'(q|S)  \cdot R.
\end{equation*}
 To give a geometric description of Zagier reduction, let $P$ be the unique simple directed path from $q$ to the river,  then continuing  leftwards along it. Assume first that $q$ lies above the river. If $q$ has the same direction as $P$ then by Proposition \ref{four} (iv) we see that $q|\mz(q)$ is on $P$ and directed against it. If $q$ is directed against $P$ then, with Proposition \ref{four} (iii), right turns are made towards the river going one edge beyond $P$ with $q|\mz(q)$ directed away from the river. Repeating this, as on the right of Figure \ref{exg}, Zagier reduction must reach the river and then produces a cycle of forms pointing up from it. These are the Z-reduced forms $[a,b,c]$ with nearby region labels $a$, $c$ positive and $a-b+c$ negative. %They are also called ``positive confluences'' or inlets in \cite[p. 22]{smi18}.
For forms below the river the reduction proceeds similarly.

\begin{adef} \la{zredd}
{\rm A form $[a,b,c]$  is {\em Z-reduced} or {\em Zagier reduced} if $a, c >0$ and $b>a+c$. It  is {\em Z*-reduced} if $a, c >0$ and $a+b+c<0$.
}
\end{adef}

This alternative Z*-reduction is sometimes convenient and used in \cite[p. 24]{VZ13}, for example. Clearly $q$ is Z*-reduced if and only if $q|S$ is Z-reduced. The number of Z and Z*-reduced forms in a fixed discriminant is finite -- this follows from Lemma \ref{har} with the number of distinct configurations on rivers being finite, or from Lemma \ref{hop} below.   Figure \ref{zriv} shows how the Zagier reduction in Figure \ref{exg} continues.
%*************************************
% zred long river
\SpecialCoor
\psset{griddots=5,subgriddiv=0,gridlabels=0pt}
\psset{xunit=0.23cm, yunit=0.23cm, runit=0.2cm}
%\psset{xunit=1cm, yunit=1cm, runit=1cm}
\psset{linewidth=1pt}
\psset{dotsize=5pt 0,dotstyle=*}
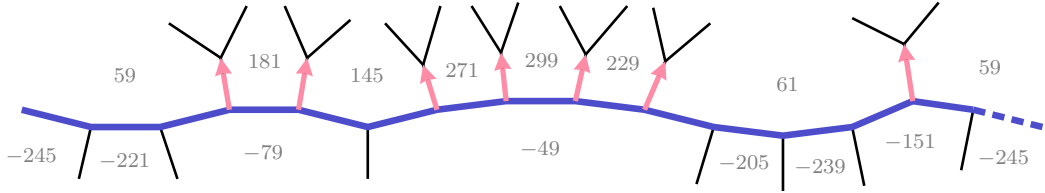
\begin{figure}[ht]
\centering
\begin{pspicture}(2,0)(61,12) %\psgrid

%\psline(0,0)(10,0)(10,5)(0,5)(0,0)
%\psset{arrowscale=2,arrowinset=0.5}
\psset{arrowscale=1.2,arrowinset=0.1,arrowlength=1.0}
\newrgbcolor{light}{0.9 0.9 1.0}
%\newrgbcolor{light}{0 0 1}
%\newrgbcolor{blue2}{0.4 0.4 1.0}
\newrgbcolor{blue2}{0.3 0.3 0.8}

\psline(5.3,1)(6,4)
\psline(11,1)(10,4)
\psline(10.5,10)(13.5,8)(14,5)
\psline(15,11)(13.5,8)
\psline(17.2,11)(18.5,8)(18,5)
\psline(21,10.2)(18.5,8)
\psline(22,1)(22,4)

\psline(23,10)(25.2,7.6)(26,5)
\psline(26.2,11)(25.2,7.6)
\psline(28,11)(29.7,8.3)(30,5.5)
\psline(30.7,11.2)(29.7,8.3)

\psline(33.1,11)(34.6,8.2)(34,5.5)
\psline(37,11)(34.6,8.2)
\psline(38.5,10.9)(39.2,7.8)(38,5)
\psline(41.8,10)(39.2,7.8)

\psline(41,0.7)(42,4)
\psline(46,0.3)(46,3.5)
\psline(50.7,0.6)(50,4)

\psline(50,10.3)(53,8.8)(53.5,5.5)
\psline(55,11.2)(53,8.8)
\psline(56.3,1.5)(57,5)

%river
\psline[linewidth=2.5pt,linecolor=blue2](2,5)(6,4)(10,4)(14,5)(18,5)(22,4)(26,5)(30,5.5)
(34,5.5)(38,5)(42,4)(46,3.5)(50,4)(53.5,5.5)(57,5)

%\psline[linewidth=2.5pt,linecolor=blue2,ArrowInside=->,ArrowInsidePos=0.6](2,5)(6,4)
\psline[linewidth=2.5pt,linecolor=blue2,linestyle=dashed](57,5)(61,4)

\psline[linewidth=2pt,linecolor=red2]{->}(14,5)(13.5,8)
\psline[linewidth=2pt,linecolor=red2]{->}(18,5)(18.5,8)
\psline[linewidth=2pt,linecolor=red2]{->}(26,5)(25.2,7.6)
\psline[linewidth=2pt,linecolor=red2]{->}(30,5.5)(29.7,8.3)
\psline[linewidth=2pt,linecolor=red2]{->}(34,5.5)(34.6,8.2)
\psline[linewidth=2pt,linecolor=red2]{->}(38,5)(39.2,7.8)
\psline[linewidth=2pt,linecolor=red2]{->}(53.5,5.5)(53,8.8)

\rput(8,7){\textcolor{gray}{$_{59}$}}
\rput(16.1,7.8){\textcolor{gray}{$_{181}$}}
\rput(22.0,7){\textcolor{gray}{$_{145}$}}
\rput(27.5,7.3){\textcolor{gray}{$_{271}$}}

\rput(32.1,7.9){\textcolor{gray}{$_{299}$}}
\rput(36.8,7.7){\textcolor{gray}{$_{229}$}}
\rput(46.3,6.5){\textcolor{gray}{$_{61}$}}
\rput(58,7.5){\textcolor{gray}{$_{59}$}}

\rput(2.6,2.3){\textcolor{gray}{$_{-245}$}}
\rput(8,2){\textcolor{gray}{$_{-221}$}}
\rput(16,2.5){\textcolor{gray}{$_{-79}$}}
\rput(32,2.7){\textcolor{gray}{$_{-49}$}}

\rput(43.8,1.9){\textcolor{gray}{$_{-205}$}}
\rput(48.2,1.7){\textcolor{gray}{$_{-239}$}}
\rput(53.4,3.1){\textcolor{gray}{$_{-151}$}}
\rput(58.8,2.2){\textcolor{gray}{$_{-245}$}}

\end{pspicture}
\caption{A cycle of seven Z-reduced forms}
\label{zriv}
\end{figure}
%*************************************
Reverse the direction of the Z-reduced forms in the figure  to see the Z*-reduced forms. To obtain the forms pointing down from the river, on the negative region side, in a topograph containing $q$, apply Zagier reduction to $-q$ to get $q_1, \dots, q_k$ and the desired forms are $-q_1, \dots, -q_k$.

The next theorem follows from our pictorial reasoning.

\begin{theorem} \cite[Satz 1, p. 122]{z81}
Every form of non-square discriminant $D>0$ is explicitly equivalent to a cycle of Z-reduced forms. Two  forms are equivalent if and only if their  cycles of Z-reduced forms agree.
\end{theorem}

Note that it is easy to have this reduction move  rightward along the river instead of leftward: replace $\mz(q)$ with $S \cdot M_L(q|S)  \cdot L$.

The Z and Z*-reduced forms  may be parameterized, as in \cite[p. 123]{z81}.
Define the set
\begin{equation} \la{ome}
  \Omega_D:= \left\{ (a,k)\in \Z^2 :|k|<\sqrt{D}, \  k^2 \equiv D \bmod 4, \  a>\frac{\sqrt{D}+k}2, \  a \Bigm| \frac{D-k^2}4  \right\}.
\end{equation}
In the following results we also allow $D$ to be a square.

\begin{lemma} \la{hop}
The Z*-reduced forms of any discriminant $D>0$ are exactly $[a,k-2a,*]$ for $(a,k)\in \Omega_D$. The Z-reduced forms of this discriminant are exactly $[a,-k+2a,*]$ for $(a,k)\in \Omega_D$.
%These forms are primitive if and only if $\gcd(a,k,D)=1$.
\end{lemma}
\begin{proof}
 Suppose that $q=[a,b,c]$ is Z*-reduced. It is easier to characterize forms $[a',b',c']$ when  $a' c'<0$ since then $(b')^2+4|a'c'|=D$ clearly has finitely many solutions. Here, as in \e{uu},
\begin{equation*}
  q | U = [a+b+c,-b-2a,a] \quad \implies  \quad (2a+b)^2+4a|a+b+c|=D.
\end{equation*}
Let $k=2a+b$ so that $|k|<\sqrt{D}$ and $k^2 \equiv D \bmod 4$. Then $a \mid (D-k^2)/4$ and $c>0$ implies $a>(k+\sqrt{D})/2$.  Hence
\begin{equation*}
  [a,b,c]=[a,k-2a,a-k+(k^2-D)/(4a)]
\end{equation*}
takes the desired form with $(a,k)\in \Omega_D$. It is easy to check that the converse is also true. The lemma's second statement  follows from the first since $[a,b,c]$ is Z*-reduced if and only if $[a,-b,c]$ is Z-reduced.
\end{proof}

\begin{cor} \la{rop}
Let $q=[a,b,c]$ have any discriminant $D>0$. Then
\begin{align*}
  \sqrt{D+4} \lqs b \lqs (D+1)/2 \qquad & \text{if $q$ is $Z$-reduced}, \\
   -(D+1)/2 \lqs b \lqs  -\sqrt{D+4} \qquad & \text{if $q$ is $Z^*$-reduced}.
\end{align*}

\end{cor}
\begin{proof}
In the Z-reduced case, by Lemma \ref{hop}, $b=2a-k \lqs \frac{D-k^2}2 -k$. Then use $k^2+2k+1 \gqs 0$ to obtain the upper bound. For the lower bound $b^2=D+4ac \gqs D+4$. The other case has $-b$ instead of $b$.
\end{proof}

\subsection{Reduction comparison}

We may briefly compare the simple reduction introduced in Section \ref{pnsd} with Gauss and Zagier reduction. As seen in Figure \ref{exg}, after a possible initial sequence of left turns,  Gauss and Zagier reduction make successive right turns towards the river, leaving the optimum path $P$ to reverse direction. Figure \ref{exg} shows that starting with a form above the river, Gauss and Zagier reduction work in essentially the same way, differing only when they reach the river. For forms below the river they work similarly, though Gauss reduction uses successive left turns towards the river with reversals while Zagier reduction uses successive right turns. This follows from Proposition \ref{four}.

Simple reduction works with alternating sequences of first left then right turns towards the river. It remains on the optimum path $P$ except for an initial reversal if the form being reduced is directed away from the river. Gauss and Zagier reduction have the advantage of using a single reduction matrix, $\mg$ or $\mz$, to reduce forms. Simple reduction requires $M_L$ and $M_R$ but produces a more direct path.

\section{On the river} \la{ive}

\subsection{Automorphs}
The group of automorphs of  $q=[a,b,c]$ and the stabilizer group of $z \in \C$ are defined as
$$
  \aut(q)  := \left\{ M \in \SL(2,\Z): q|M = q\right\}, \qquad
  \stab(z)  := \left\{ M \in \SL(2,\Z): Mz = z\right\}.
$$
$\aut(q)$ has a well-known simple structure that we will require. It is clear that $\aut(\lambda q) = \aut(q)$, so we need only consider primitive forms. Put
\begin{equation}\label{gqtu}
 G_q(t,u)  := \begin{pmatrix} (t- b u)/2 & -c u \\ a u & (t+ b u)/2 \end{pmatrix}.
\end{equation}

\begin{prop}\cite[Thm. 202]{la58}, \cite[Thm. 3.9]{bue89} \label{au}
Let $q$ be a primitive form of non-square discriminant $D$. Then
\begin{equation} \la{aust}
 \aut(q) = \stab(\ze_q) = \left\{ G_q(t,u): t,u \in \Z, t^2- D u^2 =4 \right\}.
\end{equation}
\end{prop}
\begin{proof}
If $M=(\begin{smallmatrix}\alpha & \beta\\\g &\delta\end{smallmatrix}) \in \stab(\ze_q)$ then $\g \ze_q^2+(\delta-\alpha)\ze_q-\beta=0$. Since $a \ze_q^2+b\ze_q+c=0$ for $\gcd(a,b,c)=1$ and $[\Q(\ze_q):\Q]=2$, we must have
\begin{equation}\label{gdbx}
\g=a u, \qquad \delta-\alpha = b u, \qquad \beta=-c u
\end{equation}
for some integer $u$. Letting $t=\alpha+\delta=\text{tr}(M)$ also shows
$$
\alpha=(t- b u)/2, \qquad \delta=(t+ b u)/2,
$$
and $\det(M)=1$ implies $t^2-u^2 D =4$.
We have shown that
\begin{equation}\label{stabb}
 \stab(\ze_q) \subseteq \left\{ G_q(t,u): t,u \in \Z, t^2- D u^2 =4 \right\}.
\end{equation}

Next, a computation verifies %, (see \cite[p. 45]{dav}),
\begin{equation}\label{csh}
  q|G_q(t,u)=q  \qquad \text{for} \qquad  t^2- D u^2  =4, \quad D\in \Z.
\end{equation}
Then \e{roots} and \e{csh} imply $G_q(t,u) \ze_q=\ze_q$ and we have equality in \e{stabb}. Lastly, $\aut(q) \subseteq \stab(\ze_q)$ by \e{roots} and $\aut(q) \supseteq \stab(\ze_q)$ by \e{csh} and \e{stabb}.
\end{proof}

Consequently, the automorphs of primitive forms  of non-square discriminant $D$ are in bijection with  solutions of the Pell equation $t^2- D u^2  =4$. It is also routine to check that the map $\psi:\aut(q) \to \C^*$ given by $G_q(t,u) \mapsto (t+u\sqrt{D})/2$ is an injective homomorphism for these $D$ values.  For example, the trivial automorphs correspond to $t=\pm 2, u=0$   and map to $\pm 1$.

\begin{cor} \la{aut}
Let $q$ be a primitive form of   discriminant $D$. Then $\aut(q)=\{\pm I\}$  for $D$ a nonzero square or $D<-4$. Also $\aut(q)/\{\pm I\}\cong \Z$, $\Z/3\Z$   and $\Z/2\Z$  for $D=0$, $-3$ and $-4$, respectively. (See Figures \ref{det0} and \ref{-3-4}.)
\end{cor}
\begin{proof}
By Proposition \ref{au}, the small number of solutions to $t^2- D u^2  =4$ for $D<0$ gives the result in these cases, with the image of $\psi$ being the 6th, 4th or 2nd roots of unity. For this, see also \cite[Sect. 8, Satz 2]{z81}. When $D=0$, check directly that $q|M = q$ for $q=[0,0,1]$ iff $M=\pm T^j$. Hence $\aut(q)/\{\pm I\}\cong \Z$ and any  equivalent form will have a conjugate automorph group. Lastly, let $\mathcal T$ be a topograph with   square discriminant $D>0$. It contains a unique reduced form $q_0$, (recall Definition \ref{defx}). Any nontrivial $M$ in $\aut(q_0)$ corresponds to a path in $\mathcal T$ linking $q_0$ to a copy. But $q_0$ lies on the left lake where it meets the river and by the structure of $\mathcal T$, seen in Section \ref{cla}, there can be no copies. So $\aut(q_0)=\{\pm I\}$ and the same is true for any form $q$ on $\mathcal T$ as their automorph groups are conjugate.
\end{proof}

We may now focus on simplifying the structure of $\aut(q)$ in the remaining case. % of $q$  a primitive form of non-square discriminant $D>0$.

\begin{lemma} \la{www}
Let $q$ be a primitive form of non-square discriminant $D>0$. If $t^2- D u^2  =4$ only has the trivial solutions $(t,u)=(\pm 2,0)$ then $\aut(q)=\{\pm I\}$. If there are non-trivial solutions, let $(t_0,u_0)$ be the smallest in positive integers. In this case
\begin{equation} \label{styy}
  \aut(q) = \left\{ \pm G_q^n : n \in \Z\right\} \qquad \text{for} \qquad G_q:= G_q(t_0,u_0).
\end{equation}
\end{lemma}
\begin{proof}
The first statement follows from \e{aust}. Now assume the non-trivial solution $(t_0,u_0)$ exists. The image of $\psi$ is a subgroup $S$ of $\R^*$. Since $(t+u\sqrt{D})/2 \cdot (t-u\sqrt{D})/2 =1$ it follows that both factors are $>0$ if $t>0$. Hence $S=\pm S_0$ where $S_0$ restricts to $t>0$ and is a multiplicative subgroup of $\R_{>0}$. Then $\log(S_0)$ is an additive subgroup of $\R$ with smallest positive element $\log((t_0+u_0\sqrt{D})/2)$. It must be the case that this element generates $\log(S_0)$ and so \e{styy} follows.
\end{proof}

\begin{prop} \label{t0}
Let $q$ be a simple primitive form on a topograph with non-square discriminant $D>0$. Let $M=(\begin{smallmatrix}\alpha & \beta\\\g &\delta\end{smallmatrix})=L^{a_0}R^{a_1} \cdots$ for $a_i\gqs 0$ correspond to a shortest period of the river, giving $q|M=q$. Then
$t_0=\alpha+\delta$, $u_0=\gcd(\g,\delta-\alpha,\beta)$ gives the smallest positive integer solution to $t^2- D u^2  =4$. We have that
$M$  is a primitive hyperbolic matrix, $M=G_q:=G_q(t_0,u_0)$ and \e{styy} holds.
\end{prop}
\begin{proof}
Since $M$ is a non-trivial element of $\aut(q)$, we have immediately from Lemma \ref{www} that the non-trivial solution $(t_0,u_0)$ exists and that \e{styy} holds. Hence $M=\pm G_q^k$ for some $k \in \Z$. In fact, since all the entries of $M$ are $>0$ (it contains at least one $L$ and one $R$ by Lemma \ref{har}) and all the entries of $G_q$ are $>0$ ($q$ is simple), we must have $M= G_q^k$ for some $k \in \Z_{\gqs 1}$. Now $G_q$ corresponds to a topograph path from $q$ to a copy of $q$. This gives a path along the river, and with positivity of entries again,  $G_q= M^\ell$ for some $\ell \in \Z_{\gqs 1}$.   Consequently $M$ and $G_q$ give the same river path and $M=G_q$.

The matrix $M$ is hyperbolic since tr$(M)>2$. To show $M$ is primitive, suppose $M=N^r$ for hyperbolic $N \in \SL(2,\Z)$ and $r\gqs 1$. The two fixed points of $N$ must be the fixed points of $M$. Hence $N$ must fix $z_q$. Proposition \ref{au} then implies that $N \in \aut(q)$ and so $N=\pm M^k$ for some $k \in \Z$ by \e{styy}. Then $M = \pm M^{k r}$. As all the powers of a hyperbolic matrix are distinct we must have $r=1$, as desired.
\end{proof}

\begin{adef} \la{unit}
{\rm For a non-square discriminant $D>0$ we use the following notation.
\begin{enumerate}
  \item Let $(t_0, u_0)$ or  $(t_D, u_D)$ refer to the smallest  solution to $t^2-Du^2=4$ in positive integers. We know that this solution exists by Proposition \ref{t0}.
  \item Set $\varepsilon_D:=(t_D+u_D\sqrt{D})/2$.
  \item Define $G_q:= G_q(t_D,u_D)$ with \e{gqtu}.%, for $q$ of discriminant $D$.
\end{enumerate}
}
\end{adef}

Note that if $D$ is a fundamental discriminant then $\varepsilon_D$ is the smallest unit $>1$ of norm $1$ in the ring of integers of $\Q(\sqrt{D})$. Combining  Propositions \ref{au}, \ref{t0} and Lemma \ref{www} gives:

\begin{cor}
Let $q$ be any  primitive form  with non-square discriminant $D>0$. Then
\begin{equation} \label{st2}
  \aut(q) = \stab(z_q) = \left\{ \pm G_q^n : n \in \Z\right\}.
\end{equation}
\end{cor}

We may characterize $G_q \in \SL(2,\Z)$ for primitive $q=[a,b,c]$ as the generator of $\aut(q)/\{\pm I\}$  with positive trace and with bottom left entry having the same sign as $a$, (this distinguishes it from $G_q^{-1}$). The following relation is also useful and can be verified with a computation:
\begin{equation}\label{idm}
  G_{q'} = M G_q M^{-1} \qquad \text{for} \qquad q'=q|M  \qquad \text{with} \qquad   M\in \SL(2,\Z).
\end{equation}
The next corollary follows quickly from \e{gqtu} and Proposition \ref{t0}. See also \cite[Prop. 1.4]{sa82} and \cite[Sect. 2.1]{ri21}.

%Then $\varepsilon_D$ is a unit in the ring of integers of $\Q(\sqrt{D})$ with norm $1$. It is the smallest such unit $>1$.

\begin{cor} \label{tu2}
The map $q \mapsto G_q$ is a bijection from the set of primitive forms of non-square discriminant $D>0$ to the set of primitive hyperbolic matrices $(\begin{smallmatrix}\alpha & \beta\\\g &\delta\end{smallmatrix})\in \SL(2,\Z)$ with trace $t_D$ and $\gcd(\g,\delta-\alpha,\beta)=u_D$. The inverse of the map is  $(\begin{smallmatrix}\alpha & \beta\\\g &\delta\end{smallmatrix}) \mapsto [\g,\delta-\alpha,-\beta]/u_D$.
\end{cor}

By \e{idm} and  Corollary \ref{tu2}, equivalence classes of primitive forms correspond to conjugacy classes of primitive hyperbolic matrices and hence closed geodesics on the modular surface $\G\backslash \H$.

\subsection{Class numbers for $D>0$} \la{rty}

A version of the following elegant result is stated in \cite[Eq. (6.4)]{zk75} and given as an exercise in \cite[p. 138]{z81}.
Recall Z*-reduced forms from Definition \ref{zredd}, and  $\varepsilon_D$, $G_q$ from Definition \ref{unit}.

\begin{prop} \la{proz}
Let $q_1, q_2, \dots, q_n$ be all the  Z*-reduced forms on a primitive topograph with non-square discriminant $D>0$. Then the product of their first roots has a simple evaluation:
$
\ze_{q_1} \ze_{q_2} \cdots \ze_{q_n} =  \varepsilon_D.
$
\end{prop}
\begin{proof}
For any  $q$ of this discriminant $D$, a computation reveals that
\begin{equation}\label{rev}
  G_q \begin{pmatrix} \ze_q \\ 1 \end{pmatrix} = \varepsilon_D \begin{pmatrix}  \ze_q \\ 1 \end{pmatrix}.
\end{equation}
From each Z*-reduced $q_i$, it is easy to see on the topograph that the next Z*-reduced form is reached by left turns, ($L=T$), followed by an $S$. See Figure \ref{zriv} with the arrows reversed. We may assume the ordering $q_{i+1}=q_i|T^{b_i}S$ with indices $\bmod$ $n$ and integers $b_i\gqs 2$.
Starting at the simple form $q':=q_1|T$,
\begin{equation*}
  G_{q'}= T^{b_1-1}S T^{b_2}S T^{b_3}S \cdots T^{b_n}S T
\end{equation*}
by Proposition \ref{t0}, requiring primitivity. Hence, with \e{idm},
\begin{equation} \label{rev2}
  G_{q_1} = T^{b_1}S T^{b_2}S T^{b_3}S \cdots T^{b_n}S.
\end{equation}
For convenience write the roots $\ze_{q_i}$  as  $w_i$. Then
\begin{equation*} %\la{la}
  w_{i+1}=\ze_{q_{i+1}}= \ze_{q_i|T^{b_i}S} = S T^{-b_i} \ze_{q_i} = \begin{pmatrix} 0 & -1 \\ 1 & -b_i \end{pmatrix} w_i = \frac{1}{b_i-w_i}.
\end{equation*}
Fixing $k$, a simple induction on decreasing $j$ %, relying on \e{la},
finds
\begin{align*}
  T^{b_j}S T^{b_{j+1}}S  \cdots T^{b_k}S \begin{pmatrix} w_{k+1} \\ 1 \end{pmatrix}
  & = \begin{pmatrix} b_j & -1 \\ 1 & 0 \end{pmatrix}
  \begin{pmatrix} b_{j+1} & -1 \\ 1 & 0 \end{pmatrix} \cdots
  \begin{pmatrix} b_k & -1 \\ 1 & 0 \end{pmatrix}\begin{pmatrix} w_{k+1} \\ 1 \end{pmatrix}\\
  & = \begin{pmatrix} w_j w_{j+1} \cdots w_{k+1} \\ w_{j+1} \cdots w_{k+1} \end{pmatrix},
\end{align*}
for $j\lqs k$. Thus, with $j=1$, $k=n$ and \e{rev}, \e{rev2},
\begin{equation*}
   G_{q_1} \begin{pmatrix} w_{1} \\ 1 \end{pmatrix} =  \begin{pmatrix} w_1 w_{2} \cdots w_n w_{1} \\ w_{2} \cdots w_n w_{1} \end{pmatrix} = \varepsilon_D \begin{pmatrix}  w_1 \\ 1 \end{pmatrix},
\end{equation*}
completing the proof.
\end{proof}

The next corollary follows directly from Proposition \ref{proz} by including all primitive topographs of discriminant $D$. It may be compared with Dirichlet's class number formula for positive fundamental discriminants, given in \cite[Sect. 9, Satz 3]{z81}, \cite[Cor. 5.6.10]{coh93} and \cite[p. 3996]{dit21} for example. Then \e{bot} has the advantages of a simpler proof, applying to a larger set of $D$s and not requiring Kronecker symbols.

\begin{cor}
For non-square discriminants $D>0$,
\begin{equation}\label{bot}
\varepsilon_D^{h(D)} = \prod_{\substack{ a+b+c \,< \, 0 \,  <  \, a,  \, c \\ b^2-4ac  \, = \,  D, \, \gcd(a,b,c) \, = \, 1}} \frac{-b+\sqrt{D}}{2a}.
\end{equation}
\end{cor}

The product in \e{bot} is finite and may be computed using Corollary \ref{rop}. For example, when $D=148$ then $\varepsilon_D = 73+12\sqrt{37}$, $h(D)=3$ and both sides of \e{bot} agree. Zagier's parametrization in Lemma \ref{hop} implies that all the factors on the right of \e{bot} are $>1$. This  parametrization also gives:

\begin{cor} Recall $\Omega_D$ from \e{ome}. For non-square discriminants $D>0$,
$$
h^*(D) \log \varepsilon_D = \sum_{(a,k) \, \in \, \Omega_D} \log\left(1+\frac{\sqrt{D}-k}{2a}\right).
$$
\end{cor}

We remark that Proposition \ref{proz} has the following analog when $D$ is a square. Suppose a primitive topograph $\mathcal T$ contains the reduced form $[0,m,r]$ for $m>1$ as in Definition \ref{defx}. If $q_1, q_2, \dots, q_n$ are all the  Z*-reduced forms on  $\mathcal T$ then
$
\ze_{q_1} \ze_{q_2} \cdots \ze_{q_n} =  r.
$

\subsection{Binary necklaces} \la{neck}

Let $\mathcal T$ be a topograph with non-square discriminant $D>0$. Starting at any simple form $q$, (so directed rightwards), record the sequence of $L$ and $R$ turns along the river until we reach a copy of $q$ for the first time. This is the  river sequence from \cite[p. 395 - 397]{ri21} and we follow and build on the interesting discussion there. The river sequence of $\mathcal T$ is defined up to cyclic permutations, forming a {\em binary necklace} with $0=L$ and $1=R$, say. We may call a necklace {\em repeating} (not to overuse the terms primitive/imprimitive) if it is  made of smaller identical parts. For example $1011\cdot 1011 \cdot 1011$ is repeating while $101100$ is not. The river shown in Figures \ref{smriv}, \ref{griv} and \ref{zriv} has non-repeating river sequence
\begin{equation*}
  LLRRLRRRRLLLRL = 00110111100010.
\end{equation*}

\begin{theorem} \la{bina}
With their river sequences, primitive topographs of non-square discriminant $D>0$ are in bijection with non-repeating binary necklaces of length at least  $2$.
\end{theorem}
\begin{proof}
Starting with a primitive topograph $\mathcal T$, its river sequence  corresponds to a binary necklace $B$ of length at least $2$ by Lemma \ref{har}. It is non-repeating since the corresponding matrices $M=L^{a_0}R^{a_1}\cdots $  are primitive by Proposition \ref{t0}.

Beginning with a non-repeating binary necklace $B$ of length $\gqs 2$, let $M$ be made out of the corresponding sequence of $L$s and $R$s. Then $M$ equals $(\begin{smallmatrix}\alpha & \beta\\\g &\delta\end{smallmatrix})$ with positive integer entries and trace necessarily $>2$. Let
\begin{equation*}
  q = [\g, \delta-\alpha,-\beta]/g \qquad \text{for} \qquad g=\gcd(\g,\delta-\alpha,\beta).
\end{equation*}
The form $q$ is primitive and simple with $q|M=q$. Its discriminant is $\frac{(\alpha+\delta)^2-4}{g^2}$ which is positive and not a square. Then $q$ lies on a topograph $\mathcal T$ and it is clear that this  inverts  the first map.
\end{proof}

For example, the first necklaces correspond to topographs of these discriminants, including the Zagier reduced form coming first lexicographically:
\begin{alignat*}{4}
  01, \quad & D =5, &\quad &[1,3,1]&  \qquad \qquad 00001, \quad & D =32, &\quad &[1,6,1] \\
  001, \quad & D =12, &\quad &[1,4,1]&  \qquad  \qquad 00011, \quad & D =60, &\quad &[2,10,5]\\
  011, \quad & D =12, &\quad &[2,6,3]&  \qquad  \qquad 00101, \quad & D =96, &\quad &[3,12,4] \\
  0001, \quad & D =21, &\quad &[1,5,1]&  \qquad  \qquad 00111, \quad & D =60, &\quad &[3,12,7] \\
  0011, \quad & D =8, &\quad &[1,4,2]&  \qquad  \qquad 01011, \quad & D =96, &\quad &[5,14,5] \\
  0111, \quad & D =21, &\quad &[3,9,5]&  \qquad  \qquad 01111, \quad & D =32, &\quad &[4,12,7]
\end{alignat*}
The number of non-repeating binary necklaces (also known as Lyndon words) of length $n\gqs 2$ is the sequence A001037 in the OEIS with formula
$$
\frac 1n \sum_{d | n} \mu(n/d) \cdot 2^d.
$$
This is also discussed in \cite[p. 455]{cark17}. In that paper the periodicity is used to quotient topographs and produce what they term {\em \c{c}arks} where the infinite river is transformed into a circle and the topograph lies in an annulus. See also the binary necklaces in \cite{smi18}.
We remark that any multiple of a primitive topograph will have the same river sequence. Therefore all topographs of non-square discriminant $D>0$ have non-repeating river sequences. %Also, for each $n\gqs 2$, there are infinitely many non-primitive topographs with river sequences of that length.

For $q=[a,b,c]$ define $q^*:=[c,b,a]$. As a river sequence application, the relationship between $q$ and $q^*$  will be developed next.

\begin{prop} \label{-1}
Let $q$ be a primitive form of non-square discriminant $D$. Then $q \sim -q^*$ if and only if there exists an  integer solution $(t,u)$ to $t^2-Du^2=-4$.
\end{prop}
\begin{proof}
Suppose that $q|M=-q^*$ for $q=[a,b,c]$ and $M=(\begin{smallmatrix}\alpha & \beta\\\g &\delta\end{smallmatrix}) \in \SL(2,\Z)$.
%where we may assume that $\beta+\g\gqs 0$.
Then $M^{-1} \ze_q= \ze_{-q^*}=1/\ze_q$ implies $M(1/\ze_q)=\ze_q$. Hence $\delta \ze_q^2 +(\g - \beta)\ze_q-\alpha=0$. Let $u_*=\gcd(\delta, \g - \beta, \alpha)\gqs 1$. Then for some $\kappa=\pm 1$, as in \e{gdbx}, we must have
$$
\delta= \kappa a u_*, \qquad \g- \beta = \kappa b u_*, \qquad \alpha=-\kappa c u_*.
$$
Let $t_* =   \kappa(\beta + \g)$. Consequently
$$
\beta=\kappa (t_*- b u_*)/2, \qquad \g = \kappa (t_*+ b u_*)/2,
$$
and we have shown that
\begin{equation}\label{gq2}
M  = \kappa \begin{pmatrix} -c u_* & (t_*- b u_*)/2 \\ (t_*+ b u_*)/2 &  a u_*  \end{pmatrix},
\end{equation}
%Then $t_*\gqs 0$
 and the determinant of $M$ being $1$ is equivalent to $t_*^2- D u_*^2  =-4$.
%Therefore $\mathcal O_D$ has a unit $\varepsilon^*_D = (t_*+u_*\sqrt{D})/2$ of norm $-1$.

In the other direction, use that $q|M=-q^*$ for $M$ given by \e{gq2} when $t_*^2- D u_*^2  =-4$.
\end{proof}
%Have
%\begin{equation*}
%  \ze_{-q}=\ze_q', \qquad \ze_{q^*}=1/\ze_q', \qquad \ze_{-q^*}=1/\ze_q.
%\end{equation*}

It is straightforward to see that the topographs containing $q$, $q^*$, $-q$ and $-q^*$ have related river sequences. For example,
\begin{equation} \la{cba}
\begin{aligned}
  q & =[a,b,c] & & LRLLRRR & &\\
  q^* & =[c,b,a] & & RRRLLRL &\qquad &\text{order reversed}\\
  -q & =[-a,-b,-c]& & LLLRRLR &\qquad &\text{order reversed and letters switched}\\
  -q^* & =[-c,-b,-a]& \qquad & RLRRLLL &\qquad &\text{letters switched}
\end{aligned}
\end{equation}
This gives an easy way to check if $q \sim q^*$: see if the river sequence for the topograph containing $q$ is invariant under order reversal. Similarly $q \sim -q$ if we have invariance under order reversal with the letters switched -- this is pointed out in \cite[p. 396]{ri21}. See also \cite[p. 456]{cark17} for a discussion of these symmetries. In relation to Proposition \ref{-1}, we may next give  a simple way to tell if $q \sim -q^*$.

\begin{lemma} \la{lexx}
Any non-repeating binary necklace $B$ that is invariant under switching all its bits takes the form $X\cdot \overline{X}$ where $\overline{X}$ switches all the bits of $X$.
\end{lemma}
\begin{proof}
$B$  must have even length $2m$, or else switching bits gives a different number of $0$s and $1$s. Suppose $k$ with $0<k<2m$ gives the distance to shift $B$ left so that it equals $\overline{B}$. Then shifting  $B$ by $2k$ gives $B$ so that $2m|2k$. Hence $k=m$.
\end{proof}

Easily, if  $B=X\cdot \overline{X}$ then any cyclic permutation of $B$ has the same form.

\begin{theorem} \la{xx}
Suppose a primitive topograph of non-square discriminant $D>0$ has river sequence $M$, giving  $q|M=q$. If $M=X\cdot \overline{X}$ where $\overline{X}$ switches all the letters of $X$, then $q|X=-q^*$ and, with $X=(\begin{smallmatrix}\alpha & \beta\\\g &\delta\end{smallmatrix})$, the minimal positive integer solution to $t_*^2-Du_*^2=-4$ has  $t_* = \beta + \g$, $u_*=\gcd(\delta, \g - \beta, \alpha)$. If $M \neq X\cdot \overline{X}$ then $q \not\sim -q^*$ and there are no solutions.
\end{theorem}
\begin{proof}
Suppose $M=X\cdot \overline{X}$. With $J=(\begin{smallmatrix}0 & 1\\ 1 & 0\end{smallmatrix})$ we have $R=JLJ$ and $L=JRJ$ so that $\overline{X}=JXJ$. Hence we can write $M=N^2$ for $N=XJ$ with $\det(N)=-1$. Note that $N$ fixes two distinct real numbers since the quadratic equation to find them has discriminant \text{tr}$(N)^2-4\det(N)$. The fixed points of $N$ and $M$ must be the same, implying
\begin{equation} \la{ze}
  \ze_q = N \ze_q = XJ  \ze_q = X  (1/\ze_q) = X \ze_{-q^*}.
\end{equation}
Let $Q=q|X$. We have $\ze_Q=X^{-1}\ze_q = \ze_{-q^*}$ by \e{roots} and \e{ze}. Similarly, the second roots match: $\ze'_Q= \ze'_{-q^*}$. The discriminants of $Q$ and $-q^*$ also match and hence $q|X = Q = -q^*$.

The proof of Proposition \ref{-1} now shows that $X$ is given by \e{gq2} and hence $t_*$ and $u_*$ satisfy $t_*^2-Du_*^2=-4$. Also $G_q=M=(XJ)^2$ implies
\begin{equation} \label{we}
  t_D=(t_*^2+Du_*^2)/2, \qquad u_D=t_*u_*.
\end{equation}
In general  $(t,u)\mapsto (\frac{t^2+Du^2}2,tu)$ gives a map from solutions of $t^2-Du^2=-4$ to solutions of $t^2-Du^2=4$. Therefore $(t_*,u_*)$ is a minimal solution of the former since $(t_D,u_D)$ is a minimal solution of the latter.

Lastly, if $q\sim -q^*$ then the river sequence for $q$ is invariant under switching its letters. By Lemma \ref{lexx} it must take the form $X\cdot \overline{X}$.
\end{proof}

See also \cite[Chap. 3]{bue89} for similar results based on continued fractions.
For example, there are $4$ primitive topographs when $D=145$ and their river sequences in binary are
\begin{gather*}
    0010100 \cdot 1101011, \\
    000001000 \cdot 111110111, \\
    000001110 \cdot 111110001, \\
    0000000000010 \cdot 1111111111101.
\end{gather*}
These take the form $X\cdot \overline{X}$ and by Theorem \ref{xx} we obtain the solution $t_*=24$, $u_*=2$ to $t_*^2-Du_*^2=-4$.
Proposition \ref{-1} now implies that $q \sim -q^*$ for all forms of this discriminant.

When a minimal positive integer solution $(t_*, u_*)$ to $t_*^2-Du_*^2=-4$ exists we may define the unit  $\varepsilon^*_D := (t_*+u_*\sqrt{D})/2$ of norm $-1$. Then \e{we} implies that $(\varepsilon^*_D)^2 = \varepsilon_D$.

\subsection{Binary words and equivalence in the wide sense} \la{word}

The river sequence for a topograph with square discriminant $D=m^2$ is the natural one: start at the leftmost river edge and record the sequence of $L$ and $R$ turns from there, moving along the river to the rightmost river edge. If there are not two lakes (i.e. $D=0$), or the two lakes are adjacent, then we say there is no river sequence. If the river has only one edge then we may say the river sequence is the empty set.

Let a primitive topograph $\mathcal T$ have discriminant $D=m^2\gqs 4$. Its river sequence can be easily computed as follows. By Theorem \ref{redsq} it contains a reduced form $[0,m,r]$ on the left lake with $0<r<m$. Let $q=[0,m,r]|S =[r,-m,0]$. Then  the associated path and matrix $M=L^{a_0}R^{a_1} \cdots $  from
\begin{equation}\label{vg}
  \ze_q = m/r=\langle a_0, a_1, \dots, a_n\rangle \qquad \text{for} \qquad 0<r<m, \quad \gcd(r,m)=1,
\end{equation}
 leads to the right lake. Dropping the first and last symbols gives the river sequence of $\mathcal T$. (The first symbol must be $L$ and switching the last symbol between $L$ and $R$ does not affect the value of the continued fraction -- see the discussion before Definition \ref{defx}.) Any non-primitive multiples of $\mathcal T$ have the same river sequence. As before, replacing $L$ with $0$ and $R$ with $1$ gives a binary word.

\begin{theorem}
With their river sequences, primitive topographs of square discriminant $\gqs 4$ are in bijection with  binary words. % corresponding to river turns.
\end{theorem}
\begin{proof}
As seen above, these topographs produce binary words from their river sequences. Being primitive with discriminant $D=m^2 \gqs 4$ implies that the binary word is of length $\gqs 0$ because $a_0+ \cdots +a_n$ in \e{vg} must be $\gqs 2$.

Starting with a binary word, add $0$ at each end and let $a_0$ be the  number of initial $0$s,  $a_1$ the number of following $1$s and so on. Let $m/r$ equal the continued fraction $\langle a_0, a_1, \dots\rangle$ in lowest terms. Since $a_i \gqs 1$ for $i=0,1, \dots$ we have $0<r<m$. The topograph containing $[0,m,r]$ is the result. (As noted earlier, adding a $1$ instead of $0$ on the right of the initial binary word also gives $m/r$.) It can be seen that this map is the inverse of the first one.
\end{proof}

The first binary words correspond to topographs of these discriminants, including their reduced forms:
\begin{alignat*}{4}
  \text{none}, \quad & D =1, &\quad &[0,1,1]&  \qquad \qquad 10, \quad & D =25, &\quad &[0,5,3] \\
  \{\}, \quad & D =4, &\quad &[0,2,1]&  \qquad  \qquad 11, \quad & D =16, &\quad &[0,4,3]\\
  0, \quad & D =9, &\quad &[0,3,1]&  \qquad  \qquad 000, \quad & D =25, &\quad &[0,5,1] \\
  1, \quad & D =9, &\quad &[0,3,2]&  \qquad  \qquad 001, \quad & D =49, &\quad &[0,7,2] \\
  00, \quad & D =16, &\quad &[0,4,1]&  \qquad  \qquad 010, \quad & D =64, &\quad &[0,8,3] \\
  01, \quad & D =25, &\quad &[0,5,2]&  \qquad  \qquad 011, \quad & D =49, &\quad &[0,7,3]
\end{alignat*}
Also, a river with all left turns has binary word $000 \cdots 00$ and corresponds to the topograph containing  $[0,m,1]$. One with all right turns has $111 \cdots 11$ and corresponds to $[0,m,m-1]$. The non-repeating condition of Theorem \ref{bina} is not needed here.
For $n\gqs 1$  there are exactly $2^{n-1}$ primitive topographs with square discriminant and rivers of length $n$.

Topographs containing $q$, $q^*$, $-q$ and $-q^*$ have the same relations among their river sequences as in \e{cba}, but for binary words now instead of necklaces. It follows that we can never have $q \sim -q^*$ for a form $q$ of square discriminant with river sequence of length $\gqs 1$, since switching letters will always produce a different sequence. This fact lets us extend  Proposition \ref{-1} to all discriminants:

\begin{theorem}\label{-m}
Let $q$ be a primitive form of any discriminant $D$. Then $q \sim -q^*$ if and only if there exist  integer solutions $t$, $u$ to $t^2-Du^2=-4$.
\end{theorem}
\begin{proof}
With Proposition \ref{-1} we need only treat square discriminants.  A primitive topograph $\mathcal T$ of discriminant $D=0$ has one lake and all other region labels positive, as in Figure \ref{det0}, or all others negative. For any form $q$ on $\mathcal T$ we have $q \not\sim -q^*$ since $-q^*$ cannot appear on $\mathcal T$ as well. There are also no solutions to $t^2-Du^2=-4$ in this case.

For $D=1$ and $D=4$ there are the solutions $(t_*,u_*)=(0,2)$ and $(0,1)$, respectively. In these cases $q|M=-q^*$ for $M$ given by \e{gq2}.

For $D=m^2$ with $m\gqs 3$ we have $a_0+ \cdots +a_n$ in \e{vg} at least $3$. Therefore any topograph of this discriminant has river sequence of length $\gqs 1$ and so all forms on it have $q \not\sim -q^*$. In this case there are also no solutions to $t^2-Du^2=-4$  since $u\neq 0$ and $(mu)^2-t^2>4$ for $mu \gqs 3$ and $mu>t\gqs 0$.
\end{proof}

There is a second important notion of equivalence of forms in the literature; see for example \cite[p. 62]{z81} and \cite[Sect. 1.5]{VZ13}. Two forms $q_1$ and $q_2$ are {\em equivalent in the wide sense} if $q_1|M=\det(M) q_2$ for some $M$ in $\GL(2,\Z)$.  Let $h_1(D)$ denote the number of primitive classes of discriminant $D$ in this sense. The following result is well known, at least for ideal classes \cite[p. 52]{dav}.

\begin{theorem}
For any discriminant $D$ we have the following. If $D\lqs 0$ then $h(D)=h_1(D)$. For $D> 0$ we also have $h(D)=h_1(D)$ if  there exist  integer solutions $t$, $u$ to $t^2-Du^2=-4$, while $h(D)=2 h_1(D)$ if there are no solutions. In particular, $h(D)=2 h_1(D)$ for  square $D>4$ and $h(1)=h_1(1)=h(4)=h_1(4)=1$.
\end{theorem}
\begin{proof}
We know from Theorem \ref{-m} that $q \sim -q^*$ for any primitive form of discriminant $D$ if and only if there exist  integer solutions $t$, $u$ to $t^2-Du^2=-4$.

Suppose first that $D>$ and no solution exists. The map $q \mapsto -q^*$ is an involution permuting the $h(D)$ equivalence classes of forms without fixing any. Hence we may list  class representatives $q_1$, $-q_1^*, \dots, q_k, -q_k^*$ and $h(D)=2k$. For $J=(\begin{smallmatrix}0 & 1\\ 1 & 0\end{smallmatrix})$, we have $-q^* = \det(J) q|J$ so that $q$ and $-q^*$ are equivalent in the wide sense. However $q_i$ and $q_j$ cannot be equivalent in the wide sense for $i \neq j$. If they were, then applying $J$ implies that $q_i \sim -q_j^*$, a contradiction. Therefore $h_1(D)=k$ in this case, giving $h(D)=2 h_1(D)$.

If $D\lqs 0$ then the same arguments apply, since $t^2-Du^2=-4$ has no solutions. The only difference is that now $h(D)=k$ by definition, counting only topographs with nonnegative region labels, i.e. forms that do not represent negative numbers. Hence $h(D)= h_1(D)$.

Finally, assume $D>$ and that solutions to $t^2-Du^2=-4$ exist. If $q_1$ and $q_2$ are  equivalent in the wide sense with $q_1|M=\det(M) q_2$ and $\det(M) = -1$, then applying $J$ gives
\begin{equation*}
  q_1|M J = -q_2|J = -q_2^* \sim q_2 \qquad \implies \qquad q_1 \sim q_2.
\end{equation*}
It follows that $h(D)=h_1(D)$ in this case.
\end{proof}

For non-square $D>0$, our  work in this section gives an easy algorithm to find the minimal solution of $t^2-Du^2=4$, decide whether  $t^2-Du^2=-4$ has  solutions, and find the minimal one if it does.

\begin{algo4}
{\rm
Start with $q=[a,b,c]$ the principal form \e{prin} of discriminant $D$, which is primitive and simple. If $a+b+c<0$, let $q \mapsto q|L$ and otherwise let $q \mapsto q|R$. This follows the river rightwards as in Figure \ref{alo}. Repeat until you get back to the principal form and this produces a river sequence  $M=L^{a_0}R^{a_1} \cdots$ for $a_i\gqs 0$. The minimal solution of  $t^2-Du^2=4$ and $\varepsilon_D$ are now given by Proposition \ref{t0}. If this sequence takes the form $X\cdot \overline{X}$, where $\overline{X}$ switches all the letters of $X$, then $t^2-Du^2=-4$ has a minimal solution which is given, along with $\varepsilon^*_D$,  by Theorem \ref{xx}. Otherwise solutions do not exist.
}
\end{algo4}

\section{Infinite series for topographs and class numbers} \label{di}
Duke, Imamo\=glu and T\'oth in \cite{dit21} reconsidered and extended work of Hurwitz from \cite{hur}, expressing class numbers as infinite series. By  restricting these series to involve a single equivalence class, (this possibility is mentioned in \cite[p. 3998]{dit21}), we obtain interesting results for topographs.

Recall that $\G=\PSL(2,\Z)$. As shown in \cite[Sects. 5, 7]{dit21}, the Poincar\'e series $P(\tau)$ with
\begin{equation*}
  P(\tau)= P(\tau;s_1,s_2,s_3):=\sum_{\g \in \G} \mathcal H(\g\tau), \qquad %\text{for} \qquad
\mathcal H(\tau)= \mathcal H(\tau;s_1,s_2,s_3):=\frac{\Im( \tau)^{s_1+s_2+s_3}}{|\tau|^{2s_2}|\tau-1|^{2s_3}},
\end{equation*}
are absolutely and uniformly convergent for $\Re(s_1), \Re(s_2), \Re(s_3) \gqs 1$ and $\tau$ in compact subsets of $\H$. These series are invariant under permutations of $(s_1,s_2,s_3)$. Also, 
\begin{equation}\label{ph}
  P(\tau;1,1,1) = 3\pi/2, \qquad P(\tau;1,2,2) = 3\pi/4 \qquad \text{for all} \quad \tau\in \H.
\end{equation}
 The topograph properties needed in the following sections are summarized in Section \ref{summ}.

\subsection{Negative discriminants}
For $q=[a,b,c]$ positive definite (so that $a, c>0$) of discriminant $D<0$ with first root $\ze_q$,
\begin{equation*}
  \mathcal H(\ze_q;s_1,s_2,s_3) = \left(\frac{\sqrt{|D|}}{2}\right)^{s_1+s_2+s_3}\frac 1{ a^{s_1}c^{s_2}(a+b+c)^{s_3}}.
\end{equation*}
Let $\mathcal T$ be a topograph of discriminant $D<0$ with $Q$ a form on it. Write $w_{\mathcal T} = w_Q$ for $|\aut(Q)/\{\pm I\}|$. By Corollary \ref{aut}, $w_{\mathcal T}$ is $3$ if $\mathcal T$ contains $[a,a,a]$, $2$ if it contains $[a,0,a]$ and $1$ otherwise.
We have
\begin{align*}
    P(\ze_Q) = \sum_{\g \in \G}  \mathcal H(\g \ze_Q) & = \sum_{\g \in \G}  \mathcal H( \ze_{Q|\g})\\
   & = w_Q \left(\frac{\sqrt{|D|}}{2}\right)^{s_1+s_2+s_3} \sum_{q=[a,b,c] \sim Q} \frac 1{ a^{s_1}c^{s_2}(a+b+c)^{s_3}},
\end{align*}
for positive definite $Q$, with the sum over distinct $q$. When $s_1=s_2=s_3=1$ this means, with \e{ph}, that
\begin{equation} \label{tru}
  \frac{12 \pi}{w_Q} = |D|^{3/2} \sum_{q=[a,b,c] \sim Q} \frac 1{ a (a+b+c)c}.
\end{equation}
As a sum over configurations in the topograph containing $Q$, we see in \e{tru} the  labels of regions adjacent to a vertex. In this way,  based on the methods in \cite{dit21}, the next theorem gives  topographic versions of results of Hurwitz in \cite{hur}.

\begin{theorem} \label{mik}
Let $\mathcal T$ be any topograph of discriminant $D<0$. Then
\begin{equation}\label{top}
   |D|^{3/2} \sum_{
\psset{xunit=0.12cm, yunit=0.12cm, runit=0.2cm}
%\psset{xunit=1cm, yunit=1cm, runit=1cm}
\psset{linewidth=1pt}
\psset{dotsize=7pt 0,dotstyle=*}
\begin{pspicture}(2,3)(14,8.3)
          \psline(5,5)(3,8)
\psline(5,5)(3,2)
\psline(5,5)(9,5)

%\psline(9,5)(11,8)
%\psline(9,5)(11,2)

\rput(2,5){$r$}
\rput(7,7.2){$s$}
\rput(7,2.8){$t$}

\rput(12.5,5){$ \in \mathcal T$}
        \end{pspicture}
}
\frac 1{|r s t|} =  4 \pi, \qquad \qquad |D|^{5/2}
\sum_{
\psset{xunit=0.12cm, yunit=0.12cm, runit=0.2cm}
%\psset{xunit=1cm, yunit=1cm, runit=1cm}
\psset{linewidth=1pt}
\psset{dotsize=7pt 0,dotstyle=*}
\begin{pspicture}(2,3)(14,8.3)
          \psline(5,5)(3,8)
\psline(5,5)(3,2)
\psline(5,5)(9,5)

%\psline(9,5)(11,8)
%\psline(9,5)(11,2)

\rput(2,5){$r$}
\rput(7,7.2){$s$}
\rput(7,2.8){$t$}

\rput(12.5,5){$ \in \mathcal T$}
        \end{pspicture}
}
\frac{|r+s+t|}{|r s t|^2} =  24 \pi,
\end{equation}
where we sum over all vertices of $\mathcal T$, (each vertex contributing one term).
\end{theorem}
\begin{proof}
The first equality in \e{top} follows from \e{tru}, combining the three configurations at each vertex of $\mathcal T$, and including absolute values to allow for the negative definite case. The term $w_Q$ drops out as the topograph includes the $w_Q$ copies of each form.

 The second equality follows similarly, using $(s_1, s_2, s_3)=(1,2,2)$.
\end{proof}

 Summing \e{top} over all topographs of discriminant $D$ then implies Hurwitz's infinite series formulas. These involve distinct forms $q=[a,b,c]$ and dividing by the multiplicity factor $w_q$ gives the Hurwitz class numbers  $H$ from Section \ref{<}:
\begin{align}
 H(|D|)  & = \frac{|D|^{3/2}}{12 \pi}\sum_{\substack{b^2-4ac=D \\ a>0}} \frac 1{ a (a+b+c)c} \la{hu}\\
 & =  \frac{|D|^{5/2}}{72 \pi}\sum_{\substack{b^2-4ac=D \\ a>0}} \frac {2a+b+2c}{a^2 (a+b+c)^2 c^2}, \la{hu2}
\end{align}
(see \cite[p. 3997]{dit21}). Include the condition $\gcd(a,b,c)=1$ in \e{hu}, \e{hu2} to obtain $h(D)$ for $D<-4$. These expressions may be compared with the following finite sum from Theorem \ref{negred}, for example:
\begin{equation*}
  h^*(D)= \sum_{\substack{b^2-4ac=D,\ |b|\lqs a \lqs c, \\ b\gqs 0\text{ \ if \ }|b|=a\text{ \ or \ }a=c}} 1 \qquad \qquad (D<0).
\end{equation*}

\subsection{Positive non-square discriminants}
Recall $\varepsilon_D$ from Definition \ref{unit}, with an algorithm to  compute it at the end of Section \ref{ive}. The next result answers affirmatively our question  in the introduction as to whether Theorem \ref{t3}, (\cite[Thm. 3]{dit21}), has a topographic analog.

\begin{theorem} \label{mt}
Let $\mathcal T$ be any topograph of non-square discriminant $D>0$. Define $\mathcal T_\star$ to equal $\mathcal T$ except that all the river edges are relabeled with $\sqrt{D}$ when directed rightwards, ($-\sqrt{D}$ when directed leftwards). Then
\begin{equation}\label{topp}
     D^{3/2}
\sum_{
\psset{xunit=0.12cm, yunit=0.12cm, runit=0.2cm}
%\psset{xunit=1cm, yunit=1cm, runit=1cm}
\psset{linewidth=1pt}
\psset{dotsize=7pt 0,dotstyle=*}
\begin{pspicture}(2,3)(14,8.3)
\psline{->}(5,5)(3,8)
\psline{->}(5,5)(3,2)
\psline{->}(5,5)(9,5)

%\psline(9,5)(11,8)
%\psline(9,5)(11,2)

\rput(5.5,2.8){\textcolor{red}{$_e$}}
\rput(2,6.1){\textcolor{red}{$_f$}}
\rput(7.5,6.8){\textcolor{red}{$_g$}}

\rput(12.8,5){$ \in \mathcal T_\star$}
        \end{pspicture}
}
 \frac{1}{|e f g|}= 2\log \varepsilon_D,
\end{equation}
where we sum over all vertices of $\mathcal T_\star$ modulo the river period, (each vertex contributing one term).
\end{theorem}

Whereas Theorem \ref{mik} required region labels, this result for positive $D$ requires edge labels.
It avoids the river edges since they can be zero. A configuration $[a,b,c]$ corresponds to a river edge if and only if $ac<0$. Hence edge labels $e$ are river edge labels if and only if $|e|< \sqrt{D}$. This means that no edge labels for $\mathcal T_\star$ are within $\sqrt{D}$ of $0$.

%*************************************
% geos
\SpecialCoor
\psset{griddots=5,subgriddiv=0,gridlabels=0pt}
\psset{xunit=0.38cm, yunit=0.38cm, runit=0.38cm}
%\psset{xunit=1cm, yunit=1cm, runit=1cm}
\psset{linewidth=1pt}
\psset{dotsize=3pt 0,dotstyle=*}
\begin{figure}[ht]
\centering
\begin{pspicture}(0,-4)(36,6) %\psgrid

%\psline(0,0)(10,0)(10,5)(0,5)(0,0)
%\psset{arrowscale=2,arrowinset=0.5}
\psset{arrowscale=1.4,arrowinset=0.3,arrowlength=1.1}
\newrgbcolor{light}{0.9 0.9 1.0}
\newrgbcolor{light}{0.8 0.7 1.0}
%\newrgbcolor{light}{0 0 1}
%\newrgbcolor{blue2}{0.4 0.4 1.0}
\newrgbcolor{blue2}{0.1 0.2 0.8}
\newrgbcolor{pale}{1 0.85 0.65}
\newrgbcolor{light2}{0.87 0.57 0.77}

% rgt
\rput(31,2.5){%
\begin{pspicture}(0,0)(10,5)
\psline[linecolor=white,fillstyle=solid,fillcolor=pale](0,0)(10,0)(10,5)(0,5)(0,0)

\psline(0,0)(10,0)

\rput(1,4){$\H$}

%\psline[ArrowInside=->, ArrowInsidePos=0.5](0,0)(9,9)

\rput(5,-3){$D>0$ \ square} %\phantom{non-}

\psarc[linecolor=gray](5,0){3}{0}{180}

\psarc[linecolor=gray](2,1.5){1.5}{0}{360}
\psarc[linecolor=gray](8,1){1}{0}{360}

\psarc[linewidth=1.8pt](5,0){3}{37}{126}

\psline[linecolor=gray]{->}(2,3)(2,2)
\psline[linecolor=gray]{->}(8,2)(8,1.2)

\rput(2,-0.9){$z_Q$}
\psdot(2,0)

\rput(8,-0.9){$z'_Q$}
\psdot(8,0)

\end{pspicture}}

%mid
\rput(18,2.5){%
\begin{pspicture}(0,0)(10,5)
\psline[linecolor=white,fillstyle=solid,fillcolor=pale](0,0)(10,0)(10,5)(0,5)(0,0)

\psline(0,0)(10,0)

\rput(1,4){$\H$}

%\psline[ArrowInside=->, ArrowInsidePos=0.5](0,0)(9,9)

\rput(5,-3){$D>0$ \ non-square}

\psarc[linecolor=gray](5.5,0){3.5}{0}{180}
\psarc[linewidth=1.8pt](5.5,0){3.5}{80}{115}

\rput(2,-0.9){$z_Q$}
\psdot(2,0)

\rput(9,-0.9){$z'_Q$}
\psdot(9,0)

\end{pspicture}}

% lft
\rput(5,2.5){%
\begin{pspicture}(0,0)(10,5)
\psline[linecolor=white,fillstyle=solid,fillcolor=pale](0,0)(10,0)(10,5)(0,5)(0,0)

\psline(0,0)(10,0)

\rput(6,3){$z_Q$}
\psdot(7,3)

\rput(1,4){$\H$}

%\psline[ArrowInside=->, ArrowInsidePos=0.5](0,0)(9,9)
\rput(5,-3){$D<0$}

\end{pspicture}}

\end{pspicture}
\caption{Points and paths of integration along geodesics} \la{paf}
\end{figure}
%*************************************

\begin{proof}[Proof of Theorem \ref{mt}] We are following \cite[Sect. 6]{dit21}, though  the treatment must be changed slightly since we are dealing with only one equivalence class.
Let $S_q$ be the geodesic arc in $\H$ from $\ze_q$ to $\ze'_q$. For any fixed $z$ on $S_q$ let $C_q$ be the part of $S_q$ between $z$ and $G_q z$ where $G_q$ is the generator of the automorphs of $q$ given in Definition \ref{unit}. Then, as depicted in the middle of Figure \ref{paf},
\begin{equation}\label{is}
  \int_{C_Q}  P(\tau) \, d\tau_Q = \sum_{q \sim Q} I_q  \qquad \text{with} \qquad I_q:= \int_{S_q}  \mathcal H(\tau) \, d\tau_q
\end{equation}
and $d\tau_q := \sqrt{D} d\tau/q(\tau,1)$. For $q=[a,b,c]$, ($a$ cannot be $0$ since $D$ is not a square),
\begin{align}
  I_q & = \int_0^{\pi}  \mathcal H\left( -\frac b{2a} + \frac{\sqrt{D}}{2a} e^{i \theta a/|a|}\right) \frac{d\theta}{\sin \theta} \notag\\
   & = \frac{4^{s_2+s_3}D^{(s_1+s_2+s_3)/2}}{|a|^{s_1-s_2-s_3}}\int_0^\infty \frac{u^{s_1+s_2+s_3-1} \,du}{(1+u^2)^{s_1}(A'^2+A^2 u^2)^{s_2}(B'^2+B^2 u^2)^{s_3}}, \label{is2}
\end{align}
when $A=b+\sqrt{D}$, $A'=b-\sqrt{D}$,  $B=b+2a+\sqrt{D}$ and $B'=b+2a-\sqrt{D}$.

In the case $s_1=s_2=s_3=1$,  \e{is} and \e{is2}, along with $ \int_{C_Q}   d\tau_Q = 2\log \varepsilon_D$, imply
\begin{equation} \label{ham}
  3\pi \log \varepsilon_D = 16 D^{3/2}\sum_{[a,b,c] \sim Q} \int_0^\infty \frac{|a| u^{2} \,du}{(1+u^2)(A'^2+A^2 u^2)(B'^2+B^2 u^2)}.
\end{equation}
The integral evaluates to
\begin{equation*}
  \frac{\pi |a|}{2(|A|+|A'|)(|B|+|B'|)(|AB'|+|A'B|)}
\end{equation*}
according to \cite[Lemma 2]{dit21}, producing different answers depending on the signs of $AA'$ and $BB'$ which we denote by $\pm \pm$:
\begin{equation} \label{++}
\begin{aligned}
  \int^{++} & = \frac{\pi |a|}{32 a \cdot b(b+2a)(b+2c)}, & \qquad \int^{+-} & = -\frac{\pi |a|}{32 a \cdot D b}, \\
  \int^{-+} & = \frac{\pi |a|}{32 a \cdot D (b+2a)}, & \int^{--} & = -\frac{\pi |a|}{32 a \cdot D (b+2c)}.
\end{aligned}
\end{equation}
Note that for $a>0$
\begin{alignat}{5}
  AA' & >0 & \quad &\iff \quad & b^2 & >D & \quad &\iff \quad & c &>0, \label{oh}\\
  BB' & >0 & \quad &\iff \quad & (b+2a)^2 & >D  & \quad &\iff \quad &  a+b+c &>0.  \label{ohh}
\end{alignat}
Assume $a>0$. We may apply $S$ in the $+-$ case to obtain $[a',b',c']=q'=q|S$ with $a',c'>0$, $b'>a'+c'$. In other words $q'$ is Z-reduced, recalling Definition \ref{zredd}. Also $-1/b$ becomes $1/b'$. Similarly, in the $-+$ case apply $L$ to get a Z-reduced form and $1/(b+2a)=1/b'$. In the $--$ case we cannot easily obtain an equivalent Z-reduced form. Instead apply $R$ to get a form $q'$ with $a',c'<0$, $b'<a'+c'$, meaning that $-q'$ is Z-reduced. Then $-1/(b+2c)=-1/b'$.

Assume $a<0$. The inequalities are now reversed: $AA'  >0 \iff  c <0$ and $BB'  >0 \iff  a+b+c <0$. Apply $S$ in the $+-$ case and $L$ in the $-+$ case to obtain $q'$ with $-q'$ Z-reduced. Lastly apply $R$ in the $--$ case to obtain $q'$ that is Z-reduced. Notice that $|a|/a$ gives a minus sign in \e{++} when $a<0$, making these terms positive. Altogether
\begin{equation*}
  6 \log \varepsilon_D =  \sum_{\substack{[a,b,c] \sim Q \\ \text{$a$, $c$, $a+b+c$ same sign}}} \frac{D^{3/2}}{|b(b+2a)(b+2c)|}+
  \sum_{\substack{q=[a,b,c] \sim Q \\  \text{$q$ or $-q$ is Z-reduced}}} \frac{3D^{1/2}}{|b|}.
\end{equation*}
On the topograph $\mathcal T$ containing $Q$, modulo the river period,
\begin{equation} \label{lab}
   2\log \varepsilon_D =  \sum_{\substack{\text{vert. $v \in \mathcal T$} \\  \text{$v \notin$ river}}} \frac{D^{3/2}}{|e f g|}+
  \sum_{\substack{\text{vert. $v \in \mathcal T$} \\  \text{$v \in$ river}}} \frac{D^{1/2}}{|e|},
\end{equation}
where the first series involves the edge labels $e, f, g$ directed outwards from vertex $v$ not on the river. The second sum is finite and uses the label of the unique edge, directed outwards from  $v$ on the river, that is not a river edge.
Equality \e{topp} follows.
\end{proof}

Compare  \cite[Eq. (3.5)]{dit21} with its topographic version in the next theorem.

\begin{theorem} \label{mt2}
Let $\mathcal T$ be any topograph of non-square discriminant $D>0$. Define $\mathcal T_\star$ to equal $\mathcal T$ except that all the river edges are relabeled with $\sqrt{D}$ when directed rightwards. Then
\begin{equation}\label{topp2}
\sum_{
\psset{xunit=0.12cm, yunit=0.12cm, runit=0.2cm}
%\psset{xunit=1cm, yunit=1cm, runit=1cm}
\psset{linewidth=1pt}
\psset{dotsize=7pt 0,dotstyle=*}
\begin{pspicture}(2,3)(14,8.3)
\psline{->}(5,5)(3,8)
\psline{->}(5,5)(3,2)
\psline{->}(5,5)(9,5)

%\psline(9,5)(11,8)
%\psline(9,5)(11,2)

\rput(5.5,2.8){\textcolor{red}{$_e$}}
\rput(2,6.1){\textcolor{red}{$_f$}}
\rput(7.5,6.8){\textcolor{red}{$_g$}}

\rput(12.8,5){$ \in \mathcal T_\star$}
        \end{pspicture}
}
 \left( \frac{D^{5/2}|e+f+g|}{|e f g|^2} + \frac{D^{9/2}}{3|e f g|^3} \right)= 2\log \varepsilon_D,
\end{equation}
where we sum over all vertices of $\mathcal T_\star$  modulo the river period.
\end{theorem}
\begin{proof}
With $(s_1, s_2, s_3)=(1,2,2)$, \e{ham} becomes, similarly to \cite[Sect. 8]{dit21},
\begin{equation*}
  \frac{3\pi}2 \log \varepsilon_D = 2^8 D^{5/2}\sum_{[a,b,c] \sim Q} \int_0^\infty \frac{|a|^3 u^{4} \,du}{(1+u^2)(A'^2+A^2 u^2)^2(B'^2+B^2 u^2)^2}.
\end{equation*}
The integral is
\begin{equation*}
  \frac{\pi |a|^3\Bigl((|A|+|A'|)(|B|+|B'|)+ |AB'|+|A'B|\Bigr)}{4(|A|+|A'|)^2(|B|+|B'|)^2(|AB'|+|A'B|)^3}.
\end{equation*}
As in \e{++},
\begin{equation*} %\label{oo}
\begin{aligned}
  \int^{++} & = \frac{\pi |a|^3}{2^{10} a^3}\cdot \frac{3ab+2ac+b^2}{b^2(b+2a)^2(b+2c)^3}, & \qquad \int^{+-} & = \frac{\pi |a|^3}{2^{10} a^3}\cdot \frac{a-b}{D^2 b^2}, \\
  \int^{-+} & = \frac{\pi |a|^3}{2^{10} a^3}\cdot \frac{b+3a}{D^2 (b+2a)^2}, & \int^{--} & = \frac{\pi |a|^3}{2^{10} a^3}\cdot \frac{ab+2ac-D}{D^2 (b+2c)^3}.
\end{aligned}
\end{equation*}
Apply $S$ in the $+-$ case to obtain $q'=q|S$ where $(a-b)/b^2$ equals $(c'+b')/b'^2$ and $q'$ is Z-reduced if $a>0$ or $-q'$ is Z-reduced if $a<0$. Apply $L$ in the $-+$ case to get $q'=q|L$ where $(b+3a)/(b+2a)^2$ equals $(a'+b')/b'^2$ and $q'$ is Z-reduced if $a>0$ or $-q'$ is Z-reduced if $a<0$. Apply $R$ in the $-+$ case to get $q'=q|R$ where $(ab+2ac-D)/(b+2c)^3$ equals $(a'+c'-b')/b'^2 - D/b'^3$ and $-q'$ is Z-reduced if $a>0$ or $q'$ is Z-reduced if $a<0$. When $a<0$ the $|a|^3/a^3$ factor produces $-1$. The contribution to $\int^{+-}+\int^{-+}+\int^{--}$ from $q'$ that is Z-reduced is therefore
\begin{equation*}
  \frac{c'+b'}{b'^2}+\frac{a'+b'}{b'^2}-\frac{a'+c'-b'}{b'^2}+\frac{D}{b'^3} = \frac{3}{b'}+\frac{D}{b'^3}.
\end{equation*}
The contribution  from $q'$ when $-q'$ is Z-reduced is the same times $-1$.
Altogether
\begin{equation*}
  6 \log \varepsilon_D =  \sum_{\substack{q=[a,b,c] \sim Q \\  \text{$q$ or $-q$ Z-reduced}}} \left(\frac{3D^{1/2}}{|b|}+\frac{D^{3/2}}{|b|^3}\right)+\sum_{\substack{[a,b,c] \sim Q \\ \text{$a$, $c$, $a+b+c$ same sign}}} \frac a{|a|} \cdot \frac{D^{5/2}(3ab+2ac+b^2)}{b^2(b+2a)^2(b+2c)^3}.
\end{equation*}
For the second series use the edge labels $e=-b$, $f=b+2a$ and $g=b+2c$, directed out from the vertex. Then
\begin{equation*}
  \frac{3ab+2ac+b^2}{b^2(b+2a)^2(b+2c)^3} = \frac{3a(b+2c)+D}{b^2(b+2a)^2(b+2c)^3} = \frac{3(e+f)}{2(e f g)^2}+ \frac{e f D}{(e f g)^3}.
\end{equation*}
Adding the three configurations at the vertex makes
\begin{equation*}
  \frac{3(e+f+g)}{(e f g)^2}+ \frac{(e f +f g + g e)D}{(e f g)^3} = \frac{3(e+f+g)}{(e f g)^2}- \frac{D^2}{(e f g)^3}.
\end{equation*}
Thinking of $a, a+b+c, c$ as the region labels $r, s, t$ surrounding the vertex, they are all positive or all negative. Note that $e+f+g=r+s+t$ by \e{r to e}. In the  positive region case we have $e+f+g>0$ and also $e f g<0$ since two edges are positive and one negative. In the  negative region case $e+f+g<0$ and  $e f g>0$ since one edge is  positive and two are negative.
On the topograph $\mathcal T$ containing $Q$ therefore
\begin{equation*}
   2\log \varepsilon_D =  \sum_{\substack{\text{vert. $v \in \mathcal T$} \\  \text{$v \notin$ river}}} \left(\frac{D^{5/2}|e+f+g|}{|e f g|^2}+\frac{D^{9/2}}{3|e f g|^3} \right)+
  \sum_{\substack{\text{vert. $v \in \mathcal T$} \\  \text{$v \in$ river}}} \left(\frac{D^{1/2}}{|e|}+\frac{D^{3/2}}{3|e|^3} \right),
\end{equation*}
with the same notation as \e{lab}, and
 \e{topp2} follows.
\end{proof}

Theorem \ref{mik} uses the region label expressions $r+s+t$, $rst$  while Theorems \ref{mt}, \ref{mt2} use the outward directed edge label expressions $e+f+g$, $e f g$. With  \e{r to e}, \e{e to r} these are related by
\begin{equation*}
  8rst=-efg-D(e+f+g), \qquad  \qquad e f g=-8rst-D(r+s+t).
\end{equation*}
For an example of Theorems \ref{mt} and \ref{mt2} take the topograph in Figure \ref{riv96} with $D=96$. Then $2\log \varepsilon_D =4.5848633$,  to the accuracy shown. The left sides of \e{topp} and  \e{topp2} are $4.5838550$ and $4.5848597$, respectively, when including all vertices within $15$ edges of the river.

Summing \e{topp} over all topographs of discriminant $D$ then gives the class number formula \cite[Eq. (3.3)]{dit21}. Summing \e{topp2} finds a more symmetric version of the formula \cite[Eq. (3.5)]{dit21}.
Theorems \ref{mik}, \ref{mt} and \ref{mt2} correspond to the first two cases of the family of formulas of Hurwitz and \cite[Sect. 8]{dit21}, and may likewise be extended to give faster converging series.

\subsection{Square discriminants} \la{sqrd}
Let  $u/v \in \Q$ be a fraction in lowest terms. The Ford circle in $\H$ that is tangent to $\R$ at $u/v$  may be parameterized by $u/v-1/(x+i v^2)$ for $x \in \R$. With $\lambda>0$,
\begin{equation*}
  \mathcal C_{u/v}(\lambda) \quad \text{given by} \quad \frac uv -\frac 1{x+i \lambda v^2} \quad \text{for} \quad  x \in \R,
\end{equation*}
is a scaled version with diameter $1/(\lambda v^2)$. Fixing $\lambda$, the set of circles $\mathcal C_{u/v}(\lambda)$ for $u/v \in \Q$ is permuted by the action of $\G$ provided we include $\mathcal C_{\pm 1/0}(\lambda)$ parameterized by $x+i \lambda$. Precisely,
\begin{equation}\label{ford}
  \g \mathcal C_{u/v}(\lambda) = \mathcal C_{\g (u/v)}(\lambda) \quad \text{for all} \quad \g \in \G.
\end{equation}

In this section we have forms $q=[a,b,c]$ of discriminant $D=m^2$ where $a$ and $c$ can now be zero. We assume  that $D \neq 0$. Recall $S_q$, the geodesic arc in $\H$ from $\ze_q$ to $\ze'_q$. We would like to use \e{is} but the group of automorphs of $q$ is trivial (Corollary \ref{aut}) and $ \int_{S_q} P(\tau) \, d\tau_q$ does not converge. To fix this, remove the arcs of $S_q$ inside the scaled Ford circles at $\ze_q$ and $\ze_q'$. It is convenient to have $\lambda=t/m$ for $t>0$. So let $S_q(t)$ be the part of $S_q$ between
\begin{equation} \label{wa}
  w_q:= \ze_q -\frac{m}{a+i t v^2}  \qquad \text{and} \qquad w_q':= \ze_q' +\frac{m}{a-i t (v')^2}
\end{equation}
where $v$ and $v'$ are the (positive) denominators of $\ze_q$ and $\ze_q'$, respectively. Note that $v$ and $v'$ are divisors of $a$. Here and in \e{wa} we are assuming $a\neq 0$. When $a=0$, (and hence $b\neq 0$), we have as in Definition \ref{zer},
\begin{equation} \la{a=0}
  \ze_q=\infty, \quad \ze_q'=-c/b \quad \text{if} \quad b<0 \qquad \text{and} \qquad
  \ze_q=-c/b, \quad \ze_q'=\infty \quad \text{if} \quad b>0.
\end{equation}
Then $S_q$ is the vertical line above $-c/b$ and we find $S_q(t)$ to be the part of $S_q$ going from
\begin{equation} \label{phil}
  w_q:= -c/b+i t/m  \quad \text{to} \quad w_q':= -c/b+i m/(t v^2)  \quad \text{if} \quad b<0
\end{equation}
or the same with $w_q$ and $w_q'$ swapped if $b>0$. In \e{phil}, $v$ means the denominator of $-c/b$.
The benefit of this complicated setup is that
\begin{equation*}
  \g^{-1} S_q(t) = S_{q|\g}(t) \quad \text{for all} \quad \g \in \G, \quad \text{and} \quad  S_q(t) \to S_q \quad \text{as} \quad t \to \infty.
\end{equation*}
%It can also be seen that $S_q(t) \to S_q$ as $t \to \infty$.

For our next result, define the period $1$ function
\begin{equation*}
  W_1(x)  := 2\Re  \int_{0}^{\infty} \frac{y}{y^2+1} \cdot \frac{1}{e^{\pi(y+2i x)}-1} \, dy.
\end{equation*}

\begin{theorem} \label{sq}
Let $\mathcal T$ be any topograph of square discriminant $D=m^2>1$. As before, define $\mathcal T_\star$ to equal $\mathcal T$ except that all the river edges are relabeled with $\sqrt{D}=m$ when directed rightwards. Denote by $r$ and $s$ the congruence classes mod $m$ of the lake adjacent region labels. Then
\begin{equation}\label{toww}
 W_1\left(\frac rm \right)+W_1\left(\frac sm \right) + 
     m^{3}
\sum_{
\psset{xunit=0.12cm, yunit=0.12cm, runit=0.2cm}
%\psset{xunit=1cm, yunit=1cm, runit=1cm}
\psset{linewidth=1pt}
\psset{dotsize=7pt 0,dotstyle=*}
\begin{pspicture}(2,3)(14,8.3)
\psline{->}(5,5)(3,8)
\psline{->}(5,5)(3,2)
\psline{->}(5,5)(9,5)

%\psline(9,5)(11,8)
%\psline(9,5)(11,2)

\rput(5.5,2.8){\textcolor{red}{$_e$}}
\rput(2,6.1){\textcolor{red}{$_f$}}
\rput(7.5,6.8){\textcolor{red}{$_g$}}

\rput(12.8,5){$ \in \mathcal T_\star$}
        \end{pspicture}
}
\frac{1}{|e f g|}  = 2\log\left(\frac m{2\gcd(m,r)} \right),
\end{equation}
where we sum over all vertices of $\mathcal T_\star$ that are not on a lake, (each vertex contributing one term as usual). For $D=m=1$, equation \e{toww} is valid if $2$ is added to the right.
\end{theorem}

%*************************************
% w1
\SpecialCoor
\psset{griddots=5,subgriddiv=0,gridlabels=0pt}
\psset{xunit=6cm, yunit=6cm, runit=2cm}
%\psset{xunit=1cm, yunit=1cm, runit=1cm}
\psset{linewidth=1pt}
\psset{dotsize=5pt 0,dotstyle=*}
\begin{figure}[ht]
\centering
\begin{pspicture}(-0.2,-0.12)(1.2,0.3) %\psgrid

%\psline(0,0)(10,0)(10,5)(0,5)(0,0)
%\psset{arrowscale=2,arrowinset=0.5}
\psset{arrowscale=2.4,arrowinset=0.3,arrowlength=1.1}
\newrgbcolor{light}{0.9 0.9 1.0}
\newrgbcolor{light}{0.8 0.7 1.0}
%\newrgbcolor{light}{0 0 1}
%\newrgbcolor{blue2}{0.4 0.4 1.0}
\newrgbcolor{blue2}{0.1 0.2 0.8}
\newrgbcolor{pale}{1 0.85 0.65}
\newrgbcolor{light2}{0.87 0.57 0.77}

\psline[linecolor=light2](0,0.3)(0,-0.12)
\psline[linecolor=light2](-0.66,0)(1.66,0)

\psline[linecolor=light2](0.5,0.02)(0.5,-0.02)
\psline[linecolor=light2](1,0.02)(1,-0.02)

\psline[linecolor=light2](-0.5,0.02)(-0.5,-0.02)
\psline[linecolor=light2](1.5,0.02)(1.5,-0.02)

\psline[linecolor=light2](-0.02,-0.1)(0.02,-0.1)
\psline[linecolor=light2](-0.02,0.1)(0.02,0.1)
\psline[linecolor=light2](-0.02,0.2)(0.02,0.2)

\savedata{\mydata}[
{{-0.66, -0.0844619}, {-0.64, -0.0919651}, {-0.62, -0.0984038}, 
{-0.6, -0.103806}, {-0.58, -0.108196}, {-0.56, -0.111591}, {-0.54, 
-0.114006}, {-0.52, -0.115451}, {-0.5, -0.115932}, {-0.48, 
-0.115451}, {-0.46, -0.114006}, {-0.44, -0.111591}, {-0.42, 
-0.108196}, {-0.4, -0.103806}, {-0.38, -0.0984038}, {-0.36, 
-0.0919651}, {-0.34, -0.0844619}, {-0.32, -0.0758609}, {-0.3, 
-0.0661229}, {-0.28, -0.0552028}, {-0.26, -0.0430484}, {-0.24, 
-0.0296}, {-0.22, -0.0147897}, {-0.2, 0.00146004}, {-0.18, 
  0.0192375}, {-0.16, 0.0386431}, {-0.14, 0.0597911}, {-0.12, 
  0.082812}, {-0.1, 0.107855}, {-0.08, 0.13509}, {-0.06, 
  0.164714}, {-0.04, 0.196952}, {-0.02, 0.232067}, {0., 
  0.270363}, {0.02, 0.232067}, {0.04, 0.196952}, {0.06, 
  0.164714}, {0.08, 0.13509}, {0.1, 0.107855}, {0.12, 
  0.082812}, {0.14, 0.0597911}, {0.16, 0.0386431}, {0.18, 
  0.0192375}, {0.2, 
  0.00146004}, {0.22, -0.0147897}, {0.24, -0.0296}, {0.26, 
-0.0430484}, {0.28, -0.0552028}, {0.3, -0.0661229}, {0.32, 
-0.0758609}, {0.34, -0.0844619}, {0.36, -0.0919651}, {0.38, 
-0.0984038}, {0.4, -0.103806}, {0.42, -0.108196}, {0.44, -0.111591}, 
{0.46, -0.114006}, {0.48, -0.115451}, {0.5, -0.115932}, {0.52, 
-0.115451}, {0.54, -0.114006}, {0.56, -0.111591}, {0.58, -0.108196}, 
{0.6, -0.103806}, {0.62, -0.0984038}, {0.64, -0.0919651}, {0.66, 
-0.0844619}, {0.68, -0.0758609}, {0.7, -0.0661229}, {0.72, 
-0.0552028}, {0.74, -0.0430484}, {0.76, -0.0296}, {0.78, -0.0147897}, 
{0.8, 0.00146004}, {0.82, 0.0192375}, {0.84, 0.0386431}, {0.86, 
  0.0597911}, {0.88, 0.082812}, {0.9, 0.107855}, {0.92, 
  0.13509}, {0.94, 0.164714}, {0.96, 0.196952}, {0.98, 0.232067}, {1.,
   0.270363}, {1.02, 0.232067}, {1.04, 0.196952}, {1.06, 
  0.164714}, {1.08, 0.13509}, {1.1, 0.107855}, {1.12, 
  0.082812}, {1.14, 0.0597911}, {1.16, 0.0386431}, {1.18, 
  0.0192375}, {1.2, 
  0.00146004}, {1.22, -0.0147897}, {1.24, -0.0296}, {1.26, 
-0.0430484}, {1.28, -0.0552028}, {1.3, -0.0661229}, {1.32, 
-0.0758609}, {1.34, -0.0844619}, {1.36, -0.0919651}, {1.38, 
-0.0984038}, {1.4, -0.103806}, {1.42, -0.108196}, {1.44, -0.111591}, 
{1.46, -0.114006}, {1.48, -0.115451}, {1.5, -0.115932}, {1.52, 
-0.115451}, {1.54, -0.114006}, {1.56, -0.111591}, {1.58, -0.108196}, 
{1.6, -0.103806}, {1.62, -0.0984038}, {1.64, -0.0919651}, {1.66, 
-0.0844619}}
  ]
\dataplot[linewidth=0.8pt,plotstyle=line]{\mydata}

\rput(-0.1,-0.1){$_{-0.1}$}
%\rput(-0.1,0){$_{\phantom{-0.}0}$}
\rput(-0.1,0.1){$_{\phantom{-}0.1}$}
\rput(-0.1,0.2){$_{\phantom{-}0.2}$}

\rput(0.5,-0.06){$_{0.5}$}
\rput(1,-0.06){$_{1}$}

\rput(-0.5,-0.06){$_{-0.5}$}
\rput(1.5,-0.06){$_{1.5}$}

\rput(1.7,0.04){$x$}

\rput(0.8,0.2){$W_1(x)$}

%\psline[ArrowInside=->, ArrowInsidePos=0.5](0,0)(9,9)

\end{pspicture}
\caption{The graph of $W_1(x)$}
\label{www1}
\end{figure}
%*************************************

\begin{proof}
The analog of \e{is}, as illustrated on the right of Figure \ref{paf}, is
\begin{equation} \label{is2x}
   \int_{S_Q(t)}  P(\tau) \, d\tau_Q = \sum_{q \sim Q} I_q(t)  \qquad \text{with} \qquad I_q(t) := \int_{S_q(t)}  \mathcal H(\tau) \, d\tau_q.
\end{equation}
Set $s_1=s_2=s_3=1$ so that $P(\tau) = 3\pi/2$.
On the left, assuming $Q=[0,b',c']$, we have from \e{phil}
\begin{equation} \label{3p}
  \int_{S_Q(t)}  P(\tau) \, d\tau_Q = \frac{3\pi}2 \int_{m/(t v^2)}^{t/m} \frac{dy}{y} = 3\pi\log(tv/m) = 3\pi \log(t)-3\pi \log(\gcd(b',c')),
\end{equation}
which is independent of the sign of $b'$.
As in the proof of Theorem \ref{mt} we next compute the terms $I_q(t)$. The new cases here are when $a$, $a+b+c$ or $c$ are $0$, corresponding to configurations at a lake. We treat these cases first.

For $q=[a,b,c]$ with $a=0$ we have by \e{phil},
\begin{equation} \label{kp}
  I_q(t)= \int_{m/(t v^2)}^{t/m} \frac{m^4 y^2 \, dy}{(c^2+m^2y^2)((c+b)^2+m^2y^2)} ,
\end{equation}
for $v$  the denominator of $c/m$. When $a\neq 0$,  in the notation of \e{is2} and with \e{wa},
\begin{equation} \label{go}
  I_q(t) = 16 m^{3} \int_{|a|/(t v^2)}^{t(v')^2/|a|} \frac{|a| u^{2} \,du}{(1+u^2)(A'^2+A^2 u^2)(B'^2+B^2 u^2)},
\end{equation}
for $v$ and $v'$  the denominators of $\ze_q$ and $\ze_q'$, respectively.
As in \e{oh}, when $a\neq 0$ we have $AA'=0 \iff c=0$. So when $AA'=0$, replacing $q$ by $q|S=[0,*,*]$ transforms \e{go} into \e{kp} but with $-b$ instead of $b$. (For example, when $A=0$ we have $A'=-2m$, $B=2a$, $B'=2(a-m)$, $v$ the denominator of $m/a$ and $v'=1$. Use $u=m/(|a| y)$.) When $a\neq 0$ we have $BB'=0 \iff a+b+c=0$. In the case $BB'=0$, replacing $q$ by $q|R=[0,*,*]$ transforms \e{go} into \e{kp} in the same way. The regions adjacent to each lake have labels in arithmetic progression as seen in Proposition \ref{prok}. If they are $\equiv r \bmod m$ at a lake, then combining the three configurations at each lake vertex and summing gives the contribution
\begin{equation} \label{cr}
  3\sum_{c \, \equiv \, r \bmod m} \int_{m/(t v^2)}^{t/m} \frac{m^4 y^2 \, dy}{(c^2+m^2y^2)((c+m)^2+m^2y^2)}
\end{equation}
to  \e{is2x}. Interchanging summation and integration is justified and produces a series that may be evaluated explicitly next.

\begin{lemma} \label{v}
For $m \in \Z_{\gqs 1}$ and $y>0$,
\begin{equation*}
  \sum_{c \, \equiv \, r  \bmod m}  \frac{m^4 y^2}{(c^2+m^2y^2)((c+m)^2+m^2y^2)} = \frac{2\pi y}{4y^2+1} \left(
  1+2\Re \frac{1}{e^{2\pi(y+i r/m)}-1}\right).
\end{equation*}
\end{lemma}
\begin{proof}
The summands may be broken up using
\begin{equation} \label{bro}
   \frac{m^2-2c m}{c^2+m^2y^2}+ \frac{m^2+2(c+m) m}{(c+m)^2+m^2y^2} = \frac{m^4( 4y^2+1)}{(c^2+m^2y^2)((c+m)^2+m^2y^2)},
\end{equation}
from which it follows that
\begin{equation} \label{euro}
  \sum_{c \, \equiv \, r  \bmod m} \frac{2m^2}{c^2+m^2y^2} = \sum_{c \, \equiv \, r  \bmod m}\frac{m^4( 4y^2+1)}{(c^2+m^2y^2)((c+m)^2+m^2y^2)}.
\end{equation}

Starting with $\sin(\pi x)$ and differentiating the log of its Weierstrass product produces Euler's
 formula
\begin{equation*}
  \sum_{n \in \Z} \frac{1}{n^2+t^2}= \frac{\pi}{t}\left( 1+\frac 2{e^{2\pi t}-1}\right) \qquad (t>0).
\end{equation*}
Starting instead with $\sin(\pi (x+r)/m)\sin(\pi (x-r)/m)$ gives, for $t=i x$,
\begin{equation} \label{eul}
  \sum_{c \, \equiv \, r  \bmod m} \frac{1}{c^2+t^2}= \frac{\pi}{m t}\left( 1+2\Re \frac 1{e^{2\pi (t+i r)/m}-1}\right)\qquad (t>0).
\end{equation}
The proof is finished by combining \e{euro} and \e{eul}.
\end{proof}

Then \e{cr} can be written as
\begin{equation} \label{poy}
  3 \int_{m/(t v^2)}^{t/m} \frac{2\pi y \, dy}{4y^2+1} + 3\cdot 2\Re\int_{m/(t v^2)}^{t/m} \frac{2\pi y}{4y^2+1} \cdot \frac{1}{e^{2\pi(y+i r/m)}-1} \, dy,
\end{equation}
and the first term in \e{poy} is
\begin{equation} \la{io}
  \frac{3\pi}2 \log(t) + \frac{3\pi}4\left( \log\left(\frac{4}{m^2}+\frac 1{t^2} \right) - \log\left(\frac{4m^2}{t^2 v^4}+1 \right)\right).
\end{equation}
The sum of the $\frac{3\pi}2 \log(t)$ terms from each of the two lakes equals  the $3\pi \log(t)$ term in \e{3p}. Therefore, subtracting $3\pi \log(t)$ from both sides of the first equality in \e{is2x} ensures that the positive terms on the right side converge as $t\to \infty$. In the limit we obtain
\begin{equation*}
  -3\pi \log(\gcd(b',c')) = \frac{3\pi}2 \log\left(\frac{4}{m^2} \right) + \frac{3\pi}2 W_1\left(\frac{r}{m} \right) + \frac{3\pi}2 W_1\left(\frac{s}{m} \right) + \sum_{q \sim Q} I_q,
\end{equation*}
where we are summing over $q$ with $a$, $AA'$ and $BB'$ nonzero, and $I_q=I_q(\infty)$. This sum was computed in the proof of Theorem \ref{mt}. Lastly, take $Q=[0,b',c']=[0,m,r]$ to obtain \e{toww}. When $D=1$  the two lakes meet, see Figure \ref{1 4}, and the double counting must be corrected.
\end{proof}

Summing \e{toww} over primitive topographs leads to a class number formula. For $m>1$,
\begin{equation} \la{wr}
 h(m^2) \log\left(\frac m2 \right) =  \sum_{a,\, c,\, a+b+c >0}  \frac{m^3}{3b(b+2a)(b+2c)}+ \sum_{\text{$[a,b,c]$ $Z$-reduced}} \frac mb+\sum_{\substack{1\lqs r < m\\ \gcd(m,r)=1}} W_1 \left(\frac rm \right) ,
\end{equation}
where $b^2-4ac=m^2$ and $\gcd(a,b,c)=1$ in the first two sums. We used Proposition \ref{prok} for \e{wr}, and already know  by Proposition \ref{hm2} that $h(m^2)=\phi(m)$.

Next set
\begin{equation}\label{w2}
  W_2(x):=2 \Re \int_0^\infty \frac{y(3y^4+5y^2+6)}{(y^2+1)^3} \cdot
  \frac{1}{e^{\pi(y+2i x)}-1} \, dy.
\end{equation}

\begin{theorem} \label{sq2}
Let $\mathcal T$ be any topograph of square discriminant $D=m^2>1$.  Denote by $r$ and $s$ the congruence classes mod $m$ of its lake adjacent region labels. Then
\begin{equation}\label{toww2}
\frac {W_2\left(\frac rm \right)+ W_2\left(\frac sm \right)+1}3 +
\sum_{
\psset{xunit=0.12cm, yunit=0.12cm, runit=0.2cm}
%\psset{xunit=1cm, yunit=1cm, runit=1cm}
\psset{linewidth=1pt}
\psset{dotsize=7pt 0,dotstyle=*}
\begin{pspicture}(2,3)(14,8.3)
\psline{->}(5,5)(3,8)
\psline{->}(5,5)(3,2)
\psline{->}(5,5)(9,5)

%\psline(9,5)(11,8)
%\psline(9,5)(11,2)

\rput(5.5,2.8){\textcolor{red}{$_e$}}
\rput(2,6.1){\textcolor{red}{$_f$}}
\rput(7.5,6.8){\textcolor{red}{$_g$}}

\rput(12.8,5){$ \in \mathcal T_\star$}
        \end{pspicture}
}
 \left( \frac{m^{5}|e+f+g|}{|e f g|^2} + \frac{m^{9}}{3|e f g|^3} \right)  = 2\log\left(\frac {m}{2\gcd(m,r)} \right),
\end{equation}
where we sum over all vertices of $\mathcal T_\star$ that are not on a lake. For $D=m=1$, equation \e{toww2} is valid if $8/3$ is added to the right.
\end{theorem}
\begin{proof}
Set  $(s_1, s_2, s_3)=(1,2,2)$ so that $P(\tau) = 3\pi/4$. As in \e{is2x}, \e{3p},
 assuming $Q=[0,b',c']$,
\begin{equation} \label{3p2}
  \frac{3\pi}2 \big( \log(t)- \log(\gcd(b',c'))\big) =  \sum_{q \sim Q} I_q(t).
\end{equation}
For $q=[a,b,c]$ with $a=0$ we have similarly to \e{kp},
\begin{equation} \label{kp2}
  I_q(t)= \int_{m/(t v^2)}^{t/m} \frac{m^8 y^4 \, dy}{(c^2+m^2y^2)^2((c+b)^2+m^2y^2)^2} ,
\end{equation}
for $v$  the denominator of $c/m$. When $a\neq 0$, the analog of \e{go} is
\begin{equation} \label{go2}
  I_q(t) = 2^8 m^{5} \int_{|a|/(t v^2)}^{t(v')^2/|a|} \frac{|a|^3 u^{4} \,du}{(1+u^2)(A'^2+A^2 u^2)^2(B'^2+B^2 u^2)^2},
\end{equation}
for $v$ and $v'$  the denominators of $\ze_q$ and $\ze_q'$, respectively.
At a lake vertex we have three configurations with $a$, $c$ or $a+b+c$ equaling $0$. Applying $S$ when $c=0$ or $R$ when $a+b+c=0$ gets them into the form  $[0,*,*]$ and \e{go2} transforms into
\begin{equation*}
  \int_{m/(t v^2)}^{t/m} \frac{m^6 y^4 \, dy}{(c^2+m^2y^2)((c-b)^2+m^2y^2)^2}, \qquad
   \int_{m/(t v^2)}^{t/m} \frac{m^6 y^4 \, dy}{((c-b)^2+m^2y^2)(c^2+m^2y^2)^2},
\end{equation*}
respectively. Hence the analog of \e{cr} is
\begin{equation} \label{cr2}
  \sum_{c \, \equiv \, r \bmod m} \int_{m/(t v^2)}^{t/m} \frac{2(m^8 y^6 +m^8 y^4 +c m^7 y^4 +c^2 m^6 y^4)}{(c^2+m^2y^2)^2((c+m)^2+m^2y^2)^2}\, dy,
\end{equation}
giving the total contribution to \e{3p2} from all configurations on a lake with  adjacent regions $\equiv r \bmod m$. Interchanging the sum and integral is valid and we need:

\begin{lemma} \la{mz}
For $m \in \Z_{\gqs 1}$ and $y>0$,
\begin{multline} \label{den}
  \sum_{c \, \equiv \, r \bmod m} \frac{2(m^8 y^6 +m^8 y^4 +c m^7 y^4 +c^2 m^6 y^4)}{(c^2+m^2y^2)^2((c+m)^2+m^2y^2)^2}
   = \frac{2\pi y}{(4y^2+1)^3}(24 y^4+18y^2+1) \\
   + 4\pi \Re \frac{y^2}{(4y^2+1)^2}\left(\frac{24 y^4+18y^2+1}{y(4y^2+1)} \cdot
  \frac{1}{e^{2\pi(y+i r/m)}-1} + \frac{2\pi e^{2\pi(y+i r/m)}}{(e^{2\pi(y+i r/m)}-1)^2}\right).
\end{multline}
\end{lemma}
\begin{proof}
Let $X_c$ denote each summand on the left  of \e{den}. Since
\begin{equation*}
  Z_c:=\left(\frac{m^2 y^2}{c^2+m^2y^2} - \frac{m^2 y^2}{(c+m)^2+m^2y^2} \right)^2 = \frac{m^8 y^4+4cm^7y^4+4c^2m^6 y^4}{(c^2+m^2y^2)^2((c+m)^2+m^2y^2)^2},
\end{equation*}
we may write
\begin{equation} \la{ert}
  2X_c= (4y^2+3)Y_c  + Z_c, \quad \text{for} \quad Y_c:=\frac{m^8 y^4}{(c^2+m^2y^2)^2((c+m)^2+m^2y^2)^2}.
\end{equation}
Then a calculation using the square of \e{bro} shows
\begin{equation}  \la{ert2}
   \sum_{c \, \equiv \, r \bmod m} Y_c = \frac{y^2}{(4y^2+1)^2}\left((-8y^2+2)U+(8y^2+10)V \right)
\end{equation}
for
\begin{equation*}
  U:= \sum_{c \, \equiv \, r \bmod m} \frac{m^4 y^2}{(c^2+m^2y^2)^2}, \qquad  V:= \sum_{c \, \equiv \, r \bmod m} \frac{m^4 y^2}{(c^2+m^2y^2)((c+m)^2+m^2y^2)}.
\end{equation*}
Also, easily,
\begin{equation}  \la{ert3}
   \sum_{c \, \equiv \, r \bmod m} Z_c = 2y^2 U - 2y^2 V.
\end{equation}
Altogether, \e{ert}, \e{ert2} and \e{ert3} imply
\begin{equation} \label{rr}
   \sum_{c \, \equiv \, r \bmod m} X_c = \frac{y^2}{(4y^2+1)^2}\left(4U+(24y^2+14)V \right).
\end{equation}
Differentiating \e{eul} finds
\begin{equation*}
  U=  \Re \left[\frac{\pi}{2y}\left(1+\frac{2}{e^{2\pi(y+i r/m)}-1} \right)+ \frac{2\pi^2 e^{2\pi(y+i r/m)}}{(e^{2\pi(y+i r/m)}-1)^2}\right].
\end{equation*}
Inserting this into \e{rr}, along with the formula for $V$ in Lemma \ref{v}, completes the proof.
\end{proof}

Since
\begin{equation} \la{mz2}
  \int \frac{2\pi y}{(4y^2+1)^3}(24 y^4+18y^2+1)\, dy = \frac{3\pi}{8}\log(4y^2+1)- \frac{\pi}{8}\frac{12y^2+1}{(4y^2+1)^2},
\end{equation}
we see that the $\frac{3\pi}{8}\log(4y^2+1)$ term contributes the only part of \e{cr2} that grows with $t$, giving \e{io} divided by $2$. For each of the two lakes, the sum of the $\frac{3\pi}{4} \log(t)$ terms equals  the $\frac{3\pi}{2} \log(t)$ term in \e{3p2}. Subtracting these from both sides of \e{3p2} allows us to take the limit as $t\to \infty$. We obtain from \e{cr2}, via Lemma \ref{mz} and \e{mz2},
\begin{equation*}
  \frac{3\pi}4 \log\left(\frac 2m\right)+\frac{\pi}8 +
  \frac{\pi}4 W_2\left(\frac rm\right).
\end{equation*}
This used that the integral from $0$ to $\infty$ of the last term in  \e{den}, after $y\to y/2$, is
\begin{equation*}
  \frac{\pi}2 \Re \left[\int_0^\infty \frac{y(3y^4+9y^2+2)}{(y^2+1)^3} \cdot \frac{1}{e^{\pi(y+2i u)}-1}\,dy
  + \int_0^\infty \frac{y^2}{(y^2+1)^2} \cdot \frac{2\pi e^{\pi(y+2i u)}}{(e^{\pi(y+2i u)}-1)^2}\,dy \right]
\end{equation*}
for $u=r/m$. Applying integration by parts to the second integral above and adding it to the first produces $
  \frac{\pi}4 W_2(u)$.
The contributions to \e{3p2} from all the configurations away from the lakes are handled as in the proof of Theorem \ref{mt2}.
\end{proof}

For an example of Theorems \ref{sq} and \ref{sq2} take the topograph in Figure \ref{ribbon} with $D=18^2$.  To the accuracy shown, the left sides of \e{toww} and  \e{toww2} are $4.3911059$ and $4.3944308$, respectively, when including all vertices within $15$ edges of the middle river vertex. The right sides are $4.3944492$.

Theorems \ref{sq} and \ref{sq2} are the first two cases of an expected family of results, as in \cite[Sect. 8]{dit21}. Perhaps  $W_1$ and $W_2$ are related to known functions, or can be interpreted as sums over a topograph. Roughly, $W_2(x) \approx 5 W_1(x)$.

\subsection{Discriminant zero}
To complete the discussion, we look at the remaining case.
Let $\mathcal T$ be a topograph of discriminant $D=0$ with region labels assumed to be non-negative. Let the $\gcd$ of these labels be $g$. Then $\mathcal T$ is $g$ times the  primitive topograph containing $Q=[0,0,1]$ seen in Figure \ref{det0}. We exclude the $0$ topograph with $g=0$.    As in \e{qm},
\begin{equation} \la{qm2}
  Q(x,y)=y^2, \qquad Q(x,y)|M  = \g^2 x^2+2\g \delta xy+\delta^2 y^2 \qquad \text{for} \qquad M = \begin{pmatrix} \alpha &\beta \\ \g &\delta \end{pmatrix}.
\end{equation}
Also $Q$ is periodic along its lake border with $Q|T=Q$.
With this invariance under $T$ it is natural to replace the Poincar\'e series $P(\tau)$ by the Eisenstein series
\begin{equation*}
   E(\tau,s)=\sum_{M \in \G_\infty\backslash \G} \Im( M \tau)^s = \frac 12 \sum_{\substack{\g, \delta \in \Z \\  \gcd(\g, \delta)=1}} \frac{y^s}{|\g \tau+\delta|^{2s}},
\end{equation*}
for $\tau=x+i y$, $y>0$ and $\Re(s)>1$, with $\G_\infty$ generated by $T$.

\begin{lemma}
For $\Re(s)>1$,
\begin{equation*}
  y^s \sum_{q=[a,b,c] \sim [0,0,1]} \frac 1{(a(x^2+y^2)+b x+c)^s} = E(x+i y,s).
\end{equation*}
\end{lemma}
\begin{proof}
We have
\begin{equation*}
  E(x+i y,s) = \sum_{M \in \G_\infty\backslash \G} \frac{y^s}{|\g (x+i y) +\delta|^{2s}} = \sum_{M \in \G_\infty\backslash \G} \frac{y^s}{(\g^2 x^2+2\g\delta x+ \delta^2+ \g^2 y^2)^{s}},
\end{equation*}
and the result then follows by \e{qm2}.
\end{proof}

\begin{cor}
Let $\mathcal T$ be a topograph of discriminant $0$ whose non-negative region labels have $\gcd = g \gqs 1$. Then for $\Re(s)>1$,
\begin{equation*}
  g^s
 \sum_{
\psset{xunit=0.12cm, yunit=0.12cm, runit=0.2cm}
%\psset{xunit=1cm, yunit=1cm, runit=1cm}
\psset{linewidth=1pt}
\psset{dotsize=7pt 0,dotstyle=*}
\begin{pspicture}(3,3)(14.9,7)
\psline(5,3)(5,7)
%\psline{->}(5,5)(3,2)
%\psline{->}(5,5)(9,5)

%\psline(9,5)(11,8)
%\psline(9,5)(11,2)

\rput(3,5){$a$}
\rput(7,5){$c$}

\rput(12,5.3){$ \in \mathcal T$}
        \end{pspicture}
}
 \frac{1}{(a+c)^s} = \frac 12 E(i,s),
\end{equation*}
where the sum is over all edges of $\mathcal T$, (counted once), modulo the lake border period.
\end{cor}

{\small \bibliography{qubib} }

{\small %\footnotesize
\vskip 5mm
\noindent
\textsc{Department of Mathematics, The CUNY Graduate Center, 365 Fifth Avenue, New York, NY 10016-4309, U.S.A.}

\noindent
{\em E-mail address:} \texttt{cosullivan@gc.cuny.edu}
}

\end{document}